\documentclass[a4paper,10pt,reqno]{article}

\usepackage[english]{babel}
\usepackage{amsmath}
\usepackage{amsfonts}
\usepackage{amssymb}
\usepackage{amsthm}
\usepackage{float}
\usepackage{graphics}
\usepackage{graphicx}
\usepackage[colorlinks=false]{hyperref}
\usepackage{vmargin}

\def\T{\mathbb{T}}

\renewcommand{\to}{\rightarrow}

\numberwithin{equation}{section}

\theoremstyle{plain}
\newtheorem{teor}{Theorem}[section]
\newtheorem{ese}[teor]{Example}
\newtheorem{prop}[teor]{Proposition}
\newtheorem{lem}[teor]{Lemma}
\newtheorem{cor}[teor]{Corollary}

\newcommand{\bdm}{\begin{displaymath}}
\newcommand{\edm}{\end{displaymath}}
\newcommand{\bpb}{\begin{prob}}
\newcommand{\epb}{\end{prob}}

\newcommand{\beq}{\begin{equation}}
\newcommand{\eeq}{\end{equation}}
\newcommand{\bem}{\begin{multline}}
\newcommand{\eem}{\end{multline}}
\newcommand{\bes}{\begin{ese}}
\newcommand{\ees}{\end{ese}}
\newcommand{\bde}{\begin{defi}}
\newcommand{\ede}{\end{defi}}
\newcommand{\bpr}{\begin{prop}}
\newcommand{\epr}{\end{prop}}
\newcommand{\ble}{\begin{lem}}
\newcommand{\ele}{\end{lem}}
\newcommand{\bte}{\begin{teor}}
\newcommand{\ete}{\end{teor}}
\newcommand{\bco}{\begin{cor}}
\newcommand{\eco}{\end{cor}}

\theoremstyle{definition}
\newtheorem{defi}[teor]{Definition}
\newtheorem{remark}[teor]{Remark}
\usepackage{geometry}
\begin{document}

\title{\textbf{Quasi-periodic solutions for quasi-linear generalized KdV equations}}

\date{}

\author{Filippo Giuliani\thanks{SISSA, Via Bonomea 265, 34136, Trieste, Italy, fgiulian@sissa.it}}

\maketitle

\begin{abstract}
\noindent We prove the existence of Cantor families of small amplitude, linearly stable, quasi-periodic solutions of \textit{quasi-linear} autonomous Hamiltonian generalized KdV equations. We consider the most general quasi-linear \textit{quadratic} nonlinearity. The proof is based on an iterative Nash-Moser algorithm. To initialize this scheme, we need to perform a bifurcation analysis taking into account the strongly perturbative effects of the nonlinearity near the origin. In particular, we implement a weak version of the Birkhoff normal form method. The inversion of the linearized operators at each step of the iteration  is achieved by pseudo-differential techniques, linear Birkhoff normal form algorithms and a linear KAM reducibility scheme.
\end{abstract}


\tableofcontents

\section{Introduction}
We prove the existence and the stability of Cantor families of quasi-periodic, small amplitude, solutions of the Hamiltonian quasi-linear generalized KdV equations 
\begin{equation}\label{KdV}
u_t+u_{xxx}+\mathcal{N}_2(x, u, u_x, u_{xx}, u_{xxx})=0,
\end{equation}
under periodic boundary conditions $x\in \mathbb{T}$, where
\begin{equation}\label{N2}
\mathcal{N}_2(x, u, u_x, u_{xx}, u_{xxx}):=-\partial_x[(\partial_u f)(x, u, u_x)-\partial_x((\partial_{u_x} f)(x, u, u_x))]
\end{equation}
and $f$ is the most general \textit{quasi-linear} Hamiltonian density
\begin{equation}\label{perturbazione}
\begin{aligned}
f(x, u, u_x):=\, & c_1\,u_x^3+c_2\,u_x^2\,u+c_3\,u^3+c_4\, u_x^4+c_5\,u_x^3\,u+c_6\,u_x^2\,u^2+c_7 \,u^4+f_{\geq 5}(x, u, u_x),
\end{aligned}
\end{equation}
where the coefficients $c_i, i=1,2, \dots, 7$ are real numbers, and 
\begin{equation}\label{parteomogeneagrado5}
f_{\geq 5}(x, u, u_x):=f_5(u, u_x)+f_{\geq 6}(x, u, u_x)
\end{equation}
is the sum of the homogeneous component of $f$ of degree five and all the higher order terms.\\
We assume that the Hamiltonian density $f$ in \eqref{KdVHamiltonian} belongs to $C^q(\mathbb{T}\times \mathbb{R}\times \mathbb{R}; \mathbb{R})$ for some large $q$.\\

The equation \eqref{KdV} can be formulated as a Hamiltonian PDE $u_t=\partial_x\,\nabla_{L^2} H$, where $\nabla_{L^2} H$ is the $L^2(\mathbb{T})$ gradient of the Hamiltonian
\begin{equation}\label{KdVHamiltonian}
H(u)=\int \frac{u_x^2}{2}+f(x, u, u_x)\, dx
\end{equation}
on the real phase space
\begin{equation}
H_0^1(\mathbb{T}_x):=\left\{ u\in H^1(\mathbb{T}, \mathbb{R}) : \int_{\mathbb{T}} u(x)\,dx=0\right\}
\end{equation}
endowed with the non-degenerate symplectic form
\begin{equation}\label{SymplecticForm}
\Omega(u, v):=\int_{\mathbb{T}} (\partial_x^{-1} u)\,v\,dx, \quad \forall u, v\in H_0^1(\mathbb{T}_x),
\end{equation}
where $\partial_x^{-1}u$ is the periodic primitive of $u$ with zero average defined by
\[
\partial_x^{-1} e^{\mathrm{i} j x}=\frac{1}{\mathrm{i} j}\,e^{\mathrm{i}\,j\,x}\quad\mbox{if}\,\,j\neq 0,\qquad\qquad \partial_x^{-1} 1=0.
\]
The phase space $H_0^1(\T_x)$ is invariant under the flow of the equation \eqref{KdV}.\\
The Poisson bracket induced by $\Omega$ between two functions $F, G\colon H_0^1(\mathbb{T})\to \mathbb{R}$ is
\begin{equation}\label{PoissonBracket}
\{ F(u), G(u) \}:=\Omega(X_F, X_G)=\int_{\mathbb{T}} \nabla F(u)\,\partial_x \nabla G(u)\,dx,
\end{equation}
where $X_F$ and $X_G$ are the vector fields associated to the Hamiltonians $F$ and $G$, respectively.\\

By \eqref{perturbazione} the nonlinearity $\mathcal{N}_2$ vanishes at order two at $u=0$ and \eqref{KdV} may be seen, in a small neighbourhood of the origin, as a \textit{small} perturbation of the Airy equation
\begin{equation}\label{Airy}
u_t+u_{xxx}=0.
\end{equation}
The equation \eqref{KdV} is \textit{completely resonant}, namely its linearized problem at the origin \eqref{Airy} possesses only the $2\pi$-periodic in time solutions
\begin{equation}
u(t, x)=\sum_{j\in\mathbb{Z}\setminus\{0\}} u_j\,e^{\mathrm{i}\,j^3\,t}\,e^{\mathrm{i}\,j\,x}.
\end{equation}
Then the existence of quasi-periodic solutions of \eqref{KdV} is due only to the presence of the nonlinearity. For this reason, we need to perform a bifurcation analysis which is mainly affected by the quasi-linear monomials of degree three and four in \eqref{perturbazione}. Another difficulty is that, since the equation \eqref{KdV} is completely resonant, the diophantine frequency vector of the expected quasi-periodic solutions, if any, are $O(\lvert u_j \rvert^2)$-close to integer vectors. \\

We briefly present some literature related to this paper.\\
The KAM theory for PDE's has been developed in the eighties by Kuksin, with the pioneering work \cite{ImaginarySpectrum}, and by Wayne \cite{Wayne}, Craig-Wayne \cite{Newton}, P\"oeschel \cite{SomeNLpde} for the one dimensional nonlinear wave and Schr\"odinger equations, and, at a later time, in higher dimensional cases, by Bourgain \cite{Bourgain}, Eliasson-Kuksin \cite{EliassonKuksin}, Berti-Bolle \cite{NLS}, Geng-Xu-You \cite{Geng}, Procesi-Procesi \cite{ProcesiProcesiNormal}-\cite{PP}, Wang \cite{Wang}, Eliasson-Grebert-Kuksin \cite{EGK}.\\
The first results with unbounded perturbations have been proved by Kuksin in \cite{Korteg} and Kappeler-P\"oeschel \cite{KdVeKAM} for KdV (see \cite{KukHuang} for a survey on known results for the KdV equation), by Liu-Yuan \cite{Liu}, Zhang-Gao-Yuan \cite{Zhang} for derivative NLS, and by Berti-Biasco-Procesi \cite{DNLW}-\cite{BBP} for derivative NLW.\\
All the aforementioned papers treat semilinear problems, namely the case in which the nonlinearity depends on derivatives of order $m$, with $m\le n-1$, where $n$ is the highest order of the derivatives appearing in the unperturbed system.\\
For quasi-linear and fully nonlinear PDE's, i.e. in the case $m=n$, the progress are more recent.\\
The first results in this direction are due to Iooss-Plotnikov \cite{IP}-\cite{IP2}, Iooss-Plotnikov-Toland \cite{Iooss}, Plotnikov-Toland \cite{PT} for periodic solutions of water-waves equations. In the spirit of the method implemented in these papers, Baldi in \cite{Ono} provides the existence of periodic solutions for the Benjamin-Ono equation.\\
Baldi, Berti, Montalto prove the first existence results of quasi-periodic solutions for quasi-linear and fully nonlinear PDE's, in the forced case for the Airy equation \cite{Airy}, and in the autonoumous case for the KdV and mKdV equation in \cite{KdVAut} and \cite{MKdV}. In particular, they consider in \cite{KdVAut} the Hamiltonian
\begin{equation}\label{Hkdv}
H_{KdV}+\int_{\T} f_{\geq 5}(x, u, u_x)\,dx, \quad \mbox{where} \quad H_{KdV}:=\int_{\T} \frac{u_x^2}{2}+u^3\,dx,
\end{equation}
namely, the Hamiltonian \eqref{KdVHamiltonian} without the monomials of degree three and four in the variables $(u, u_x)$, see \eqref{perturbazione}.
These works are based on Nash-Moser methods and a reducibility scheme that diagonalize completely the linearized system at any approximate solution. This procedure permits to prove also the linear stability of the solutions.\\
More recently, in \cite{FeolaProc} and \cite{FePro} Feola-Procesi provide the existence and the stability of quasi-periodic solutions for quasi-linear and fully nonlinear perturbations of the Schr\"odinger equation in dimension one. We mention also the recent work by Montalto \cite{Riccardo} on quasi-periodic solutions for the forced Kirchoff equation.\\

The aim of this paper is to generalize the results obtained in \cite{KdVAut} and \cite{MKdV} considering the most general Hamiltonian density \eqref{perturbazione}.
We are interested in understanding the effect, over infinite times, of a \textit{quadratic} and \textit{quasi-linear} Hamiltonian perturbation in a small neighbourhood of the origin, where the polynomial perturbations of lowest degree are much \textit{stronger}.
This is significant in view of the study of small amplitude solutions for many fluid dynamics equations, like Degasperis-Procesi and water waves-type equations, which involve this kind of nonlinearities.



\subsection{Main result}

The solutions that we find are localized in Fourier space close to finitely many \textit{tangential sites}
\begin{equation}\label{TangentialSites}
S^+:=\{\overline{\jmath}_1, \dots, \overline{\jmath}_{\nu} \}, \quad S:=S^+\cup(-S^+)=\{ \pm j : j\in S^+\}, \quad \overline{\jmath}_i\in \mathbb{N}\setminus\{ 0\}, \quad \forall i=1,\dots, \nu
\end{equation}
and the linear frequencies of oscillation on the tangential sites are
\begin{equation}\label{LinearFreq}
\overline{\omega}:=(\overline{\jmath}_1^3,\dots, \overline{\jmath}_{\nu}^3)\in\mathbb{N}^{\nu}.
\end{equation}

The set $S$ is required to be even because we look for real valued solutions of \eqref{KdV}. Moreover, we also assume the following hypotesis on $S$:
\begin{itemize}
\item[$(\mathtt{S})$] $\nexists\,\, j_1, j_2, j_3, j_4\in S$ such that
$$j_1+j_2+j_3+j_4\neq 0, \,\,j_1^3+j_2^3+j_3^3+j_4^3-(j_1+j_2+j_3+j_4)^3=0.$$
\end{itemize}

We decompose the phase space as 
\[
H_0^1(\mathbb{T}):=H_S\oplus H_S^{\perp}, \quad H_S:=\mbox{span}\{ e^{\mathrm{i}\,j\,x} : j\in S \}, \quad H_S^{\perp}:=\{u=\sum_{j\in S^c} u_j\,e^{\mathrm{i}\,j\,x}\in H_0^1(\mathbb{T}) \},
\]
and we denote by $\Pi_S, \Pi_S^{\perp}$ the corresponding orthogonal projectors. The subspaces $H_S$ and $H_S^{\perp}$ are symplectic respect to the $2$-form $\Omega$ (see \eqref{SymplecticForm}). We write
\begin{equation}
u=v+z, \quad v:=\Pi_S u:=\sum_{j\in S} u_j\,e^{\mathrm{i}\,j\,x}, \quad z=\Pi_S^{\perp} u:=\sum_{j\in S^c} u_j\,e^{\mathrm{i}\,j\,x},
\end{equation}
where $v$ is called the \textit{tangent} variable and $z$ the \textit{normal} one. In the following, we will identify $v=(v_j)_{j\in S}$ and $z=(z_j)_{j\in S^c}$.\\


We shall also assume ``non-resonant" and ``non-degeneracy" conditions for the nonlinearity \eqref{perturbazione}.

\begin{defi}\label{EqRisonanti}
We say that the coefficients $c_1, \dots, c_7$ are \textit{resonant} if the following holds
\begin{equation}\label{coeffris}
c_3=c_7=2 c_1^2-c_4=7 c_2^2-6 c_6= 0
\end{equation}
and we say that $c_1, \dots, c_7$ are \textit{non-resonant} if \eqref{coeffris} does not hold.
\end{defi}
Moreover, we require the following ``non-degeneracy" conditions on the coefficients $c_1,\dots, c_7$
\begin{itemize}
\item[$(\mathtt{C}1)$] fixed $\nu\in\mathbb{N}$, the coefficients $c_1, \dots, c_7$ satisfy
\begin{equation}\label{RisonanzaPlus}
\left(7-16\nu \right) c_2^2\neq 6\,(1- 2\nu) c_6,
\end{equation}
\item[$(\mathtt{C}2)$] fixed $\nu\in\mathbb{N}$, the coefficients $c_1, \dots, c_7$ satisfy
\begin{equation}\label{condizioneAB}
\nu\,\,\frac{ 3 c_6-4 c_2^2}{9 c_4-18 c_1^2}\notin \{ j^2+k^2+j k \,:\, j, k\in \mathbb{Z}\setminus\{0\}, \,j\neq k\}.
\end{equation}
\end{itemize}

Before stating the main result, we introduce a notion of ``genericity" according to the one given by Biasco-Berti-Procesi \cite{BBP}, Procesi-Procesi \cite{ProcesiProcesiNormal} and Feola \cite{FePro}.
\begin{defi}\label{Generic}
Fixed $\nu\in\mathbb{N}$ and given a non-trivial, i.e non identically zero, polynomial $P(z)$, with $z\in\mathbb{C}^{\nu}$, we say that a vector of integers $z_0\in\mathbb{N}^{\nu}$ is \textit{generic} if $P(z_0)\neq 0$.\\
We shall say that ``\textit{there is a generic choice of the tangential sites $S$ for which some condition holds}" if this condition is satisfied by every vectors of integers $(\overline{\jmath}_1, \dots \overline{\jmath}_{\nu})$ that are not zeros of some non trivial polynomial. 
\end{defi}

\begin{teor}\label{F.Giuliani}
Given $\nu\in\mathbb{N}$, let $f\in C^q$ (with $q:=q(\nu)$ large enough) satisfy \eqref{perturbazione}. If $c_1, \dots, c_7$ in \eqref{perturbazione} are non-resonant (see Definition \ref{EqRisonanti}) and conditions $(\mathtt{C}1)$-$(\mathtt{C}2)$ hold, then for a generic choice of tangential sites (see Definition \ref{Generic} and \eqref{TangentialSites}), in particular satisfying $(\mathtt{S})$, the equation \eqref{KdV} possesses small amplitude quasi-periodic solutions, with diophantine frequency vector $\omega:=\omega(\xi)=(\omega_j)_{j\in S^+}\in\mathbb{R}^{\nu}$, of the form
\begin{equation}\label{SoluzioneEsplicita}
u(t, x)=\sum_{j\in S^+}2\,\sqrt{j\,\xi_j}\,\cos(\omega_j t+j x)+o(\sqrt{\lvert \xi \rvert}), \qquad \omega_j=j^3+O(\xi_j)
\end{equation}
for a Cantor-like set of small amplitudes $\xi\in\mathbb{R}^{\nu}_+$ with density $1$ at $\xi=0$. The term $o(\sqrt{\lvert \xi \rvert})$ is small in some $H^s$-Sobolev norm, $s<q$. These quasi-periodic solutions are linearly stable.
\end{teor}

Let us make some comments on the assumptions of Theorem \ref{F.Giuliani}.
\begin{itemize}
\item The non-resonance condition stated in Definition \ref{EqRisonanti} arises by asking that the frequency-amplitude map \eqref{Frequency-AmplitudeMAP} is a diffeomorphism. The invertibility of this map is equivalent to require that $\det \mathbb{M}\neq 0$, where the determinant of $\mathbb{M}$ is a polynomial in the variables $(c_1, \dots, c_7, \overline{\jmath}_1, \dots, \overline{\jmath}_{\nu})$. In Theorem \ref{F.Giuliani} we fix non-resonant coefficients $c_1, \dots, c_7$ and we prove in Lemma \ref{TwistLemma} that the condition $\det\mathbb{M}\neq 0$ is satisfied for a generic choice of the tangential sites $S$. We remark that this explicit condition could be verified by fixing the integers $\overline{\jmath}_1, \dots, \overline{\jmath}_{\nu}$ and choosing the real parameters $c_1, \dots, c_7$ outside the zeros of some polynomial.
\item For the measure estimates of Section $9.1$, we shall avoid some lower order resonances by imposing the assumptions $(\mathtt{H}1)$ and $(\mathtt{H}2)_{j, k}$ (see \eqref{H1}, \eqref{H2}). These ones imply that some polynomials are non zero at $(c_1, \dots, c_7, \overline{\jmath}_1, \dots, \overline{\jmath}_{\nu})$. If $(\mathtt{C}1)$-$(\mathtt{C}2)$ hold and $c_1, \dots, c_7$ are non-resonant then these polynomials are not trivial in the variables $(\overline{\jmath}_1, \dots, \overline{\jmath}_{\nu})$ (see Lemma \ref{Lemmata1} and Lemma \ref{Lemmata2}) and, for a finite number of $j, k\in S^c$, $(\mathtt{H}1)$ and $(\mathtt{H}2)_{j, k}$ are verified by fixing non-resonant parameters $c_1, \dots, c_7$ and by choosing a generic set of integers $\{\overline{\jmath}_1, \dots, \overline{\jmath}_{\nu}\}$.
\item As in \cite{KdVAut}, we assume the Hypotesis $(\mathtt{S})$, because we want to perform three steps of Birkhoff normal form. Indeed the smallness condition \eqref{SmallnessConditionNM} required in Theorem \ref{NashMoser} depends on the quadraticity of the nonlinearity in \eqref{KdV}.
We remark that he assumption $(\mathtt{S})$ can be reformulate as a condition that is satisfied for a generic choice of the tangential sites.
\end{itemize}

The proof of Theorem \ref{F.Giuliani} follows the scheme adopted in \cite{KdVAut} and \cite{MKdV}.
We now shortly present the strategy of the proof of Theorem \ref{F.Giuliani} underlying the main differences with these works.\\

\noindent\textbf{Bifurcation analysis}. We cannot consider \eqref{KdV} as a perturbation problem for the linearized equation at the origin $u_t+u_{xxx}=0$, because, as we said above, this equation is completely resonant, hence the frequency vector of its solutions does not satisfy any diophantine condition. Thus, the main modulation of the frequency vector of the solutions with respect to its amplitude is due to the nonlinearity $\mathcal{N}_2$, defined in \eqref{N2}. In order to control the shift of the linear frequencies under the effect of the nonlinearity near the origin and to find approximate quasi-periodic solutions for \eqref{KdV}, in Section $3$ we perform a weak version of the Birkhoff normal form algorithm. After two steps of this procedure, we are able to find a finite dimensional submanifold of the phase space foliated by approximately invariant tori, from which the expected quasi-periodic solutions of \eqref{KdV} bifurcate. On this subspace we introduce action-angle variables (see Section $4$) and we use the ``unperturbed" actions $\xi$ of these tori as parameters for our problem. We require also that the frequency-amplitude map $\alpha(\xi)$ in \eqref{Frequency-AmplitudeMAP}, namely the function associating the actions to the frequencies, is a diffeomorphism (see Lemma \ref{TwistLemma}), so that we could consider both as independent parameters. \\
The presence of the quasi-linear monomials of degree three and four in the Hamiltonian \eqref{KdVHamiltonian} makes significantly harder the computations of the new Hamiltonian after two steps of Birkhoff normal form with respect to the case examined in \cite{KdVAut} for the Hamiltonian $H_{KdV}$ (recall \eqref{Hkdv}). Because of the integrability of the KdV system, in \cite{KdVAut} the \textit{twist condition}, namely, the invertibility of the frequency-amplitude map, is obtained for every choice of the tangential set $S$ (see \eqref{TangentialSites}). On the contrary, for the general case \eqref{KdVHamiltonian} the twist condition depends on the choice of the parameters $c_1, \dots, c_7$ and the tangential sites $\overline{\jmath}_1,\dots, \overline{\jmath}_{\nu}$.\\
In Lemma \ref{TwistLemma} we provide the invertibility of the frequency-amplitude map for a large choice of the tangential sites and of the coefficients.\\

\noindent\textbf{Nonlinear functional setting}. After the rescaling \eqref{Rescaling}, we look for quasi-periodic solutions with frequency vector $\omega$ for the $(\omega, \varepsilon)$-parameter family of Hamiltonians \eqref{Hepsilon}. We assume that $\omega$ belongs to the image of the restriction of the frequency-amplitude map $\alpha(\xi)$ on a small compact subset of $\mathbb{R}^{\nu}$ that does not contain the origin (see \eqref{OmegaEpsilon}).\\
In Section $5$ we formulate this problem as the search of the zeros of the nonlinear functional $\mathcal{F}(\omega, i(\omega))$ defined in \eqref{NonlinearFunctional}, where $\omega$ is considered as an external parameter and $\varphi\mapsto i(\varphi)$ is a torus embedded in the phase space. We find a solution $i_{\infty}(\omega t)$ for $\mathcal{F}=0$, which will correspond to a quasi-periodic solution with frequency vector $\omega$ of the original equation \eqref{KdV}, by constructing, through a Nash-Moser iteration, a sequence $(i_n)_{n\geq 0}$ of approximate solutions that converges to it, see Theorem \ref{IlTeorema}.\\

\noindent\textbf{The inversion of the linearized operator at an approximate solution}. The application of a Nash-Moser scheme involves, at any step, the inversion of the linearized operator at an approximate solution and this is, in fact, the main issue of the proof. Thanks to the abstract decoupling procedure developed by Berti-Bolle in \cite{BertiBolle}, that exploits the Hamiltonian structure, the tangential and the normal linear dynamics around an approximately invariant torus can be studied separately, see Section $6$. In particular, a suitable change of coordinates around this approximate quasi-periodic solution triangularizes the linearized problem and its inversion reduces to the study of a quasi-periodically forced PDE restricted to normal directions. The operator which has to be inverted, say $\mathcal{L}_{\omega}$, is pseudodifferential with variable coefficients and it is computed in Section $7$.\\
In Section $8$ we conjugate $\mathcal{L}_{\omega}$ to a diagonal operator, which describes infinitely many harmonic oscillators
\begin{equation}
\dot{v}_j+\mu_j^{\infty}\, v_j=0, \qquad j\in S^c, \quad \mu_j^{\infty}\in\mathrm{i} \mathbb{R}.
\end{equation}
The diagonalization of $\mathcal{L}_{\omega}$ is obtained with the same transformations defined in \cite{KdVAut} and \cite{MKdV}.
The main perturbative effect to the spectrum of $\mathcal{L}_{\omega}$ is due to the term $a_1(\omega t)\partial_{xxx}$ (see \eqref{Lomega}) and the presence of $u_x$ in the cubic part of the Hamiltonian density \eqref{perturbazione} affects this coefficient. In particular, $a_1-1=O(\varepsilon)$, instead of $O(\varepsilon^3)$ as in \cite{KdVAut}. In general, the corrections of the coefficients of $\mathcal{L}_{\omega}$ are bigger in size and this fact implies some difficulties in providing the smallness condition \eqref{PiccolezzaperKamred} required in Theorem \ref{Reducibility}. Moreover, the transformation used to conjugate $\mathcal{L}_{\omega}$ to a pseudodifferential operator with a coefficient in front of $\partial_{xxx}$ independent of the $x$-variable (see Section $8.1$) has form $\mathrm{I}+O(\varepsilon)$ and so it generates new terms of order $\varepsilon^2$. These terms are not perturbative for the reducibility scheme of Theorem \ref{Reducibility} and we need to compute them in view of a linear Birkhoff normalization.\\
We also point out that we drop the assumption
\begin{equation}\label{S1}
 j_1+j_2+j_3\neq 0 \quad \mbox{for all}\,\,\,j_1, j_2, j_3\in S
\end{equation}
required in \cite{KdVAut} to get ``good" estimates on the transformations used to conjugate $\mathcal{L}_{\omega}$ to a diagonal operator. We better discuss this fact in Remark \ref{NoS1}.\\

\noindent\textbf{The Nash-Moser iteration, measure estimates and stability}. In Section $9$ we perform the nonlinear Nash-Moser iteration which proves Theorem \ref{IlTeorema} and, therefore, Theorem \ref{F.Giuliani}. \\
In the measure estimates for the sets of parameters $\mathcal{R}_{l j k}$, for which the second Melnikov conditions are violated (see \eqref{BadSets}), some technical difficulties arise. Indeed, the corrections to the normal frequencies are big in size and the indices $l, j, k$ are not tied by the conservation of the momentum, as, for instance, in \cite{FePro}, since the nonlinearity $f$ in \eqref{perturbazione} depends on $x$. From these facts, some cases result to be \textit{degenerate} and we shall impose some assumptions on the set $S$ to avoid them (see Remark \ref{Degeneratecase} and \eqref{H1}, \eqref{H2}).\\
In Section $9.2$ we prove the stability of the solution produced by the Nash-Moser algorithm exploiting the action-angle variables introduced in Section $4$ and the diagonalization procedure performed in Section $8$.

\subsubsection*{Acknowledgements}
I am greatful to Massimiliano Berti for introducing me to the study of KAM theory and for the support offered for the elaboration of this paper. I also thank Michela Procesi and Roberto Feola for useful and stimulating discussions.

\section{Preliminaries}

\subsection{Functional setting} 

\textbf{Lipschitz norm}. For a function $u\colon \Omega_0\to E, \omega \to u(\omega)$, where $(E, \lVert \cdot \rVert_E)$ is a Banach space and $\Omega_0$ is a subset of $\mathbb{R}^{\nu}$, we define the sup-norm and the lipschitz semi-norm
\begin{equation}\label{suplip}
\begin{aligned}
&\lVert u \rVert_E^{\sup}:=\lVert u \rVert_{E, \Omega_0}^{\sup}:=\sup_{\omega\in\Omega_0} \lVert u(\omega) \rVert_E,\\
&\lVert u \rVert_{E}^{lip}:=\lVert u \rVert_{E, \Omega_0}^{lip}:=\sup_{\omega_1\neq \omega_2} \frac{\lVert u(\omega_1)-u(\omega_2)\rVert_E}{\lvert \omega_1-\omega_2\rvert},
\end{aligned}
\end{equation}
and for $\gamma>0$, the Lipschitz norm
\begin{equation}\label{Lippone}
\lVert u \rVert_E^{Lip(\gamma)}:=\lVert u \rVert_{E, \Omega_0}^{Lip(\gamma)}:=\lVert u \rVert_E^{\sup}+\gamma\lVert u \rVert_{E}^{lip}.
\end{equation}
If $E=H^s$ we simply denote $\lVert u \rVert_{H^s}^{Lip(\gamma)}:=\lVert u \rVert_s^{Lip(\gamma)}$.\\

\noindent \textbf{Sobolev norms}. We denote by
\begin{equation}\label{SobolevNormtx}
\lVert u \rVert_s:=\lVert u\rVert_{H^{s}(\mathbb{T}^{\nu+1})}:=\lVert u \rVert_{H^s_{\varphi, x}}
\end{equation}
the Sobolev norms of functions $u=u(\varphi, x)\in H^{s}(\mathbb{T}^{\nu}\times\mathbb{T})$. We denote by $\lVert \cdot \rVert_{H^s_{x}}$, the Sobolev norm of functions $u(x)$ in the phase space of class $H^s$.
We consider $s_0:=(\nu+2)/2$, hence we have that $H^{s_0}(\mathbb{T}^{\nu+1})$ is continuosly embedded in $L^{\infty}(\mathbb{T}^{\nu+1})$ and any space $H^s(\mathbb{T}^{\nu+1})$ with $s\geq s_0$ is an algebra and satisfies the interpolation inequalities: for $s\geq s_0$
\begin{equation}
\lVert u\,v \rVert_s\le C(s_0)\,\lVert u \rVert_s\lVert v \rVert_{s_0}+C(s) \lVert u \rVert_{s_0}\lVert v \rVert_s, \quad \forall u, v\in H^s(\mathbb{T}^{\nu+1}).
\end{equation}
The above inequalities also hold for the norm $\lVert \cdot \rVert^{Lip(\gamma)}$.\\
We also denote
\begin{equation}
\begin{aligned}
& H^s_{S^{\perp}}(\mathbb{T}^{\nu+1}):=\left\{  u\in H^s(\mathbb{T}^{\nu+1}) : u(\varphi, \cdot)\in H_S^{\perp},\,\,\,\forall\varphi\in \mathbb{T}^{\nu}  \right\},\\
& H^s_{S}(\mathbb{T}^{\nu+1}):=\left\{  u\in H^s(\mathbb{T}^{\nu+1}) : u(\varphi, \cdot)\in H_S\,\,\,\forall\varphi\in \mathbb{T}^{\nu}  \right\}.
\end{aligned}
\end{equation}
We will use the notation $a\le b$ to denote $a\le C\,b$ for some constant $C>0$. In particular, if the constant $C:=C(s)$ depends on the index $s$, then we will use the notation $a\le_s b$.\\

\noindent \textbf{Matrices with off-diagonal decay}. A linear operator can be identified with its matrix representation. We recall the definition of the $s$-decay norm (introduced in \cite{NLS}) of an infinite dimensional matrix. This norm is used in \cite{Airy} for the KAM reducibility scheme of the linearized operators.
\begin{defi}
The $s$-decay norm of an infinite dimensional matrix $A:=(A_{i_1}^{i_2})_{i_1, i_2\in \mathbb{Z}^b}, b\geq 1$ is
\begin{equation}\label{decayNorm}
\lvert A \rvert_s^2:=\sum_{i\in \mathbb{Z}^b}\langle i \rangle^{2\,s} \left( \sup_{i_1-i_2=i}\lvert A_{i_1}^{i_2}\rvert \right)^2.
\end{equation}
For parameter dependent matrices $A:=A(\omega), \omega\in \Omega_0\subseteq \mathbb{R}^{\nu}$, the definitions \eqref{suplip} and \eqref{Lippone} become
\begin{equation}\label{decayNorm2}
\begin{aligned}
&\lvert A \rvert_s^{\sup}:=\sup_{\omega\in \Omega_0}\lvert A(\omega) \rvert_s, \,\,\,\lvert A \rvert_s^{lip}:=\sup_{\omega_1\neq \omega_2} \frac{\lvert A(\omega_1)-A(\omega_2)\rvert_s}{\lvert \omega_1-\omega_2 \rvert},\\
& \lvert A \rvert_s^{Lip(\gamma)}:=\lvert A \rvert_s^{\sup}+\gamma\lvert A \rvert_s^{lip}.
\end{aligned}
\end{equation}
\end{defi}
Such a norm is modelled on the behavior of matrices representing the multiplication operator by a function. Actually, given a function $p\in H^s(\mathbb{T}^b)$, the multiplication operator $h\to p\,h$ is represented by the T\"oplitz matrix $T_i^j=p_{i-j}$ and $\lvert T \rvert_s=\lVert p \rVert_s$. If $p=p(\omega)$ is a Lipschitz family of functions, then
\[
\lvert T \rvert_s^{Lip(\gamma)}=\lVert p \rVert_s^{Lip(\gamma)}.
\]
The $s$-norm satisfies classical algebra and interpolation inequalities proved in \cite{NLS}.
\begin{lem}
Let $A=A(\omega), B=B(\omega)$ be matrices depending in a Lipschitz way on the parameter $\omega\in\Omega_0\subseteq \mathbb{R}^{\nu}$. Then for all $s\geq s_0> b/2$ there are $C(s)\geq C(s_0)\geq 1$ such that
\begin{align*}
&\lvert A\,B \rvert_s^{Lip(\gamma)}\le C(s) \lvert A \rvert_s^{Lip(\gamma)}\lvert B \rvert_s^{Lip(\gamma)},\\
&\lvert A\,B \rvert_s^{Lip(\gamma)}\le C(s)\lvert A \rvert_s^{Lip(\gamma)}\lvert B \rvert_{s_0}^{Lip(\gamma)}+C(s_0)\lvert A \rvert_{s_0}^{Lip(\gamma)}\lvert B \rvert_s^{Lip(\gamma)}.
\end{align*}
\end{lem}
The $s$-decay norm controls the Sobolev norm, namely
\begin{equation}
\lVert A h \rVert_s^{Lip(\gamma)}\le C(s)\left( \lvert A \rvert_{s_0}^{Lip(\gamma)}\lVert h \rVert_s^{lip(\gamma)}+\lvert A \rvert_s^{Lip(\gamma)}\lVert h \rVert_{s_0}^{Lip(\gamma)} \right).
\end{equation}
Let now $b:=\nu+1$. An important sub-algebra is formed by the\textit{ T\"oplitz in time matrices} defined by
\[
A_{(l_1, j_1)}^{(l_2, j_2)}:=A_{j_1}^{j_2}(l_1-l_2),
\]
whose decay norm \eqref{decayNorm} is
\begin{equation}
\lvert A \rvert_s^2=\sum_{j\in \mathbb{Z}, l\in\mathbb{Z}^{\nu}} \left(  \sup_{j_1-j_2=j} \lvert  A_{j_1}^{j_2}(l)\rvert  \right)^2\langle l, j \rangle^{2\,s}.
\end{equation}
These matrices are identified with the $\varphi$-dependent family of operators
\[
A(\varphi):=(A_{j_1}^{j_2}(\varphi))_{j_1, j_2\in \mathbb{Z}}, \quad A_{j_1}^{j_2}(\varphi):=\sum_{l\in\mathbb{Z}^{\nu}} A_{j_1}^{j_2}(l)\,e^{\mathrm{i}\,l\cdot\varphi}
\]
which act on functions of the $x$-variables as
\[
A(\varphi): h(x)=\sum_{j\in\mathbb{Z}} h_j\,e^{\mathrm{i}\,j\,x} \mapsto A(\varphi) h(x)=\sum_{j_1, j_2\in\mathbb{Z}} A_{j_1}^{j_2}(\varphi)h_{j_2}\,e^{\mathrm{i}\,j_1\,x}.
\]
All the transformations that we construct in this paper are of this type (with $j, j_1, j_2\neq 0$ because they act on the phase space $H_0^1(\mathbb{T}_x)$).
\begin{defi}
We say that
\begin{itemize}
\item[$(1)$] a map is symplectic if it preserves the $2$-form $\Omega$ in \eqref{SymplecticForm};
\item[$(2)$] an operator $(A h)(\varphi, x):=A(\varphi) h(\varphi, x)$ is symplectic if each $A(\varphi), \varphi\in\mathbb{T}^{\nu}$, is a symplectic map of the phase space (or of a symplectic subspace like $H^{\perp}_S$);
\item[$(3)$] the operator $\omega
\cdot\partial_{\varphi}-\partial_x G(\varphi)$ is Hamiltonian if each $G(\varphi), \varphi\in\mathbb{T}^{\nu}$, is symmetric;
\item[$(4)$] an operator is real if it maps real-valued functions into real-valued functions.
\end{itemize}
\end{defi}
A Hamiltonian operator is transformed, under a symplectic map, into another Hamiltonian operator, see \cite{Airy}-Section $2.3$.\\
We conclude this preliminary section recalling the following well known lemmata about composition of functions (see, e.g., Appendix in \cite{Airy}).
\begin{lem}{(Change of variables)}\label{ChangeofVariablesLemma}
Let $p\in W^{s, \infty}(\mathbb{T}^d, \mathbb{R^d}), s\geq 1$, with $\lvert p \rvert_{1,\infty}\le 1/2$. Then the function $f(x)=x+p(x)$ is invertible, with inverse $f^{-1}(y)=y+q(y)$ where $q\in W^{s,\infty}(\mathbb{T}^d, \mathbb{R}^d),$ and $\lvert q \rvert_{s, \infty}\le C \lvert p \rvert_{s, \infty}$.\\
If, moreover, $p$ depends in a Lipschitz way on a parameter $\omega\in\Omega\subseteq\mathbb{R}^{\nu}$, and $\lVert D_x p \rVert_{L^{\infty}}\le 1/2$ for all $\omega$, then $\lvert q \rvert_{s,\infty}^{Lip(\gamma)}\le C \lvert p \rvert_{s+1, \infty}^{Lip(\gamma)}$. The constant $C:=C(d, s)$ is independent of $\gamma$.\\
If $u\in H^s(\mathbb{T}^d, \mathbb{C})$ then $(u\circ f)(x):=u(x+p(x))$ satisfies
\begin{align*}
& \lVert u\circ f \rVert_s\le C (\lVert u \rVert_s+\lvert p \rvert_{s, \infty} \lVert u \rVert_1),\quad \lVert u\circ f - u \rVert_s\le C (\lVert p \rVert_{L^{\infty}}\lVert u \rVert_{s+1}+\lvert p \rvert_{s, \infty}\lVert u \rVert_2),\\
& \lVert u\circ f \rVert_s^{Lip(\gamma)}\le C (\lVert u \rVert_{s+1}^{Lip(\gamma)}+\lvert p \rvert_{s, \infty}^{Lip(\gamma)}\lVert u \rVert_2^{Lip(\gamma)}).
\end{align*}
The function $u\circ f^{-1}$ satisfies the same bounds.
\end{lem}

\begin{lem}{(Tame product)}
Let $s\geq s_0>d/2$. Then, for all $u, v\in H^s(\T^{d})$, we have 
\begin{equation}\label{TameProduct}
\lVert u\,v \rVert_s\le C(s_0)\lVert u \rVert_{s_0}\lVert v \rVert_s+C(s)\lVert u \rVert_{s} \lVert v \rVert_{s_0}.
\end{equation}
\end{lem}
A function $f\colon \mathbb{T}^d\times B_1\to \mathbb{C}$, where $B_1:=\{y\in \mathbb{R}^m : \lvert y \rvert<1  \}$, induces the composition operator
\begin{equation}\label{composition}
\tilde{f}(u)(x):=f(x, u(x), Du(x), \dots, D^p u(x))
\end{equation}
where $D^k u(x)$ denotes the partial derivatives $\partial_x^{\alpha} u$ of order $\lvert \alpha\rvert=k$. 
\begin{lem}{(Composition of functions)}\label{lemmacomp}
Assume $f\in C^r(\mathbb{T}^d\times B_1)$. Then for all $u\in H^{r+p}$ such that $\lvert u \rvert_{p, \infty}<1$, the composition operator \eqref{composition} is well defined and $\lVert \tilde{f}(u) \rVert_r\le C \lVert f \rVert_{C^r}(\lVert u \rVert_{r+p}+1)$, where the constant $C$ depends on $r, d, p$. If $f\in C^{r+2}$ then for all $\lvert u \rvert_{p, \infty}, \lvert h \rvert_{, \infty}<1/2$,
\begin{align}
\lVert \tilde{f}(u+h)-\sum_{i=0}^k \frac{\tilde{f}^{(i)}(u)}{i!}[h^i] \rVert_r\le C\,\lVert f \rVert_{C^{r+2}}\lVert h \rVert^k_{L^{\infty}}(\lVert h \rVert_{r+p}+\lVert h \rVert_{L^{\infty}}\lVert u \rVert_{r+p}).\label{VariazioneGenerale}
\end{align}
\end{lem}
\begin{lem}\label{lemmaLip}
Let $d\in\mathbb{N},\, d/2 \,< s_0\le s,\, p\geq 0,\, \gamma>0$. Let $F$ be a $C^1$-map satisfying  the tame estimates: for all $\lVert u \rVert_{s_0+p}\le 1, h\in H^{s+p}$,
\begin{align*}
&\lVert F(u) \rVert_s\le C(s) (1+\lVert u \rVert_{s+p}),\\
&\lVert \partial_u F(u)[h] \rVert_s\le C(s)(\lVert h \rVert_{s+p}+\lVert u \rVert_{s+p}\lVert h \rVert_{s_0+p}).
\end{align*} 
For $\Omega_0\subset\mathbb{R}^{\nu}$, let $u(\omega)$ be a Lipschitz family of functions parametrized by $\omega\in\Omega_0$ with $\lVert u \rVert_{s_0+p}^{Lip(\gamma)}\le 1$. Then
\[
\lVert F(u) \rVert_s^{Lip(\gamma)}\le C(s)(1+ \lVert u \rVert_{s+p}^{Lip(\gamma)}).
\]
\end{lem}


\subsection{Fourier representation}
In order to solve the homological equations along the Birkhoff normal form procedures performed in Sections $3$ and $8$, it is convenient to use the Fourier representation
\begin{equation}
u(x)=\sum_{j\in\mathbb{Z}\setminus\{0\}} u_j\,e^{\mathrm{i}\,j\,x},
\end{equation}
where the support of $u$ excludes the zero because the elements of the phase space have zero average. Moreover, $\overline{u}_j=u_{-j}$, since the function $u$ is real-valued. The symplectic structure \eqref{SymplecticForm} writes
\begin{equation}
\Omega=\frac{1}{2}\sum_{j\neq 0} \frac{1}{\mathrm{i} j}\,d u_j\wedge d u_{-j}, \qquad \Omega(u, v)=\sum_{j\neq 0} \frac{1}{\mathrm{i} j}u_j\,v_{-j},
\end{equation}
the Hamiltonian vector field $X_H$ and the Poisson bracket \eqref{PoissonBracket} are respectively
\begin{equation}
[X_H(u)]_j=\mathrm{i}\,j\,\partial_{u_{-j}} H(u), \quad \{ F, G \}(u)=-\sum_{j\neq 0} \mathrm{i}\,j\,(\partial_{u_{-j}} F)(u)(\partial_{u_j} G)(u).
\end{equation}
We say that a homogeneous Hamiltonian of degree $n$
\begin{equation}
H(u)=\sum_{j_1, \dots, j_n\in\mathbb{Z}\setminus\{0\}} H_{j_1, \dots, j_n} u_{j_1}\dots u_{j_n}
\end{equation}
preserves the \textit{momentum} if it is supported on the set $\{ (j_1, \dots, j_n)\in\mathbb{Z}^n\setminus\{ \textbf{0} \} : j_1+\dots+ j_n=0 \}$, where we denote with $\{\textbf{0}\}$ the origin of any vector space $\mathbb{R}^n$, or, equivalently, if
\[
\{ H, M \}=0, \quad M(u)=\int_{\mathbb{T}} u^2\,dx.
\]
 We note that, by the presence of the $x$ in the arguments of the function $f$ in \eqref{KdVHamiltonian}, the \textit{momentum} is not preserved along the orbits of the equation \eqref{KdV}.\\
\section{Weak Birkhoff Normal form}

The Hamiltonian \eqref{KdVHamiltonian} is $H=H_2+H_3+H_4+H_{\geq 5}$, where
\begin{equation}\label{Hamiltonians}
\begin{aligned}
&H_2(u):=\frac{1}{2}\int_{\mathbb{T}} u_x^2\,dx, \quad H_3(u):=\int_{\mathbb{T}} c_1\,u_x^3+c_2\,u_x^2\,u+c_3\,u^3 \,dx,\\
&H_4(u):=\int_{\mathbb{T}} c_4\, u_x^4+c_5\,u_x^3\,u+c_6\,u_x^2\,u^2+c_7 \,u^4 \,dx, \quad H_{\geq 5}(u):=\int_{\mathbb{T}} f_{\geq 5}(x, u, u_x) \,dx.
\end{aligned}
\end{equation}
For a finite dimensional space 
\begin{equation}\label{FinitedimensionalSubspace}
E:=E_C:=\mbox{span}\left\{ e^{\mathrm{i}\,j\,x} : 0<\lvert j \rvert\le C \right\}, \quad C>0,
\end{equation}
let $\Pi_E$ denote the corresponding $L^2$-projector on $E$.\\
The notation $R(v^{k-q} z^q)$ indicates a homogeneous polynomial of degree $k$ in $(v, z)$ of the form
\[
R(v^{k-q} z^q)=M[\underbrace{v, \dots, v}_{(k-q)\,\,times},\underbrace{z, \dots, z}_{q\,\,times}], \quad M=k-\mbox{linear}.
\]
We denote with $H_{n, \geq k}, H_{n, k}, H_{n, \le k}$ the terms of type $R(v^{n-s}\,z^s)$, where, respectively, $s\geq k, s=k, s\le k$, that appear in the homogeneous polynomial $H_n$ of degree $n$ in the variables $(v, z)$.\\
In particular, we have
\begin{align}
&H_{3, \le 1}=\int_{\mathbb{T}} \left\{c_1(v_x^3+3\,v_x^2\,z_x)+c_2(v_x^2\,v+2\,v_x\,v\,z_x+v_x^2\,z)+c_3(v^3+3\,v^2\,z)\right\}\,dx,\label{H3minore}\\
&H_{3, \geq 2}=\int_{\mathbb{T}} \{c_1(z_x^3+3\,z_x^2\,v_x)+c_2(z_x^2\,z+z_x^2\,v+2\,z_x\,z\,v_x)+c_3(z^3+3\,v^2\,z)\} \,dx,\label{H3maggiore}\\
& H_{4, 0}=\int_{\mathbb{T}} \{c_4\,v_x^4+c_5\,v_x^3\,v+c_6\,v_x^2\,v^2+c_7\,v^4\}\,dx\label{H4zero}.
\end{align}
\begin{prop}{(\textbf{Weak Birkhoff Normal form})}\label{WBNF}
Assume Hypotesis $(\mathtt{S})$. Then there exists an analytic invertible transformation of the phase space $\Phi_B\colon H_0^1(\mathbb{T}_x)\to H_0^1(\mathbb{T}_x)$ of the form
\begin{equation}\label{IdpiuFiniteRank}
\Phi_B(u)=u+\Psi(u), \quad \Psi(u):=\Pi_E \Psi(\Pi_E u),
\end{equation}
where $E$ is a finite dimensional space as in \eqref{FinitedimensionalSubspace}, such that the transformed Hamiltonian is
\begin{equation}\label{HamiltonianaNormalizzata}
\mathcal{H}=H\circ \Phi_B=H_2+\mathcal{H}_3+\mathcal{H}_4+\mathcal{H}_5+\mathcal{H}_{\geq 6},
\end{equation}
where $H_2$ is defined in \eqref{Hamiltonians},
\begin{equation}\label{HamiltonianeMathcal}
\begin{aligned}
&\mathcal{H}_3=c_1\int_{\mathbb{Z}} (z_x^3+3\,z_x^2\,v_x)\,dx+c_2\int_{\mathbb{Z}} (z_x^2\,z+z_x^2\,v+2\,v_x\,z_x\,z)\,dx+c_3\int_{\mathbb{T}} (z^3+3 v\,z^2)\,dx,\\
&\mathcal{H}_4=H_{4, 0}^{(4)}+\mathcal{H}_{4,2}+\mathcal{H}_{4, 3}+\mathcal{H}_{4, 4},\quad \mathcal{H}_{4, 2}=R(v^2\,z^2), \quad \mathcal{H}_{4, 3}=R(v\,z^3),\\ & \mathcal{H}_{4, 4}=\int_{\mathbb{T}} c_4\,z_x^4+c_5\,z_x^3\,z+c_6\,z_x^2\,z^2+c_7\,z^4\,dx, \quad\mathcal{H}_5=\sum_{q=2}^5 R(v^{5-q}\,z^q),
\end{aligned}
\end{equation}
$H_{4, 0}^{(4)}$ is defined in \eqref{H4risonante} and $\mathcal{H}_{\geq 6}$ collects all the terms of order at least six in $(v, z)$.
\end{prop}
The rest of this section is devoted to the proof of the Proposition \ref{WBNF}.\\
We construct a symplectic map $\Phi_B$ as the composition of analytic and invertible transformations on the phase space that eliminates the terms linear in $z$ and independent of it from the Hamiltonian \eqref{KdVHamiltonian}. In this way, the Hamiltonian system \eqref{KdV} tranforms into one that is integrable and non-isocronous on the subspace $\{ z=0\}$.
\begin{remark}\label{wellDef}
We note that if $j_1, \dots , j_N\in\mathbb{Z}\setminus\{ 0\}, j_1+\dots+ j_N=0$ and at most one of these integers does not belong to $S$, then $\max_{i=1,\dots, N} \lvert j_i \rvert \le (N-1) C_S$, where $C_S:=\max_{j\in S} \lvert j \rvert$. Thus, the vector field $X_{F^{(N)}}$, generated by the finitely supported Hamiltonian 
\[
F^{(N)}=\sum_{j_1+\dots+j_N=0} F^{(N)}_{j_1 \dots j_N} u_{j_1}\dots u_{j_N},
\]
is finite rank, and, in particular, it vanishes outside the finite dimensional subspace $E:=E_{(N-1) C_S}$ (see \eqref{FinitedimensionalSubspace} ) and it has the form 
\[
X_{F^{(N)}}(u)=\Pi_E X_{F^{(N)}} (\Pi_E u).
\] 
Hence its flow $\Phi^{(N)}$ is analytic and invertible on the phase space $H_0^1(\mathbb{T}_x)$.
\end{remark}
\noindent\textbf{Step one}. First we remove the cubic terms independent of $z$ and linear in $z$ from the Hamiltonian $H_3$ defined in \eqref{Hamiltonians} .
We look for a symplectic transformation $\Phi^{(3)}$ of the phase space which eliminates the monomials $u_{j_1}\,u_{j_2}\,u_{j_3}$ of $H_3$ with at most one index outside $S$.\\
We look for $\Phi^{(3)}:=(\Phi^t_{F^{(3)}})_{|_{t=1}}$ as the time$-1$ flow map generated by the Hamiltonian vector field $X_{F^{(3)}}$, with an auxiliary Hamiltonian of the form
\[
F^{(3)}(u):=\sum_{j_1+j_2+j_3=0} F^{(3)}_{j_1\,j_2\,j_3}\,u_{j_1}\,u_{j_2}\,u_{j_3}.
\]
The transformed Hamiltonian is 
\begin{equation}\label{TransformedHamiltonian}
\begin{aligned}
& H^{(3)}:=H\circ \Phi^{(3)}=H_2+H_{3}^{(3)}+H^{(3)}_4+H^{(3)}_{\geq 5},\\
& H^{(3)}_3=H_3+\{ H_2, F^{(3)} \}, \quad H^{(3)}_4=\frac{1}{2} \{\{ H_2, F^{(3)} \}, F^{(3)} \}+\{ H_3, F^{(3)} \}+H_4,
\end{aligned}
\end{equation}
where $H_{\geq 5}^{(3)}$ collects all the terms of order at least five in $(v, z)$.
In order to find the exact expression of $F^{(3)}$, we have to solve the homological equation
\begin{equation}\label{FirstHomolEq}
H_3+\{H_2, F^{(3)}\}=H_{3, \geq 2}
\end{equation}
or, equivalently, $\{ H_2, F^{(3)}\}=-H_{3, \le 1}$, see \eqref{H3minore}. In the Fourier representation, by \eqref{PoissonBracket} and \eqref{Hamiltonians}, the equation \eqref{FirstHomolEq} writes
\begin{equation}\label{F3}
\sum_{j_1+j_2+j_3=0} \mathrm{i}\,(j_1^3+j_2^3+j_3^3)\,F_{j_1 j_2 j_3}^{(3)}\,u_{j_1}\,u_{j_2}\,u_{j_3}=\sum_{(j_1, j_2, j_3)\in\mathcal{A}_3} (-\mathrm{i}\,c_1\, j_1 j_2 j_3-c_2\,j_1 j_2+c_3)\, u_{j_1} u_{j_2} u_{j_3}
\end{equation}
where
\[
\mathcal{A}_3:=\{ (j_1, j_2, j_3)\in\mathbb{Z}^3\setminus\{ \textbf{0} \} : j_1+j_2+j_3=0\,\,\mbox{and at least}\,\,2\,\,\mbox{indices among}\,\,j_1, j_2, j_3\,\,\mbox{belong to}\,\,S  \}.
\]
We note that if $(j_1, j_2, j_3)\in \mathcal{A}_3$ then $j_1^3+j_2^3+j_3^3\neq 0$, because
\begin{equation}
j_1+j_2+j_3=0  \quad  \Rightarrow \quad  j_1^3+j_2^3+j_3^3=3\,j_1\,j_2\,j_3
\end{equation}
and $j_1, j_2, j_3\in\mathbb{Z}\setminus\{0\}$.\\
Hence, to solve the equation \eqref{FirstHomolEq} we choose
\begin{equation}\label{F3def}
F^{(3)}_{j_1 j_2 j_3}:=\begin{cases}
\dfrac{-\mathrm{i}\,c_1\, j_1 j_2 j_3-c_2\,j_1 j_2+c_3}{\mathrm{i}(j_1^3+j_2^3+j_3^3)} \qquad  \mbox{if}\,\,(j_1, j_2, j_3)\in\mathcal{A}_3,\\
0 \qquad \qquad \qquad \qquad \qquad \qquad\,\,\,\,\mbox{otherwise}.
\end{cases}
\end{equation}
By construction, all the monomials of $H_3$ with at least two indices outside $S$ are not modified by the transformation $\Phi^{(3)}$. Hence we have
\begin{align}
H_3^{(3)}=c_1\int_{\mathbb{Z}} (z_x^3+3\,z_x^2\,v_x)\,dx+c_2\int_{\mathbb{Z}} (z_x^2\,z+z_x^2\,v+2\,v_x\,z_x\,z)\,dx+c_3\int_{\mathbb{T}} (z^3+v\,z^2)\,dx.
\end{align}
Now we compute the fourth order term $H_4^{(3)}$ in \eqref{TransformedHamiltonian}. We have, by \eqref{FirstHomolEq}
\begin{equation}
H^{(3)}_4=\frac{1}{2} \{\{ H_2, F^{(3)} \}, F^{(3)} \}+\{ H_3, F^{(3)} \}+H_4=\frac{1}{2}\{H_{3, \le 1}, F^{(3)}\}+\{ H_{3}^{(3)}, F^{(3)} \}+H_4
\end{equation}
and by \eqref{F3} and \eqref{F3def} 
\begin{equation}
\begin{aligned}
F^{(3)}(u)=&-\frac{c_1}{3}\int_{\mathbb{T}} v^3\,dx-c_1\int_{\mathbb{T}} v^2\,z\,dx-\frac{c_2}{3}\int_{\mathbb{T}} (\partial_x^{-1} v)\,v^2\,dx-\frac{c_2}{3}\int_{\mathbb{T}} v^2\,(\partial_x^{-1} z)\,dx-\\[2mm]
&-\frac{2\,c_2}{3}\int_{\mathbb{T}} v\,(\partial_x^{-1} v)\,z\,dx-\frac{c_3}{3}\int_{\mathbb{T}} (\partial_x^{-1} v)^3\,dx-c_3\int_{\mathbb{T}} (\partial_x^{-1} v)^2\,(\partial_x^{-1} z)\,dx.
\end{aligned}
\end{equation}
Thus
\begin{equation}\label{partialXdiF3}
\begin{aligned}
\partial_x\nabla F^{(3)}(u)=&-c_1\partial_x(v^2)-2\,c_1\partial_x \Pi_S[v\,z]+\frac{c_2}{3} \pi_0[v^2]-\frac{c_2}{3} \partial_{xx} [(\partial_x^{-1} v)^2]-\\
&-\frac{2\,c_2}{3} \partial_x\Pi_S[(\partial_x^{-1}v) z+(\partial_x^{-1} z)v]+\frac{2\,c_2}{3}\Pi_S[v\,z]
+c_3\pi_0[(\partial_x^{-1} v)^2]+\\
&+2\,c_3\,\Pi_S[(\partial_x^{-1} v)(\partial_x^{-1} z)]
\end{aligned}
\end{equation}
where $\pi_0$ denotes the projection on the space of functions with zero space average, namely
\[
\pi_0 [u]=u(x)-\frac{1}{2\pi}\int_{\mathbb{T}} u(x)\,dx.
\]
By \eqref{H3minore}, we get
\begin{equation}\label{GradH3minore}
\begin{aligned}
\nabla H_{3, \le 1}(u)=&-3\,c_1\partial_x(v_x^2)-6\,c_1\partial_x\Pi_S[v_x\,z_x]-c_2\partial_{xx} (v^2)-2\,c_2\partial_{xx}\Pi_S[v\,z]+\\
&+c_2\pi_0[v_x^2]+2\,c_2\Pi_S[v_x\,z_x]+3\,c_3\pi_0[v^2]+6\,c_3\Pi_S[v\,z].
\end{aligned}
\end{equation}
Hence, by \eqref{PoissonBracket}, \eqref{partialXdiF3}, \eqref{GradH3minore}, we have
\begin{equation}\label{Poisson1/2H3}
\begin{aligned}
\frac{1}{2} \{ H_{3, \le 1}, F^{(3)}\}&=\frac{3\,c_1^2}{2}\int_{\mathbb{T}} \partial_x (v_x^2)\,\partial_x(v^2) \,dx-\frac{c_1\,c_2}{2}\int_{\mathbb{T}} v^2\,\partial_x(v_x^2) \,dx+\frac{c_1\,c_2}{2}\int_{\mathbb{T}} \partial_x (v_x^2) \partial_{xx}[(\partial_x^{-1} v)^2] \,dx\\
&-\frac{3\,c_1\,c_3}{2} \int_{\mathbb{T}} \partial_x (v_x^2)\,(\partial_x^{-1} v)^2+\frac{c_2^2}{6}\int_{\mathbb{T}} (\partial_x(v^2))^2 \,dx+\frac{c_2^2}{6}\int_{\mathbb{T}} \partial_{xx} (v^2)\,\partial_{xx} [(\partial_x^{-1} v)^2] \,dx \\
&-c_2\,c_3\int_{\mathbb{T}} \partial_{xx} (v^2)\,(\partial_x^{-1} v)^2\,dx-\frac{c_1\,c_2}{2}\int_{\mathbb{T}} v_x^2\,\partial_x(v^2)\,dx+\frac{c_2^2}{6}\int_{\mathbb{T}} v_x^2\,\pi_0[v^2] \,dx\\
&-\frac{c_2^2}{6}\int_{\mathbb{T}}  v_x^2\,\partial_{xx}[(\partial_x^{-1} v)^2] \,dx+\frac{c_2\,c_3}{2}\int_{\mathbb{T}} v_x^2\,\pi_0[(\partial_x^{-1} v)^2]+\frac{c_2\,c_3}{2}\int_{\mathbb{T}} (\pi_0[v^2])^2 \,dx\\
&-\frac{3\,c_3^2}{2}\int_{\mathbb{T}} v^2\,\pi_0[(\partial_x^{-1} v)^2] \,dx+R(v^3\,z)+R(v^2\,z^2).
\end{aligned}
\end{equation}
By \eqref{H3maggiore}, we get
\begin{equation}\label{GradH3magg}
\begin{aligned}
\nabla H_{3}^{(3)}(u)=&-3\,c_1\,\partial_x(z_x^2)-6\,c_1\,\partial_x\Pi_S^{\perp} [v_x\,z_x]-c_2\partial_{xx} (z^2)+c_2 \pi_0[z_x^2]-2\,c_2\partial_{xx} \Pi_S^{\perp}[v\,z]+\\
&+2\,c_2\Pi_S^{\perp}[v_x\,z_x]+3\,c_3 \pi_0[z^2]+2\,c_3 \Pi_S^{\perp}[v\,z].
\end{aligned}
\end{equation}
Thus by \eqref{PoissonBracket}, \eqref{partialXdiF3}, \eqref{GradH3magg}, we have
\begin{equation}\label{H4,2}
\begin{aligned}
\{ H_3^{(3)}, F^{(3)} \}&=3\,c_1^2\int_{\mathbb{T}} \partial_x(z_x^2)\,\partial_x (v^2)\,dx-c_1\,c_2\int_{\mathbb{T}} v^2\,\partial_x(z_x^2) \,dx+\\
&+c_1\,c_2\int_{\mathbb{T}} \partial_x (z_x^2)\,\partial_{xx}[(\partial_x^{-1} v)^2] \,dx -3\,c_1\,c_3\int_{\mathbb{T}}  (\partial_x^{-1}v)^2\,\partial_x (z_x^2) \,dx+\\
&+c_1\,c_2\int_{\mathbb{T}} (\partial_x^{-1} v)^2\,\partial_x(z_x^2) \,dx-\frac{c_2^2}{3}\int_{\mathbb{T}} v^2\,\partial_{xx} (z^2) \,dx+\\
&+\frac{c_2^2}{3}\int_{\mathbb{T}} \partial_{xx} (z^2)\,\partial_{xx} [(\partial_x^{-1} v)^2] \,dx-c_2\,c_3\int_{\mathbb{T}}   (\partial_x^{-1} v)^2 \partial_{xx}(z^2)\,dx-\\
&-c_1\,c_2\int_{\mathbb{T}} z_x^2\,\partial_x(v^2) \,dx+\frac{c_2^2}{3}\int_{\mathbb{T}} z_x^2\,\pi_0[v^2] \,dx-\\
&-\frac{c_2^2}{3}\int_{\mathbb{T}} z_x^2\,\partial_{xx} [(\partial_x^{-1} v)^2] \,dx-c_2\,c_3\int_{\mathbb{T}} z_x^2\,\pi_0[(\partial_x^{-1} v)^2] \,dx-\\
&-3\,c_1\,c_3\int_{\mathbb{T}} z^2\,\partial_x(v^2) \,dx+c_2\,c_3\int_{\mathbb{T}} v^2\,\pi_0[z^2] \,dx-\\
&-c_2\,c_3\int_{\mathbb{T}}  z^2\,\partial_{xx}[(\partial_x^{-1} v)^2] \,dx+3\,c_3^2\int_{\mathbb{T}} (\partial_x^{-1} v)^2\,\pi_0[z^2] \,dx+\\
&+R(v^3\,z)+R(v\,z^3).
\end{aligned}
\end{equation}
\textbf{Step two}. We now construct a symplectic map $\Phi^{(4)}$ to eliminate the term $H_{4,1}^{(3)}$ (which is linear in $z$) and to normalize $H_{4, 0}^{(3)}$ (which is independent of $z$). We need the following elementary lemma (Lemma 13.4 in \cite{KdVeKAM}).
\begin{lem}\label{AlgebraicLemma}
Let $j_1, j_2, j_3, j_4\in\mathbb{Z}$ such that $j_1+j_2+j_3+j_4=0$. Then
\[
j_1^3+j_2^3+j_3^3+j_4^3=-3 (j_1+j_2) (j_1+j_3) (j_2+j_3).
\]
\end{lem}
We look for a map $\Phi^{(4)}:=(\Phi^t_{F^{(4)}})_{|_{t=1}}$ which is the time$-1$ flow map of an auxiliary Hamiltonian
\[
F^{(4)}(u):=\sum_{\substack{j_1+j_2+j_3+j_4=0,\\ at\,\, least\,\, 3\,\, indices\,\, belong \,\,to\,\, S } }F^{(4)}_{j_1 j_2 j_3 j_4}\,u_{j_1} u_{j_2} u_{j_3} u_{j_4},
\]
which has the same form of the Hamiltonian $H_{4, 0}^{(3)}+H_{4, 1}^{(3)}$. The transformed Hamiltonian is
\begin{equation}\label{H^{(4)}}
H^{(4)}:=H^{(3)}\circ \Phi^{(4)}=H_2+H_3^{(3)}+H_4^{(4)}+H_{\geq 5}^{(4)}, \quad H_4^{(4)}:=\{ H_2, F^{(4)}\}+H_4^{(3)}
\end{equation}
and $H^{(4)}_{\geq 5}$ collects all the terms of order at least five in $(v, z)$. We write
\begin{equation}
H^{(3)}_4(u)=\sum_{j_1+j_2+j_3+j_4=0} H^{(3)}_{4, \,j_1 j_2 j_3 j_4} u_{j_1 j_2 j_3 j_4}.
\end{equation}
This makes sense since $H_{3, \le 1}, H_3^{(3)}$ and $F^{(3)}$ preserve the momentum, hence also $H^{(3)}_4$ does it. We choose the coefficients
\begin{equation}\label{F4def}
F^{(4)}_{j_1 j_2 j_3 j_4}:=
\begin{cases}
\dfrac{H^{(3)}_{4,\,j_1 j_2 j_3 j_4}}{\mathrm{i}(j_1^3+j_2^3+j_3^3+j_4^3)} \qquad  \mbox{if}\,\,(j_1, j_2, j_3, j_4)\in\mathcal{A}_4,\\[2mm]
0 \qquad \qquad \qquad \qquad \qquad \mbox{otherwise},\\
\end{cases}
\end{equation}
where 
\begin{align*}
\mathcal{A}_4:=\{ (j_1, j_2, j_3, j_4)\in \mathbb{Z}^4\setminus\{\textbf{0}\}\, :\, &j_1+j_2+j_3+j_4=0, j_1^3+j_2^3+j_3^3+j_4^3\neq 0,\\
& \mbox{and at most one among}\,\,j_1, j_2, j_3, j_4\,\,\mbox{outside}\,\,S\}.
\end{align*}
By this definition, the symmetry of $S$ and the Lemma \ref{AlgebraicLemma}, we have $H_{4, 1}^{(4)}=0$, because there no exist $j_1, j_2, j_3\in S$ and $j_4\in S^c$ such that $j_1+j_2+j_3+j_4=0,\, j_1^3+j_2^3+j_3^3+j_4^3=0$. By construction, the terms $H^{(4)}_{4, i}=H_{4, i}^{(3)}, i=2, 3, 4$ are not changed by $\Phi^{(4)}$.\\
It remains to compute the resonant part of $H_{4, 0}^{(3)}$, i.e. the terms of $H_4^{(3)}$ of type $R(v^4)$ supported on the modes $(j_1, j_2, j_3, j_4)$ that do not belong to $\mathcal{A}_4$.\\
If we call
\begin{align*}
\mathcal{B}:=\{(j_1, j_2, j_3, j_4)\in S^4 \,\,:\,\,j_1+j_2+j_3+j_4=0,\,\, & j_1^3+j_2^3+j_3^3+j_4^3=0,\,\,j_1+j_2\neq 0\}
\end{align*}
then by \eqref{H4zero}, \eqref{Poisson1/2H3} we have
\begin{equation}\label{H4risonante}
\begin{aligned}
H_{4, 0}^{(4)}&=-\frac{3\,c_1^2}{2} \,\sum_{\mathcal{B}} (j_1+j_2)^2 j_3 j_4\,u_{j_1} u_{j_2} u_{j_3} u_{j_4}+\frac{c_2^2}{6}\, \sum_{\mathcal{B}} (j_3+j_4)^2\,u_{j_1} u_{j_2} u_{j_3} u_{j_4}-\frac{c_2^2}{6}\, \sum_{\mathcal{B}} j_3 j_4\,u_{j_1} u_{j_2} u_{j_3} u_{j_4}\\&-\frac{c_2^2}{6} \,\sum_{\mathcal{B}} (j_1+j_2)^2 (j_3+j_4)^2\frac{1}{j_1 j_2}\,u_{j_1} u_{j_2} u_{j_3} u_{j_4}+\frac{c_2^2}{6} \,\sum_{ \mathcal{B}} \frac{(j_1+j_2)^2\,j_3 j_4}{j_1 j_2}\,u_{j_1} u_{j_2} u_{j_3} u_{j_4}\\
&+\frac{3}{2} c_3^2 \sum_{\mathcal{B}} \frac{1}{\mathrm{i} j_1\,\mathrm{i} j_2} \,u_{j_1} u_{j_2} u_{j_3} u_{j_4}+\frac{c_2\,c_3}{2} \,\sum_{\mathcal{B}}u_{j_1} u_{j_2} u_{j_3} u_{j_4}-\frac{c_2\,c_3}{2} \,\sum_{\mathcal{B}} \frac{(j_1+j_2)^2}{j_1 j_2} u_{j_1} u_{j_2} u_{j_3} u_{j_4}\\
&-\frac{c_2 c_3}{2} \,\sum_{\mathcal{B}} \frac{(j_3+j_4)^2}{j_1\,j_2}\,u_{j_1} u_{j_2} u_{j_3} u_{j_4}+\frac{c_2 c_3}{2} \sum_{\mathcal{B}} \frac{j_3 j_4}{j_1 j_2}\,u_{j_1} u_{j_2} u_{j_3} u_{j_4}+c_4\sum_{\mathcal{B}\cup \{ j_1+j_2=0 \}} j_1\,j_2\,j_3\,j_4\,u_{j_1} u_{j_2} u_{j_3} u_{j_4}\\
&-c_6 \sum_{\mathcal{B}\cup\{j_1+j_2=0\}} j_1\,j_2\,u_{j_1} u_{j_2} u_{j_3} u_{j_4}+c_7 \sum_{\mathcal{B}\cup \{j_1+j_2=0\}} u_{j_1} u_{j_2} u_{j_3} u_{j_4}.
\end{aligned}
\end{equation}
By Lemma \ref{AlgebraicLemma}, if $j_1+j_2+j_3+j_4=0, j_1^3+j_2^3+j_3^3+j_4^3= 0$ then $(j_1+j_2) (j_1+j_3) (j_2+j_3)=0$. We develop all the sums in \eqref{H4risonante} with respect to the first index $j_1$. The possible cases are:
\begin{align*}
&(i)\,\,\{ j_2\neq -j_1, j_3=-j_1, j_4=-j_2\}\qquad (ii)\,\,\{ j_{2}\neq -j_1, j_3\neq -j_1, j_3=-j_2, j_4=-j_1\} \\[2mm]
&(iii)\,\, \{ j_1+j_2=0\}.
\end{align*}
If $I:=(I_{\overline{\jmath}_1}, \dots, I_{\overline{\jmath}_{\nu}})\in\mathbb{R}^{\nu}_+$ with $I_j:=\lvert u_j \rvert^2, j\in S$, we get
\begin{equation}\label{contoperRisonanzediH4}
\begin{aligned}
H_{4, 0}^{(4)}(I)=&-12\,c_1^2\sum_{j\in S^+} j^4\,I_j^2-24 c_1^2\sum_{\substack{j, j'\in S^+,\\ j\neq j'}} j^2\,j'^2\,I_{j}\,I_{j'}-\frac{7 c_2^2}{3}\sum_{j\in S^+} j^2\,I_j^2-\frac{8 c_2^2}{3}\sum_{\substack{j, j'\in S^+,\\ j\neq j'}} (j^2+j'^2)\,I_j\,I_{j'}\\
&-3\,c_3^2\sum_{j\in S^+} \frac{1}{j^2}\,I_j^2-2\,c_2\,c_3 \sum_{j\in S^+} I_j^2-8 c_2 c_3\sum_{\substack{j, j'\in S^+,\\ j\neq j'}} I_j\,I_{j'}+6\,c_4\sum_{j\in S^+} j^4\,I_j^2 +12 c_4\sum_{\substack{j, j'\in S,\\ j\neq j'}} j^2\,j'^2\,I_{j}\,i_{j'}\\
&+2 c_6\sum_{j\in S^+} j^2\,I_j^2+2 c_6\sum_{\substack{j, j'\in S^+,\\ j\neq j'}}(j^2+j'^2)\, I_j\,I_{j'}+6\,c_7\sum_{j\in S^+} I_j^2+ 12 c_7\sum_{\substack{j, j'\in S^+,\\ j\neq j'}} I_j\,I_{j'}.
\end{aligned}
\end{equation}
The Hamiltonian system $H_2+H_3^{(3)}+H_4^{(4)}$, obtained by truncation at order $4$ of the transformed Hamiltonian $H\circ \Phi^{(3)}\circ \Phi^{(4)}$, possesses the invariant submanifold $\{ z=0\}$, and, restricted to this subspace, it is integrable. Indeed, if we introduce on $H_S$ the action-angle variables $u\mapsto (\theta, I)$ by defining
\begin{equation}\label{ActionCoord}
u_j:=v_j=\sqrt{I_j}\,e^{\mathrm{i}\theta_j}, \qquad I_j=I_{-j}, \quad \theta_{-j}=-\theta_j \quad j\in S,
\end{equation}
the restriction of the Hamiltonian $H_2+H_3^{(3)}+H_4^{(4)}$ to $\{ z=0\}$, namely $\frac{1}{2}\int v_x^2\,dx+H_{4, 0}^{(4)}$, depends only on the actions $I_{\overline{\jmath}_1}, \dots, I_{\overline{\jmath}_{\nu}}$.
We will prove later that, for a generic choice of the tangential sites, this system is also non-isochronous (actually it is formed by $\nu$ decoupled oscillators).\\
Due to the presence of a quadratic nonlinearity in the equation \eqref{KdV}, we have to eliminate further monomials of $H^{(4)}$ in \eqref{H^{(4)}} in order to enter in a perturbative regime. Indeed, the minimal requirement for the convergence of the nonlinear Nash-Moser iteration is to eliminate the monomials $R(v^5)$ and $R(v^4\,z)$. Here we need the choice of the sites of Hypotesis $(\mathtt{S})$.\\

\noindent \textbf{Step three}. The homogeneous component of degree five of $H^{(4)}$ has the form
\[
H^{(4)}_5(u)=\sum_{j_1+\dots+j_5=0} H^{(4)}_{5,\,j_1,\dots, j_5}\,u_{j_1} u_{j_2} u_{j_3} u_{j_4} u_{j_5},
\]
indeed, the Hamiltonian $H^{(4)}_5$ preserves the momentum, because $f_5(u, u_x)$ does not depend on $x$ (see \eqref{parteomogeneagrado5}).
We want to remove from $H^{(4)}_5$ the terms with at most one index among $j_1, \dots, j_5$ outside $S$. We consider the auxiliary Hamiltonian
\begin{equation}\label{F5}
F^{(5)}=\sum_{\substack{j_1+\dots+j_5=0,\\ at\,\,most\,\,one\,\,index\,\,outside\,\,S}} F^{(5)}_{j_1, \dots, j_5}\,u_{j_1}\dots u_{j_5}, \quad F^{(5)}_{j_1, \dots, j_5}:=\dfrac{ H^{(4)}_{5,\,j_1,\dots, j_5}}{\mathrm{i}(j_1^3+\dots+j_5^3)}.
\end{equation}
Hypotesis $(\mathtt{S})$ implies that 
\begin{itemize}
\item[$(\mathtt{S}_0)$] there is no choice of $5$ integers $j_1, \dots, j_5\in S$ such that
\begin{equation}\label{s20}
j_1+\dots+j_5=0, \quad j_1^3+\dots+j_5^3=0,
\end{equation}
\item[$(\mathtt{S}_1)$] there is no choice of $4$ integers $j_1, \dots, j_4\in S$ and $j_5\in S^c$ such that \eqref{s20} holds.
\end{itemize}
Hence $F^{(5)}$ in \eqref{F5} is well defined. Let $\Phi^{(5)}$ be the time$-1$ flow generated by $X_{F^{(5)}}$. The new Hamiltonian is
\begin{equation}\label{H^{(5)}}
H^{(5)}:=H^{(4)}\circ \Phi^{(5)}=H_2+H_3^{(3)}+H_4^{(4)}+H^{(5)}_5+H^{(5)}_{\geq 6}, \quad H^{(5)}_5=\{ H_2, F^{(5)} \}+H^{(4)}_5,
\end{equation}
where $H^{(5)}_{\geq 6}$ collects all the terms of degree greater or equal than six, and, by the definition of $F^{(5)}$,
\begin{equation}
H^{(5)}_5=\sum_{q=2}^5 R(v^{5-q} z^q).
\end{equation}
Setting $\Phi_B:=\Phi^{(3)}\circ\Phi^{(4)}\circ \Phi^{(5)}$ and renaming $\mathcal{H}:=H^{(5)}=H\circ\Phi_B, \mathcal{H}_n=H_n^{(n)}$, by Remark \eqref{wellDef}, we conclude the proof of Proposition \ref{WBNF}.

\section{Action-angle variables}
Consider the change of variable $v\mapsto (\theta, I)$ in \eqref{ActionCoord}, where the actions $I$ are defined in the positive half space $\{v\in\mathbb{R}^{\nu} : v_i\geq 0, \forall i=1,\dots, \nu\}$ and $\theta\in\mathbb{T}^{\nu}$. The symplectic form in \eqref{SymplecticForm} restricted to the subspace $H_S$ transforms into the $2$-form
\begin{equation}
\tilde{\Omega}_S=\sum_{j\in S^+} d\theta_j\wedge \frac{1}{j}\, d I_j.
\end{equation}
Hence the Hamiltonian system $\mathcal{H}_{\le 5}:=H_2+H_3^{(3)}+H_4^{(4)}+H^{(5)}_5$ restricted to $\{ z=0\}$ writes
\begin{equation}
\begin{cases}
\dot{\theta}_j=j\,\,\dfrac{\partial}{\partial I_j} \mathcal{H}_{\leq 5}(\theta, I, 0), \qquad j\in S^+,\\[3mm]
\dot{I}_j=-\dfrac{\partial}{\partial \theta_j} \mathcal{H}_{\leq 5}(\theta, I, 0), \qquad j\in S^+.
\end{cases}
\end{equation}
We have that 
\begin{equation}\label{HamBif}
\tilde{h}(I):=\mathcal{H}_{\leq 5}(\theta, I, 0):=\sum_{j\in S^+} j^2\,I_j+H_{4, 0}^{(4)}(I)
\end{equation}
depends only by the actions $I$, and, if we call $\omega_j(I):=j\,\,\partial_{I_j}\tilde{h}(I)$, we have
\begin{equation}\label{HamiltonianSistemino}
\begin{cases}
\dot{\theta}_j=\omega_j(I), \qquad j\in S^+,\\[3mm]
\dot{I}_j=0, \qquad \quad\,\,\,\, j\in S^+.
\end{cases}
\end{equation}
By \eqref{contoperRisonanzediH4} 
\begin{equation}\label{omeghino}
\begin{aligned}
\omega_j(I)=&j^3-24\,c_1^2\, j^5\,I_j-48 c_1^2\,j^3\,\sum_{k\in S^+, k\neq j} k^2 I_k-c_2^2 \frac{14}{3}\,j^3\,I_j-\frac{16}{3} c_2^2\, j^3 \sum_{k\in S^+, k\neq j} I_k\\
&-\frac{16}{3} c_2^2\, j \sum_{k\in S^+, k\neq j} k^2 I_k-6\,c_3^2\,\frac{1}{j}\,I_j-2\,c_2\,c_3  j\,I_j-8 c_2 c_3\, j\,\sum_{k\in S^+, k\neq j} I_k+12\,c_4 j^5\,I_j\\
&+24 c_4\,j^3\, \sum_{k\in S^+, k\neq j} k^2\,I_k +4 c_6  j^3\,I_j+4 c_6 j^3 \sum_{k\in S^+, k\neq j} I_k + 4 c_6\, j \sum_{k\in S^+, k\neq j} k^2 I_k+12\,c_7\,j\, I_j\\
&+24 c_7\,j\,\sum_{k\in S^+, k\neq j} I_k.
\end{aligned}
\end{equation}
Hence, in a small neighbourhood of the origin of the phase space $H_0^1(\mathbb{T}_x)$, the submanifold $\{ z=0 \}$ is foliated by invariant tori of amplitude $\xi$ and frequency vector $\omega(\xi):=(\omega_j(\xi))_{j\in S^+}$ as in \eqref{omeghino}.\\
We shall select from this set of tori the approximately invariant quasi-periodic solutions to be continued and we will use their \textit{unperturbed actions} $\xi$ as parameters. Moreover, we shall require that the frequencies of these tori vary in a one-to-one way with the actions $\xi$. Thanks to this fact, we could control the conditions that we shall impose on the frequencies $\omega$ through the amplitudes, and viceversa.\\ 
If we call $\vec{1}$ the vector in $\mathbb{R}^{\nu}$ with all components equal to $1$ and
\begin{equation}\label{vkDS}
D_S:=\mbox{diag}_{i=1,\dots, \nu} \{\overline{\jmath}_i\}, \qquad v_k:=D_S^k\, \vec{1}, \qquad U:=\vec{1}^T\,\vec{1}
\end{equation}
then we can write, in a compact form, the vector with components $\omega_j(I)$, with $j\in S^+$, in \eqref{omeghino}, as
\begin{equation}\label{FreqAmplMap}
\omega(\xi)=\overline{\omega}+\mathbb{A}\,\xi,
\end{equation}
where $\overline{\omega}$ is the vector of the linear frequencies (see \eqref{LinearFreq}) and
\begin{equation}\label{TwistMatrix}
\begin{aligned}
\mathbb{A}:&=(24 c_1^2-12 c_4) D_S^5\{ \mathrm{I}-2 D_{-2} U D_S^2 \}+(\frac{14}{3} c_2^2-4 c_6) D_S^3+(4 c_6-\frac{16}{3} c_2^2) \{D_S^3 U+D_S U D_S^2\}\\
&+12 (c_2 c_3-c_7) D_S+(24 c_7-16 c_2 c_3) D_S U-6 c_3^2 D_S^{-1}.
\end{aligned}
\end{equation}
The function of $\xi$ in \eqref{FreqAmplMap} is the \textit{frequency-amplitude map}, which describes, at the main order, how the tangential frequencies are shifted by the amplitudes $\xi$. \\
In order to work in a neighbourhood of the unperturbed torus $\{ I\equiv D_1\xi \}$ it is advantageous to introduce a set of coordinates $(\theta, y, z)\in\mathbb{T}^{\nu}\times \mathbb{R}^{\nu}\times H_S^{\perp}$ adapted to it, defined by
\begin{equation}\label{Action-AngleVariables}
\begin{cases}
u_j:=\sqrt{I_j}\,e^{\mathrm{i}\theta_j}\,e^{\mathrm{i}\,j\,x}, \,\,\quad\,\,  I_j:=\lvert j \rvert (\xi_j+ y_j), \qquad j\in S,\\[2mm]
u_j:=z_j, \qquad \qquad \qquad \qquad \quad \qquad\quad \qquad \qquad j\in S^c,
\end{cases}
\end{equation}
where  (recall $\overline{u}_j=u_{-j}$)
\begin{equation}
\xi_{-j}=\xi_j, \quad \xi_j>0, \quad y_{-j}=y_{j}, \quad \theta_{-j}=-\theta_j, \quad \theta_j\in\mathbb{T},\,\, y_j\in\mathbb{R}, \quad \forall j\in S.
\end{equation}
For the tangential sites $S^+:=\{ \overline{\jmath}_1,\dots, \overline{\jmath}_{\nu}\}$ we will also denote 
\[
\theta_{\overline{\jmath}_i}:=\theta_i,\quad y_{\overline{\jmath}_i}:=y_{i}, \quad \xi_{\overline{\jmath}_i}:=\xi_i, \quad \omega_{\overline{\jmath}_i}=\omega_i, \quad i=1,\dots, \nu.
\]
The symplectic $2$-form $\Omega$ in \eqref{SymplecticForm} becomes
\begin{equation}\label{NewSimplForm}
\mathcal{W}:=\sum_{i=1}^{\nu} d\theta_i\wedge d y_i+\frac{1}{2} \sum_{j\in S^c} \frac{1}{\mathrm{i} j}\,d z_j\wedge d z_{-j}=\left( \sum_{i=1}^{\nu} d \theta_i\wedge d y_i \right) \oplus \Omega_{S^{\perp}}=d \Lambda,
\end{equation}
where $\Omega_{S^{\perp}}$ denotes the restriction of $\Omega$ to $H_S^{\perp}$ and $\Lambda$ is the contact $1$-form on $\mathbb{T}^{\nu}\times\mathbb{R}^{\nu}\times H_S^{\perp}$ defined by $\Lambda_{(\theta, y, z)}\colon \mathbb{R}^{\nu}\times\mathbb{R}^{\nu}\times H_S^{\perp}\to \mathbb{R}$,
\begin{equation}\label{1Form}
\Lambda_{(\theta, y, z)}[\hat{\theta}, \hat{y}, \hat{z}]:=-y\cdot \hat{\theta}+\frac{1}{2} (\partial_x^{-1} z, \hat{z})_{L^2(\mathbb{T})}.
\end{equation}
Working in a neighbourhood of the origin of the phase space, it is convenient to rescale the unperturbed actions $\xi$ and the variables $\theta, y, z$ as
\begin{equation}\label{Rescaling}
\xi \mapsto \varepsilon^2 \xi, \quad y\mapsto \varepsilon^{2 b} y, \quad z \mapsto \varepsilon^b\,z.
\end{equation}
The symplectic form in \eqref{NewSimplForm} transforms into $\varepsilon^{2 b}\,\mathcal{W}$.
Hence the Hamiltonian system generated by $\mathcal{H}$ in \eqref{HamiltonianaNormalizzata} transforms into the new Hamiltonian system
\begin{equation}\label{HamiltonianaRiscalata}
\begin{cases}
\dot{\theta}=\partial_y H_{\varepsilon}(\theta, y, z),\\
\dot{y}=-\partial_{\theta} H_{\varepsilon}(\theta, y, z),\\
\dot{z}=\partial_x \nabla_z H_{\varepsilon}(\theta, y, z),
\end{cases}
\qquad H_{\varepsilon}:=\varepsilon^{-2 b}\,\mathcal{H}\circ A_{\varepsilon},
\end{equation}
where
\begin{equation}\label{Aepsilon}
A_{\varepsilon}(\theta, y, z):=\varepsilon\,v_{\varepsilon}(\theta, y)+\varepsilon^b z, \quad v_{\varepsilon}(\theta, y):=\sum_{j\in S} \sqrt{\lvert j \rvert}\,\sqrt{\xi_j+\varepsilon^{2(b-1)} y_j}\,e^{i \theta_j} e^{i j x}.
\end{equation}
We still denote by 
\[
X_{H_{\varepsilon}}=(\partial_y H_{\varepsilon}, -\partial_{\theta} H_{\varepsilon}, \partial_x \nabla_z H_{\varepsilon})
\]
the Hamiltonian vector field in the variables $(\theta, y, z)\in \mathbb{T}^{\nu}\times\mathbb{R}^{\nu}\times H_S^{\perp}$.
We now write explicitly the Hamiltonian defined in \eqref{HamiltonianaRiscalata}.
The quadratic Hamiltonian $H_2$ in \eqref{Hamiltonians} becomes
\begin{equation}
\varepsilon^{-2 b} H_2\circ A_{\varepsilon}=const+\sum_{j\in S^{+}} j^3\,y_j+\frac{1}{2} \int_{\mathbb{T}} z_x^2\,dx,
\end{equation}
and by \eqref{Hamiltonians}, \eqref{H4,2} and \eqref{H4risonante} we have (writing $v_{\varepsilon}:=v_{\varepsilon}(\theta, y)$)
\begin{equation}\label{HamiltonianaRiscalata}
\begin{aligned}
H_{\varepsilon}(\theta, y, z)&=e(\xi)+\alpha(\xi)\cdot y+\frac{1}{2}\int_{\mathbb{T}} z_x^2\,dx+\varepsilon \int_{\mathbb{T}}(3\, c_1 z_x^2\,(v_{\varepsilon})_x+3\,c_2 z_x^2\,v_{\varepsilon}+2\,c_2 (v_{\varepsilon})_x z_x z )  \,dx\\
&+\varepsilon^b \int_{\mathbb{T}} \left( c_1\,z_x^3+c_2\,z_x^2\,z+c_3\,z^3\,dx\right)\,dx+\frac{\varepsilon^{2\,b}}{2}\,\mathbb{M}\, y\cdot  y+\varepsilon^{2 b} \int_{\mathbb{T}} (c_4\,z_x^4+c_5\,z_x^3\,z\\
&+c_6\,z_x^2\,z^2+c_7\,z^4)\,dx+\varepsilon^2 R((v_{\varepsilon}(\theta, y))^2 z^2)+\varepsilon^{1+b} R(v_{\varepsilon}(\theta, y)\,z^3)+\varepsilon^3 R((v_{\varepsilon}(\theta, y))^3 z^2)\\
&+\varepsilon^{2+b} \sum_{q=3}^5 \varepsilon^{(q-3)(b-1)} R((v_{\varepsilon}(\theta, y))^{5-q} z^q)+\varepsilon^{-2 b} \mathcal{H}_{\geq 6} (\varepsilon v_{\varepsilon}(\theta, y)+\varepsilon^b z)
\end{aligned}
\end{equation}
where the function $e(\xi)$ is a constant and
\begin{equation}\label{Frequency-AmplitudeMAP}
\alpha(\xi)=\overline{\omega}+\varepsilon^2\,\mathbb{M}\,\xi \,, \qquad \mathbb{M}:=\mathbb{A}\,D_S
\end{equation}
is the frequency amplitude-map after the change of coordinates in \eqref{Action-AngleVariables} and the rescaling in \eqref{Rescaling}. Usually $\mathbb{M}$ is called the \textit{twist matrix} and we note that is symmetric.\\
We write the Hamiltonian in \eqref{HamiltonianaRiscalata}, eliminating the constant $e(\xi)$ which is irrelevant for the dynamics, as
\begin{equation}\label{Hepsilon}
\begin{aligned}
&H_{\varepsilon}=\mathcal{N}+P, \qquad\mathcal{N}(\theta, y, z)=\alpha(\xi)\cdot y+\frac{1}{2} (N(\theta) z, z)_{L^2(\mathbb{T})},\\
&\frac{1}{2}(N(\theta) z, z)_{L^2(\mathbb{T})}:=\frac{1}{2}((\partial_z \nabla H_{\varepsilon}) (\theta, 0, 0)[z], z)_{L^2(\mathbb{T})}=\frac{1}{2} \int_{\mathbb{T}} z_x^2\,dx+\\
&+\varepsilon\int_{\mathbb{T}} c_1z_x^2 \,(v_{\varepsilon})_x(\theta, 0)\,dx+\varepsilon\int_{\mathbb{T}} c_2\,z_x^2\,v_{\varepsilon}(\theta, 0) \,dx+2\,\varepsilon\,c_2\,\int_{\mathbb{T}} z\,z_x\,(v_{\varepsilon})_x(\theta, 0)\,dx+ \dots
\end{aligned}
\end{equation}
where $\mathcal{N}$ describes the linear dynamics, and $P:=H_{\varepsilon}-\mathcal{N}$ collects the nonlinear perturbative effects.\\

As we said before, we require that the map \eqref{Frequency-AmplitudeMAP} is a diffeomorphism. This function is affine, thus its invertibility is equivalent to the nondegenerancy (or \textit{twist}) condition 
\begin{equation}\label{NonDegCond}
\det\mathbb{M}:=\det (D_S)\,\det\left(\dfrac{\partial^2}{\partial I_j \,I_k} \tilde{h}(I)\right)_{j, k\in \{1, \dots, \nu\}} \,\det(D_S) \neq 0.
\end{equation}
\begin{remark}
The inequality \eqref{NonDegCond} is equivalent to the classical Kolmogorov condition that requires the invertibility of the Hessian of the Hamiltonian $\tilde{h}$ in \eqref{HamBif}. The presence of the diagonal matrix $D_S$ in \eqref{NonDegCond} is due to the symplectic form \eqref{SymplecticForm} and the choice of the action-angle variables \eqref{Action-AngleVariables}.
\end{remark}
In the following lemma we prove that the condition \eqref{NonDegCond} is satisfied for non-resonant coefficients and a generic choice of the tangential sites (see Definition \ref{Generic}).

\begin{lem}\label{TwistLemma}
If the coefficients $c_1, \dots, c_7$ are non-resonant, for a generic choice of the tangential sites $\overline{\jmath}_1, \dots, \overline{\jmath}_{\nu}$ (see Definition \ref{Generic}) the condition \eqref{NonDegCond} is satisfied.
\end{lem}
\begin{proof}
We write $\mathbb{M}=D_S^{-1}\,\mathbb{B}\,D_S$, with
\begin{equation}
\begin{aligned}
\mathbb{B}:&=(24 c_1^2-12 c_4) D_S^6\{ \mathrm{I}-2 D_{-2} U D_S^2 \}+(\frac{14}{3} c_2^2-4 c_6) D_S^4\\
&+(4 c_6-\frac{16}{3} c_2^2) \{D_S^4 U+D_S^2 U D_S^2\}-6 c_3^2 \mathrm{I}+12 (c_2 c_3-c_7) D_S^2\\
&+(24 c_7-16 c_2 c_3) D_S^2 U,
\end{aligned}
\end{equation}
where $\mathrm{I}$ is the identity $\nu\times \nu$ matrix.
The determinant of $\mathbb{B}$ is a polynomial in the variables $(\overline{\jmath}_1, \dots , \overline{\jmath}_{\nu})$ and, if $c_3\neq 0$, it is not trivial, namely it is not identically zero. Indeed, the monomial of minimal degree of this polynomial originates from the matrix $6\,c_3^2\,\mathrm{I}$, that is invertible, and so it cannot be naught.\\ 
Similarly, if $c_3=0$ and $2\, c_1^2-c_4\neq 0$ then the monomial of maximal degree, i.e. six, is not zero, beacuse $(24\,c_1^2-12\,c_4)\,D_S^6 \{\mathrm{I}-2\,D_S^{-2}\,U\,D_S^2\}$ is invertible.\\
If $c_3=2\,c_1^2-c_4=0$ and $c_7\neq 0$ then the monomial of minimal degree, i.e. two, is $12\,c_7\,D_S^2\,(2\,U-\mathrm{I})$, that is invertible, indeed 
\[
\left(2\,U-\mathrm{I} \right)^{-1}=\mathrm{I}-\frac{2}{2\,\nu+1}\,U,
\]
where $2\,\nu+1\neq 0$, because $\nu\in\mathbb{N}$.
If $c_3=2\,c_1^2-c_4=c_7=0$ then
\begin{align*}
\mathbb{B}&=D_S^4 \left\{(\frac{14}{3} c_2^2-4 c_6) \mathrm{I}+(4 c_6-\frac{16}{3} c_2^2) \{ U+D_{-2} U D_S^2\}\right\}
\end{align*}
The matrix $U+D_S^{-2} U D_S^2$ has rank $2$ and its image is spanned by the vectors $\vec{1}:=(1, \dots, 1)$ and $v_{-2}$ 
The eigenvalues of this matrix, different from zero, are
\begin{equation}
\lambda_1:=\nu+\sqrt{\left(\sum_{i=1}^{\nu} \overline{\jmath}_i^2\right)\left(\sum_{i=1}^{\nu} \overline{\jmath}_i^{-2}\right)}, \quad \lambda_2:=\nu-\sqrt{\left(\sum_{i=1}^{\nu} \overline{\jmath}_i^2\right)\left(\sum_{i=1}^{\nu} \overline{\jmath}_i^{-2}\right)}.
\end{equation}
Then, if $7 c_2^2-6\,c_6\neq 0$ and $
\alpha:=(8\,c_2^2-6\,c_6)/(7\,c_2^2-6\,c_6)$,
we require that
\begin{equation}\label{ultimeCondizioni}
\begin{cases}
1-\alpha\,\lambda_1\neq 0,\\
1-\alpha\,\lambda_2\neq 0.
\end{cases}
\end{equation}
The conditions \eqref{ultimeCondizioni} are satisfied for every choice of the tangential sites if $4\,c_2^2=3\,c_6$; otherwise, it is satisfied by generic integer vectors $(\overline{\jmath}_i)_{i=1}^{\nu}$.
\end{proof}
\section{The nonliner functional setting}
We look for an embedded invariant torus
\begin{equation}\label{i}
i \colon \mathbb{T}^{\nu}\to \mathbb{T}^{\nu}\times \mathbb{R}^{\nu}\times H_S^{\perp}, \quad \varphi \mapsto i(\varphi):=(\theta(\varphi), y(\varphi), z(\varphi))
\end{equation}
of the Hamiltonian vector field $X_{H_{\varepsilon}}$ filled by quasi-periodic solutions with diophantine frequency $\omega\in\mathbb{R}^{\nu}$, that we consider as independent parameters. We require that $\omega$ belongs to the set
\begin{equation}\label{OmegaEpsilon}
\Omega_{\varepsilon}:=\{ \alpha(\xi) : \xi\in [1, 2]^{\nu}\},
\end{equation}
where $\alpha$ is the function defined in \eqref{Frequency-AmplitudeMAP} and, by Lemma \ref{NonDegCond}, it is a diffeomorphism for a generic choice of the tangential sites.
\begin{remark}
We could consider any compact subset of $\{v\in\mathbb{R}^{\nu} : v_i> 0, \forall i=1,\dots, \nu\}$ instead of the set $[1, 2]^{\nu}$ in the definition \eqref{OmegaEpsilon}.
\end{remark}
Since any $\omega\in\Omega_{\varepsilon}$ is $\varepsilon^2$-close to the integer vector $\overline{\omega}:=(\overline{\jmath}_1^3, \dots, \overline{\jmath}_{\nu}^3)\in\mathbb{N}^{\nu}$, we require that the constant $\gamma$ in the diophantine inequality
\begin{equation}\label{0diMelnikov}
\lvert \omega\cdot l \rvert\geq \gamma\,\langle l \rangle^{-\tau}, \quad \forall l \in\mathbb{Z}^{\nu}\setminus\{0\}
\end{equation} 
satisfies
\begin{equation}\label{gamma}
\gamma=\varepsilon^{2+a}, \quad \mbox{for\,\,some}\,\,a>0.
\end{equation}
Note that the definition of $\gamma$ in \eqref{gamma} is slightly stronger than the minimal condition, namely $\gamma\le c\,\varepsilon^2$, with $c>0$ small enough. In addition to \eqref{0diMelnikov} we shall also require that $\omega$ satisfies the first and the second order Melnikov non-resonance conditions.
We fix the amplitude $\xi$ as a function of $\omega$ and $\varepsilon$, as
\begin{equation}
\xi:=\varepsilon^{-2}\,\mathbb{M}^{-1} [\omega-\overline{\omega}],
\end{equation}
so that $\alpha(\xi)=\omega$ (see \eqref{Frequency-AmplitudeMAP}).
Consequently, $H_{\varepsilon}$ in \eqref{Hepsilon} becomes a $(\omega, \varepsilon)$-parameter family of Hamiltonians which possess an invariant torus at the origin with frequency vector close to $\omega$.\\
Now we look for an embedded invariant torus  of the modified Hamiltonian vector field $X_{H_{\varepsilon, \zeta}}=X_{H_{\varepsilon}}+(0, \zeta, 0), \zeta\in\mathbb{R}^{\nu}$, which is generated by the Hamiltonian
\begin{equation}\label{HepsilonZeta}
H_{\varepsilon, \zeta}(\theta, y, z):=H_{\varepsilon}(\theta, y, z)+\zeta\cdot\theta, \quad \zeta\in\mathbb{R}^{\nu}.
\end{equation}
We introduce $\zeta$ in order to control the average in the $y$-component of the linearized equations \eqref{D} (see \eqref{ValoreperZeta}). However, the vector $\zeta$ has no dynamical consequences. Indeed it turns out that an invariant torus for the Hamiltonian vector field $X_{H_{\varepsilon, \zeta}}$ is actually invariant for $X_{H_{\varepsilon}}$ itself.\\
Thus, we look for zeros of the nonlinear operator
\begin{align}\label{NonlinearFunctional}
\mathcal{F}(i, \zeta)&:=\mathcal{F}(i, \zeta, \omega, \varepsilon):=\mathcal{D}_{\omega} i(\varphi)-X_{\mathcal{N}}(i(\varphi))-X_P(i(\varphi))+(0, \zeta, 0)\\[3mm]
&:=\begin{pmatrix}
\mathcal{D}_{\omega} \theta(\varphi)-\partial_y H_{\varepsilon}(i(\varphi))\\
\mathcal{D}_{\omega} y(\varphi)+\partial_{\theta} H_{\varepsilon}(i(\varphi))+\zeta\\
\mathcal{D}_{\omega} z(\varphi)-\partial_x \nabla_z H_{\varepsilon}(i(\varphi))
\end{pmatrix} =\begin{pmatrix}
\mathcal{D}_{\omega} \Theta(\varphi)-\partial_y P(i(\varphi))\\
\mathcal{D}_{\omega} y(\varphi)+\frac{1}{2} \partial_{\theta} (N(\theta(\varphi))z(\varphi))_{L^2(\mathbb{T})}+\partial_{\theta} P(i(\varphi))+\zeta\\
\mathcal{D}_{\omega} z(\varphi)-\partial_x N(\theta(\varphi))\,z(\varphi)-\partial_x\nabla_z P(i(\varphi))
\end{pmatrix}    \notag
\end{align}
where $\Theta(\varphi):=\theta(\varphi)-\varphi$ is $(2\pi)^{\nu}$-periodic and we use the short notation
\begin{equation}\label{Domega}
\mathcal{D}_{\omega}:=\omega\cdot \partial_{\varphi}.
\end{equation}
The Sobolev norm of the periodic component of the embedded torus
\begin{equation}\label{frakkiI}
\mathfrak{I}(\varphi):=i(\varphi)-(\varphi, 0, 0):=(\Theta(\varphi), y(\varphi), z(\varphi)), 
\end{equation}
is
\begin{equation}
\lVert \mathfrak{I} \rVert_{s}:=\lVert \Theta \rVert_{H_{\varphi}^s}+\lVert y \rVert_{H^s_{\varphi}}+\lVert z \rVert_s
\end{equation}
where $\lVert z \rVert_s:=\lVert z \rVert_{H^s_{\varphi, x}}$ is defined in \eqref{SobolevNormtx}.\\
We link the rescaling of the domain of the variables \eqref{Rescaling} with the diophantine constant $\gamma=\varepsilon^{2+a}$ by choosing
\begin{equation}
\gamma=\varepsilon^{2+a}=\varepsilon^{2\,b}, \quad b:=1+(a/2).
\end{equation}
Other choices are possible (see Remark $5.2$ in \cite{MKdV}).
\begin{teor}\label{IlTeorema}
If $c_1, \dots, c_7$ are non-resonant and conditions $(\mathtt{C}1)$-$(\mathtt{C}2)$ hold, then for a generic choice of the tangential sites $S$, satisfying the assumption $(\mathtt{S})$, for all $\varepsilon\in (0, \varepsilon_0)$, where $\varepsilon_0$ is a positive constant small enough, there exist a constant $C>0$ and a Cantor-like set $\mathcal{C}_{\varepsilon}\subseteq\Omega_{\varepsilon}$ (see \eqref{OmegaEpsilon}), with asymptotically full measure as $\varepsilon\to 0$, namely
\begin{equation}\label{frazionemisure}
\lim_{\varepsilon \to 0} \dfrac{\lvert \mathcal{C}_{\varepsilon} \rvert}{\lvert \Omega_{\varepsilon} \rvert}=1,
\end{equation}
such that, for all $\omega\in\mathcal{C}_{\varepsilon}$, there exists a solution $i_{\infty}(\varphi):=i_{\infty}(\omega, \varepsilon)(\varphi)$ of the equation $\mathcal{F}(i_{\infty}, 0, \omega, \varepsilon)=0$. Hence the embedded torus $\varphi\mapsto i_{\infty}(\varphi)$ is invariant for the Hamiltonian vector field $X_{H_{\varepsilon}}$, and it is filled by quasi-periodic solutions with frequency $\omega$. The torus $i_{\infty}$ satisfies
\begin{equation}
\lVert i_{\infty}(\varphi)-(\varphi, 0, 0) \rVert_{s_0+\mu}^{Lip(\gamma)}\le C\,\varepsilon^{6-2 b}\,\gamma^{-1}
\end{equation}
for some $\mu:=\mu(\nu)>0$. Moreover the torus $i_{\infty}$ is linearly stable.
\end{teor}
Theorem \ref{IlTeorema} is proved in Sections $6-9$. It implies Theorem \ref{F.Giuliani} where the $\xi_j$ in \eqref{SoluzioneEsplicita} are the components of the vector $\mathbb{M}^{-1}[\omega-\overline{\omega}]$.\\

Now we give tame estimates for the composition operator induced by the Hamiltonian vector fields $X_{\mathcal{N}}$ and $X_{P}$ in \eqref{NonlinearFunctional}.\\
Since the functions $y\to \sqrt{\xi+\varepsilon^{2(b-1)} y}, \theta\to e^{\mathrm{i}\,\theta}$ are analytic for $\varepsilon$ small enough and $\lvert y \rvert\le C$, the composition lemma \ref{lemmacomp} implies that, for all $\Theta, y\in H^s(\T^{\nu}, \mathbb{R}^{\nu})$ with $\lVert \Theta \rVert_{s_0}, \lVert y \rVert_{s_0}\le 1$, one has the tame estimate
\begin{equation}
\lVert v_{\varepsilon}(\theta(\varphi), y(\varphi)) \rVert_s\le_s 1+\lVert \Theta \rVert_s+\lVert y \rVert_s.
\end{equation}
Hence the map $A_{\varepsilon}$ in \eqref{Aepsilon} satisfies, for all $\lVert \mathfrak{I} \rVert_{s_0}^{Lip(\gamma)}\le 1$
\begin{equation}
\lVert A_{\varepsilon}(\theta(\varphi), y(\varphi), z(\varphi)) \rVert_s^{Lip(\gamma)}\le_s \varepsilon (1+\lVert \mathfrak{I} \rVert_s^{Lip(\gamma)}).
\end{equation}
In the following lemma we collect tame estimates for the Hamiltonian vector fields $X_{\mathcal{N}}, X_{P}, X_{H_{\varepsilon}}$, see \eqref{Hepsilon}.
\begin{lem}\label{TameEstimatesforVectorfields}
Let $\mathfrak{I}(\varphi)$ in \eqref{frakkiI} satisfy $\lVert \mathfrak{I} \rVert_{s_0+3}^{Lip(\gamma)}\le C\,\varepsilon^{6-2 b}\gamma^{-1}$. Then
\begin{align}
&\lVert \partial_y P(i) \rVert_s^{Lip(\gamma)}\le_s \varepsilon^4+\varepsilon^{2 b} \lVert \mathfrak{I} \rVert_{s+3}^{Lip(\gamma)}, \qquad  \lVert \partial_{\theta} P(i) \rVert_s^{Lip(\gamma)}\le_s \varepsilon^{6-2 b}(1+\lVert \mathfrak{I} \rVert_{s+3}^{Lip(\gamma)}),\\
&\lVert \nabla_z P(i) \rVert_s^{Lip(\gamma)}\le_s \varepsilon^{5-b}+\varepsilon^{6-b}\gamma^{-1} \lVert \mathfrak{I} \rVert_{s+3}^{Lip(\gamma)}, \qquad\lVert X_P (i) \rVert_s^{Lip(\gamma)}\le_s \varepsilon^{6-2 b}+\varepsilon^{2 b} \lVert \mathfrak{I} \rVert_{s+3}^{Lip(\gamma)}, \label{EstimatesVecField}\\
&\lVert \partial_{\theta}\partial_y P(i) \rVert_s^{Lip(\gamma)}\le_s \varepsilon^4+\varepsilon^5\gamma^{-1} \lVert \mathfrak{I} \rVert_{s+3}^{Lip(\gamma)}, \qquad \lVert \partial_y\nabla_z P(i) \rVert_s^{Lip(\gamma)}\le_s \varepsilon^{b+3}+\varepsilon^{2 b-1}\lVert \mathfrak{I} \rVert_{s+3}^{Lip(\gamma)},\\
&\lVert \partial_{y y} P(i)-\frac{\varepsilon^{2 b}}{2} \mathbb{M} \rVert_s^{Lip(\gamma)}\le_s \varepsilon^{2+2 b}+\varepsilon^{2 b+3}\gamma^{-1} \lVert \mathfrak{I} \rVert_{s+2}^{Lip(\gamma)}
\end{align}
and for all $\hat{\imath}:=(\hat{\Theta}, \hat{y}, \hat{z})$,
\begin{align}
&\lVert \partial_y d_i X_P(i)[\hat{\imath}] \rVert_s^{Lip(\gamma)}\le_s \varepsilon^{2 b-1} (\lVert \hat{\imath} \rVert_{s+3}^{Lip(\gamma)}+\lVert \mathfrak{I} \rVert_{s+3}^{Lip(\gamma)}\lVert \hat{\imath} \rVert_{s_0+3}),\\
&\lVert d_i X_{H_{\varepsilon}}(i)[\hat{\imath}]+(0, 0, \partial_{xxx} \hat{z})\rVert_s^{Lip(\gamma)}\le_s \varepsilon (\lVert \hat{\imath} \rVert_{s+3}^{Lip(\gamma)}+\lVert \mathfrak{I} \rVert_{s+3}^{Lip(\gamma)}\lVert \hat{\imath} \rVert_{s_0+3}),\\
&\lVert d_i^2 X_{H_{\varepsilon}} (i) [\hat{\imath}, \hat{\imath}]\rVert_s^{Lip(\gamma)}\le_s \varepsilon (\lVert \hat{\imath} \rVert_{s+3}^{Lip(\gamma)}\lVert \hat{\imath} \rVert_{s_0+3}^{Lip(\gamma)}+\lVert \mathfrak{I} \rVert_{s+3}^{Lip(\gamma)}(\lVert \hat{\imath} \rVert_{s_0+3})^2).
\end{align}
\end{lem}

In the sequel we will use that, by the diophantine condition \eqref{0diMelnikov}, the operator $\mathcal{D}_{\omega}^{-1}$ (see \eqref{Domega}) is defined for all functions $u$ with zero $\varphi$-average, and satisfies
\begin{equation}\label{stimaDomega}
\lVert \mathcal{D}_{\omega}^{-1} u \rVert_s\le_s \gamma^{-1}\,\lVert u \rVert_{s+\tau}, \quad \lVert \mathcal{D}_{\omega}^{-1} u \rVert_s^{Lip(\gamma)}\le_s \gamma^{-1} \lVert u \rVert^{Lip(\gamma)}_{s+2\tau+1}.
\end{equation}


\section{Approximate inverse}
We will apply a Nash-Moser iterative scheme in order to find a zero of the functional $\mathcal{F}(i, \zeta)$ defined in \eqref{NonlinearFunctional}. In particular, we shall construct a sequence of approximate solutions of 
\begin{equation}\label{EquazioneFunzionale}
{F}(i, \zeta)=0
\end{equation}
that converges to a solution in some Sobolev norm. In order to define this sequence we need to solve some linearized equations and this is the main difficulty for implementing the Nash-Moser algorithm. \\
Zehnder noted in \cite{Z} that it is sufficient to invert these equations only approximately to get a scheme with still quadratic speed of convergence. We refer to \cite{Z} for the precise notion of \textit{approximate right inverse}, whose main feature is to be an \textit{exact right inverse} when the equation is linearized at an exact solution.
Hence, our aim is to construct an approximate right inverse of the linearized operator 
\begin{equation}\label{LinearizedOp}
d_{i, \zeta} \mathcal{F}(i_0, \zeta_0) [\hat{\imath}, \hat{\zeta}]=\mathcal{D}_{\omega} \hat{\imath}-d_i X_{H_{\varepsilon}}(i_0(\varphi))[\hat{\imath}]+(0, \hat{\zeta}, 0)
\end{equation} 
at any approximate solution $i_0$ of the equation \eqref{EquazioneFunzionale}, and to verify that satisfies some tame estimates.\\
Note that $d_{i, \zeta} \mathcal{F}(i_0, \zeta_0)=d_{i, \zeta} \mathcal{F}(i_0)$ is independent of $\zeta_0$ (see \eqref{NonlinearFunctional}).\\
We will implement the general strategy in \cite{BertiBolle}, \cite{NLW} which reduces the search of an approximate right inverse of \eqref{LinearizedOp} to the search of an approximate inverse on the normal directions only.\\

It is well known that an invariant torus $i_0$ with diophantine flow is isotropic (see e.g.\cite{BertiBolle}), namely the pull-back $1$-form $i_0^*\Lambda$ is closed, where $\Lambda$ is the Liouville $1$-form in \eqref{1Form}. This is tantamount to say that the $2$-form $\mathcal{W}$ in  \eqref{NewSimplForm} vanishes on the torus $i_0(\T^{\nu})$, because $i_0^* \mathcal{W}=i_0^* d\Lambda=d\,i_0^* \Lambda$. For an ``approximately invariant'' embedded torus $i_0$ the $1$-form $i_0^*\Lambda$ is only ``approximately closed''. In order to make this statement quantitative we consider
\begin{equation}\label{pullBackLambda}
i_0^*\Lambda=\sum_{k=1}^{\nu} a_k(\varphi)\,d\varphi_k, \quad a_k(\varphi):=-([\partial_{\varphi} \theta_0(\varphi)]^T y_0(\varphi))_k+\frac{1}{2} ( \partial_{\varphi_k} z_0(\varphi), \partial_x^{-1} z_0(\varphi))_{L^2(\T)}
\end{equation}
and we quantify how small is
\begin{equation}\label{pullBackW}
i_0^* \mathcal{W}=d\,i_0^* \Lambda=\sum_{1\le k<j\le \nu} A_{k\,j}(\varphi)\,d\varphi_k\wedge d \varphi_j, \quad A_{k\,j}(\varphi):=\partial_{\varphi_k} a_j(\varphi)-\partial_{\varphi_j} a_k(\varphi).
\end{equation}
In order to get estimates for an approximate inverse we need to take in account the size of the ``error'' function
\begin{equation}
Z(\varphi):=(Z_1, Z_2, Z_3)(\varphi):=\mathcal{F}(i_0, \zeta_0)(\varphi)=\omega\cdot \partial_{\varphi} i_0 (\varphi)-X_{H_{\varepsilon, \zeta_0}}(i_0(\varphi)),
\end{equation}
which gives a measure of how $i_0$ is near to be an exact solution.\\
Along this section we will always assume the following hypotesis (which will be proved at each step of the Nash-Moser iteration):
\begin{itemize}
\item \textbf{Assumption}. The map $\omega \mapsto i_0(\omega)$ is a Lipschitz function defined on some subset $\Omega_0\subseteq \Omega_{\varepsilon}$, where $\Omega_{\varepsilon}$ is defined in \eqref{OmegaEpsilon}, and, for some $\mu:=\mu(\tau, \nu)>0$,
\begin{equation}\label{Assumption}
\lVert \mathfrak{I}_0 \rVert_{s_0+\mu}^{Lip(\gamma)}\le \varepsilon^{6-2b} \gamma^{-1}, \quad \lVert Z \rVert_{s_0+\mu}^{Lip(\gamma)}\le \varepsilon^{6-2 b}, \quad \gamma=\varepsilon^{2+a}, \quad a\in (0, 1/6),
\end{equation}
where $\mathfrak{I}_0(\varphi):=i_0(\varphi)-(\varphi, 0, 0)$.
\end{itemize}
The next lemma proves that if $i_0$ is a solution of the equation \eqref{EquazioneFunzionale}, then the parameter $\zeta$ has to be naught, hence the embedded torus $i_0$ supports a quasi-periodic solution of the ``original'' system with Hamiltonian $H_{\varepsilon}$. 
\begin{lem}{(Lemma $6.1$ in \cite{KdVAut})}\label{Lemma6.1}
We have
\[
\lvert \zeta_0 \rvert^{Lip(\gamma)}\le C \lVert Z \rVert_{s_0}^{Lip(\gamma)}.
\]
In particular, if $\mathcal{F}(i_0, \zeta_0)=0$ then $\zeta_0=0$ and the torus $i_0(\varphi)$ is invariant for the vector field $X_{H_{\varepsilon}}$.
\end{lem}
Now we estimate the size of $i_0^*\mathcal{W}$ in terms of the error function $Z$.\\
By \eqref{pullBackLambda}, \eqref{pullBackW} we get
\[
\lVert A_{k\,j} \rVert_s^{Lip(\gamma)}\le_s \lVert \mathfrak{I}_0 \rVert^{Lip(\gamma)}_{s+2}.
\]
Moreover, we have the following bound.
\begin{lem}{(Lemma $6.2$ in \cite{KdVAut})}
The coefficients $A_{k\,j}(\varphi)$ in \eqref{pullBackW} satisfy
\begin{equation}
\lVert A_{k\,j} \rVert_s^{Lip(\gamma)}\le_s \gamma^{-1} (\lVert Z \rVert_{s+2\tau+2}^{Lip(\gamma)}+\lVert Z \rVert_{s_0+1}^{Lip(\gamma)}\lVert \mathfrak{I}_0\rVert_{s+2\tau+2}^{Lip(\gamma)}).
\end{equation}
\end{lem}
As in \cite{BertiBolle}, the idea is to analyze the operator linearized at an isotropic embedded torus $i_{\delta}$, because the isotropy of the torus allows to construct a symplectic set of coordinates around it for which the linear tangential dynamic and the normal one are decoupled. Thus, the linear system becomes ``triangular'' and the hard part is to solve the equation in the normal directions (see Section $7$).\\
Now we see that we can slightly modify $i_0$ (indeed, it is sufficient to move the $y$-component only) to obtain an isotropic torus $i_{\delta}$, that is an approximate solution as well as $i_0$. At the end of this section, we will prove that we are able to construct an approximate right inverse of \eqref{LinearizedOp} starting from an approximate inverse of $d_{i, \zeta}\mathcal{F}(i_{\delta}, \zeta_0)[\hat{\imath}, \hat{\zeta}]$.\\
 
In the paper we denote equivalently the differential $\partial_i$ or $d_i$. We use the notation $\Delta_{\varphi}:=\sum_{k=1}^{\nu}\partial^2_{\varphi_{k}}$ and we denote by $\sigma:=\sigma(\nu, \tau)$ possibly different (larger) ``loss of derivatives'' constants.\\
 
\begin{lem}{(\textbf{Isotropic torus})}{(Lemma $6.3$ in \cite{KdVAut})}
The torus $i_{\delta}=(\theta_0(\varphi), y_{\delta}(\varphi), z_0(\varphi))$ defined by 
\begin{equation}
y_{\delta}:=y_0+[\partial_{\varphi}\theta_0(\varphi)]^{-T}\rho(\varphi), \quad \rho_j(\varphi):=\Delta^{-1}_{\varphi} \sum_{k=1}^{\nu} \partial_{\varphi_j} A_{k\,j}(\varphi),
\end{equation}
is isotropic. If \eqref{Assumption} holds, then, for some $\sigma:=\sigma(\nu, \tau)$,
\begin{align}
&\lVert y_{\delta}-y_0 \rVert_s^{Lip(\gamma)}\le_s \gamma^{-1} (\lVert Z \rVert_{s+\sigma}^{Lip(\gamma)}\lVert \mathfrak{I}_0 \rVert_{s_0+\sigma}^{Lip(\gamma)}+\lVert Z \rVert_{s_0+\sigma}^{Lip(\gamma)}\lVert \mathfrak{I}_0 \rVert_{s+\sigma}^{Lip(\gamma)}),\\
&\lVert \mathcal{F}(i_{\delta}, \zeta_0) \rVert_s^{Lip(\gamma)}\le_s \lVert Z \rVert_{s+\sigma}^{Lip(\gamma)}+\lVert Z \rVert_{s_0+\sigma}^{Lip(\gamma)}\lVert \mathfrak{I}_0 \rVert_{s+\sigma}^{Lip(\gamma)}, \label{6.10}\\
& \lVert \partial_i i_{\delta} [\hat{\imath}] \rVert_s\le_s \lVert \hat{\imath} \rVert_{s}+\lVert \mathfrak{I}_0 \rVert_{s+\sigma}\lVert \hat{\imath} \rVert_{s}.
\end{align}
\end{lem}
We introduce a set of symplectic coordinates adapted to the isotropic torus $i_{\delta}$. We consider the map $G_{\delta}\colon (\Psi, \eta, w) \to (\theta, y, z)$ of the phase space $\T^{\nu}\times\mathbb{R}^{\nu}\times H_S^{\perp}$ defined by
\begin{equation}\label{Gdelta}
\begin{pmatrix}
\theta \\ y \\ z
\end{pmatrix}
:=G_{\delta}\begin{pmatrix}
\psi\\ \eta \\ w
\end{pmatrix}
:=\begin{pmatrix}
\theta_0(\psi) \\
y_{\delta}(\psi)+[\partial_{\psi}\theta_0(\psi)]^{-T}\eta+[(\partial_{\theta} \tilde{z}_0)(\theta_0(\psi))]^T\partial_x^{-1}w\\
z_0(\psi)+w
\end{pmatrix}
\end{equation}
where $\tilde{z}_0:=z_0 (\theta_0^{-1} (\theta))$ (indeed $\theta_0\colon \mathbb{T}^{\nu}\to \T^{\nu}$ is a diffeomorphism, because $\theta_0(\varphi)-\varphi$ is small). It is proved in \cite{BertiBolle} (Lemma $6.3$) that $G_{\delta}$ in \eqref{Gdelta} is symplectic, using that the torus $i_{\delta}$ is isotropic. In the new coordinates, $i_{\delta}$ is at the origin, i.e. $(\psi, \eta, w)=(\psi, 0, 0)$. The transformed Hamiltonian $K:=K(\psi, \eta, w, \zeta_0)$ is (recall \eqref{HepsilonZeta})
\begin{equation}\label{K}
\begin{aligned}
K:=H_{\varepsilon, \zeta_0}\circ G_{\delta}&=\theta_0(\psi)\cdot \zeta_0+K_{00}(\psi)+K_{10}(\psi)\cdot\eta+(K_{01}(\psi), w)_{L^2(\T)}+\frac{1}{2} K_{20}(\psi) \eta\cdot\eta+\\
&+(K_{11}(\psi) \eta, w)_{L^2(\T)}+\frac{1}{2} (K_{02}(\psi) w, w)_{L^2(\T)}+K_{\geq 3}(\psi, \eta, w)
\end{aligned}
\end{equation}
where $K_{\geq 3}$ collects the terms at least cubic in the variables $(\eta, w)$. At any fixed $\psi$, the Taylor coefficient $K_{00}(\psi)\in\mathbb{R}, K_{10}(\psi)\in\mathbb{R}^{\nu}, K_{01}(\psi)\in H_S^{\perp}, K_{20}(\psi)$ is a $\nu\times\nu$ real matrix, $K_{02}(\psi)$ is a linear self-adjoint operator of $H_S^{\perp}$ and $K_{11}(\psi)\colon\mathbb{R}^{\nu} \to H_S^{\perp}$.\\
Note that the above Taylor coefficients do not depend on the parameter $\zeta_0$.\\
The Hamilton equations associated to \eqref{K} are
\begin{equation}\label{VectorFieldK}
\begin{cases}
\dot{\psi}=K_{10}(\psi)+K_{20}(\psi)\eta+K_{11}^T(\psi) w+\partial_{\eta} K_{\geq 3}(\psi, \eta, w)\\
\begin{aligned}
\dot{\eta}=&-[\partial_{\psi} \theta_0(\psi)]^T \zeta_0-\partial_{\psi} K_{00}(\psi)-[\partial_{\psi} K_{10}(\psi)]^T \eta-[\partial_{\psi} K_{01}(\psi)]^T w-\\
&-\partial_{\psi} \left(\frac{1}{2} K_{20}(\psi)\eta\cdot \eta+(K_{11}(\psi) \eta, w)_{L^2(\T)}+\frac{1}{2} (K_{02}(\psi) w, w)_{L^2(\T)}+K_{\geq 3}(\psi, \eta, w)\right)
\end{aligned}
\\
\dot{w}=\partial_x (K_{01}(\psi)+K_{11}(\psi) \eta+K_{02}(\psi) w+\nabla_w K_{\geq 3}(\psi, \eta, w))
\end{cases}
\end{equation}
where $[\partial_{\psi} K_{10}(\psi)]^T$ is the $\nu\times \nu$ transposed matrix and $[\partial_{\psi} K_{01}(\psi)]^T, K_{11}^T(\psi)\colon H_S^{\perp} \to \mathbb{R}^{\nu}$ are defined by the duality relation 
\[
(\partial_{\psi} K_{01} (\psi)[\hat{\psi}], w)_{L^2(\T)}=\hat{\psi}\cdot [\partial_{\psi} K_{01}(\psi)]^T w, \quad \forall \hat{\psi}\in\mathbb{R}^{\nu}, w\in H_S^{\perp},
\]
and similarly for $K_{11}$.
Explicitly, for all $w\in H_S^{\perp}$, and denoting $\underline{e}_k$ the $k$-th versor of $\mathbb{R}^{\nu}$,
\begin{equation}
K_{11}^T(\psi) w=\sum_{k=1}^{\nu} (K_{11}^T(\psi) w\cdot \underline{e}_k)\,\underline{e}_k=\sum_{k=1}^{\nu} (w, K_{11}(\psi) \underline{e}_k)_{L^2(\T)} \underline{e}_k\in\mathbb{R}^{\nu}.
\end{equation}
In the next lemma we estimate the coefficients $K_{00}, K_{10}, K_{01}$ in the Taylor expansion \eqref{K}. The term $K_{10}$ describes how the tangential frequencies vary with respect to $\omega$.
Note that on an exact solution $(i_0, \zeta_0)$ we have $K_{00}(\psi)=const, K_{10}=\omega$ and $K_{01}=0$.
\begin{lem}{(Lemma $6.4$ in \cite{KdVAut})}\label{Lemma6.4}
Assume \eqref{Assumption}. Then there is $\sigma:=\sigma(\tau, \nu)$ such that
\[
\lVert \partial_{\psi} K_{00} \rVert_s^{Lip(\gamma)}+\lVert K_{10}-\omega \rVert_s^{Lip(\gamma)}+\lVert K_{01} \rVert_s^{Lip(\gamma)}\le_s\lVert Z \rVert_{s+\sigma}^{Lip(\gamma)}+\lVert Z \rVert_{s_0+\sigma}^{Lip(\gamma)} \lVert \mathfrak{I}_0 \rVert_{s+\sigma}^{Lip(\gamma)}.
\]
\end{lem}
\begin{remark}\label{isotropic}
By Lemma \ref{Lemma6.1} if $\mathcal{F}(i_0, \zeta_0)=0$ and, by Lemma \ref{Lemma6.4}, the Hamiltonian \eqref{K} simplifies to 
\begin{equation}\label{NormalForm}
K=const+\omega\cdot\eta+\frac{1}{2} K_{20}(\psi) \eta\cdot\eta+(K_{11}(\psi) \eta, w)_{L^2(\T)}+\frac{1}{2} (K_{02}(\psi) w, w)_{L^2(\T)}+K_{\geq 3}.
\end{equation}
In general, the normal form \eqref{NormalForm} provides a control of the linearized equations in the normal bundle of the torus.
\end{remark}
We now estimate $K_{20}, K_{11}$ in \eqref{K}. The norm of $K_{20}$ is the sum of the norms of its matrix entries.
\begin{lem}{(Lemma $6.6$ in \cite{KdVAut})}\label{Lemma6.6}
Assume \eqref{Assumption}. Then for some $\sigma:=\sigma(\nu, \tau)$ we have
\begin{align}
&\lVert K_{20}-\frac{\varepsilon^{2 b}}{2} \mathbb{M}\rVert_s^{Lip(\gamma)}\le_s \varepsilon^{2b+2}+\varepsilon^{2 b}\lVert \mathfrak{I}_0 \rVert_{s+\sigma}^{Lip(\gamma)}+\varepsilon^3 \gamma^{-1} \lVert \mathfrak{I}_0 \rVert_{s_0+\sigma}^{Lip(\gamma)}\lVert Z \rVert_{s+\sigma}^{Lip(\gamma)},\\
&\lVert K_{11} \eta\rVert_s^{Lip(\gamma)}\le_s \varepsilon^5 \gamma^{-1} \lVert \eta \rVert_s^{Lip(\gamma)}+\varepsilon^{2 b-1} (\lVert \mathfrak{I}_0 \rVert_{s+\sigma}^{Lip(\gamma)}+\gamma^{-1} \lVert \mathfrak{I}_0 \rVert_{s_0+\sigma}^{Lip(\gamma)}\lVert Z \rVert_{s+\sigma}^{Lip(\gamma)})\lVert \eta \rVert_{s_0}^{Lip(\gamma)},\\
&\lVert K_{11}^T w \rVert\le_s \varepsilon^5 \gamma^{-1} \lVert w \rVert_{s+2}^{Lip(\gamma)}+\varepsilon^{2 b-1} (\lVert \mathfrak{I}_0 \rVert_{s+\sigma}^{Lip(\gamma)}+\gamma^{-1} \lVert \mathfrak{I}_0 \rVert_{s_0+\sigma}^{Lip(\gamma)} \lVert Z \rVert_{s+\sigma}^{Lip(\gamma)})\lVert w \rVert_{s_0+2}^{Lip(\gamma)}.
\end{align}
In particular 
\begin{align*}
&\lVert K_{20}-\frac{\varepsilon^{2 b}}{2} \mathbb{M}\rVert_{s_0}^{Lip(\gamma)}\le \varepsilon^6 \gamma^{-1}, \quad \lVert K_{11}\eta \rVert_{s_0}^{Lip(\gamma)}\le \varepsilon^5 \gamma^{-1}\lVert \eta \rVert_{s_0}^{Lip(\gamma)}, \quad \lVert K_{11}^T w \rVert_{s_0}^{Lip(\gamma)}\le \varepsilon^5 \gamma^{-1} \lVert w \rVert_{s_0}^{Lip(\gamma)}.
\end{align*}
\end{lem}
We apply the linear change of variables
\begin{equation}\label{DG}
DG_{\delta}(\varphi, 0, 0) \begin{pmatrix}
\hat{\psi}\\ \hat{\eta}\\ \hat{w} 
\end{pmatrix} :=\begin{pmatrix}
\partial_{\psi} \theta_0(\varphi) & 0 & 0\\
\partial_{\psi} y_{\delta}(\varphi) & [\partial_{\psi} \theta_0(\varphi)]^{-T} & -[(\partial_{\theta} \tilde{z}_0)(\theta_0 (\varphi))]^T \partial_x^{-1}\\
\partial_{\psi} z_0(\varphi) & 0 & \mathrm{I}
\end{pmatrix}
\begin{pmatrix}
\hat{\psi} \\ \hat{\eta} \\ \hat{w}.
\end{pmatrix}
\end{equation}
In these new coordinates the linearized operator $d_{i, \zeta}\mathcal{F}(i_{\delta}, \zeta_0)$ is ``approximately'' the operator obtained linearizing \eqref{VectorFieldK} at $(\psi, \eta, w, \zeta)=(\varphi, 0, 0, \zeta_0)$ with $\mathcal{D}_{\omega}$ instead of $\partial_t$, namely
\begin{equation}\label{6.26}
\begin{pmatrix}
\mathcal{D}_{\omega} \hat{\psi}-\partial_{\psi} K_{10}(\varphi)[\hat{\psi}]-K_{20}(\varphi) \hat{\eta}-K_{11}^T(\varphi) \hat{w}\\
\mathcal{D}_{\omega} \hat{\eta}+[\partial_{\psi} \theta_0(\varphi)]^T \hat{\zeta}+\partial_{\psi} [\partial_{\psi} \theta_0(\varphi)]^T[\hat{\psi}, \zeta_0]+\partial_{\psi\psi} K_{00}(\varphi)[\hat{\psi}]+[\partial_{\psi} K_{10}(\varphi)]^T\hat{\eta}+[\partial_{\psi} K_{01}(\varphi)]^T \hat{w}\\
\mathcal{D}_{\omega} \hat{w}-\partial_x\{ \partial_{\psi} K_{01}(\varphi)[\hat{\psi}]+K_{11}(\varphi)\hat{\eta}+K_{02}(\varphi)\hat{w}\}.
\end{pmatrix}
\end{equation}
We give estimate on the composition operator induced by the transformation \eqref{DG}.
\begin{lem}{(Lemma $6.7$ in \cite{KdVAut})}
Assume \eqref{Assumption} and let $\hat{\imath}:=(\hat{\psi}, \hat{\eta}, \hat{w})$. Then, for some $\sigma:=\sigma(\tau, \nu)$, we have
\begin{equation}
\begin{aligned}
&\lVert DG_{\delta}(\varphi, 0, 0)[\hat{\imath}] \rVert_{s}+\lVert DG_{\delta}(\varphi, 0, 0)^{-1}[\hat{\imath}] \rVert_s\le_s \lVert \hat{\imath} \rVert_s+(\lVert \mathfrak{I}_0 \rVert_{s+\sigma}+\gamma^{-1} \lVert \mathfrak{I}_0 \rVert_{s+\sigma}^{Lip(\gamma)}\lVert Z \rVert_{s+\sigma})\lVert \hat{\imath} \rVert_{s_0}\\
&\lVert D^2G_{\delta}(\varphi, 0, 0)[\hat{\imath}_1, \hat{\imath}_2]\rVert_s\le_s \lVert \hat{\imath}_1\rVert_s \lVert \hat{\imath}_2 \rVert_{s_0}+\lVert \hat{\imath}_1 \rVert_{s_0} \lVert \hat{\imath} \rVert_s+(\lVert \mathfrak{I}_0 \rVert_{s+\sigma}+\gamma^{-1} \lVert \mathfrak{I}_0 \rVert_{s_0+\sigma}\lVert Z \rVert_{s+\sigma})\lVert \hat{\imath} \rVert_{s_0} \lVert \hat{\imath}_2 \rVert_{s_0}.
\end{aligned}
\end{equation}
Moreover the same estimates hold if we replace $\lVert \cdot \rVert_s$ with $\Vert \cdot \rVert_s^{Lip(\gamma)}$.
\end{lem}
In order to construct an approximate inverse of \eqref{6.26} it is sufficient to solve the system of equations
\begin{equation}\label{D}
\mathbb{D}[\hat{\psi}, \hat{\eta}, \hat{w}, \hat{\zeta}]:=\begin{pmatrix}
\mathcal{D}_{\omega} \hat{\psi}-K_{20} (\varphi) \hat{\eta}-K_{11}^T(\varphi) \hat{w}\\
\mathcal{D}_{\omega} \hat{\eta}+[\partial_{\psi} \theta_0(\varphi)]^T \hat{\zeta}\\
\mathcal{D}_{\omega}\hat{w}-\partial_x K_{11}(\varphi)\hat{\eta}-\partial_x K_{02} (\varphi) \hat{w}
\end{pmatrix}=\begin{pmatrix}
g_1 \\ g_2 \\ g_3
\end{pmatrix}
\end{equation}
which is obtained by \eqref{6.26} neglecting the terms that are naught at a solution, namely, by Lemmata \eqref{Lemma6.1} and \eqref{Lemma6.4}, $\partial_{\psi} K_{10}, \partial_{\psi \psi} K_{00}, \partial_{\psi} K_{00}, \partial_{\psi} K_{01}$ and $\partial_{\psi} [\partial_{\psi} \theta_0(\varphi)]^T[\cdot , \zeta_0]$.\\


\begin{remark}
We will use the following notations for the averages of a function $v(\varphi, x)$
\begin{equation}
M_x[v]:=\frac{1}{2\pi}\int_{\T} v(\varphi, x)\,dx, \quad M_{\varphi}[v]:=\frac{1}{(2\pi)^{\nu}}\int_{\T^{\nu}} v(\varphi, x)\,d\varphi 
\end{equation}
and $M_{\varphi, x}[v]:=M_{x}[M_{\varphi}[v]]=M_{\varphi}[M_x[v]]$.
\end{remark}

First, we solve the second equation, namely
\begin{equation}\label{secondeq}
\mathcal{D}_{\omega}\hat{\eta}=g_2-[\partial_{\psi} \theta_0(\varphi)] \hat{\zeta}.
\end{equation}
We choose $\hat{\zeta}$ so that the $\varphi$-average of the right hand side of \eqref{secondeq} is zero, namely
\begin{equation}\label{ValoreperZeta}
\hat{\zeta}=M_{\varphi}[ g_2].
\end{equation}
Note that the $\varphi$-averaged matrix $M_{\varphi} [(\partial_{\psi} \theta_0)^T]=M_{\varphi}[ \mathrm{I}+(\partial_{\psi} \Theta_0)^T]=\mathrm{I}$ because $\theta_0(\varphi)=\varphi+\Theta_0(\varphi)$ and $\Theta_0(\varphi)$ is periodic. Therefore
\begin{equation}\label{eta}
\hat{\eta}=\mathcal{D}_{\omega}^{-1} (g_2-[\partial_{\psi} \theta_0(\varphi)]^T M_{\varphi}[ g_2])+M_{\varphi} [\hat{\eta}], \qquad M_{\varphi} [\hat{\eta}]\in\mathbb{R}^{\nu},
\end{equation}
where the average $M_{\varphi}[\hat{\eta}]$ will be fix when we deal with the first equation.\\
We now analyze the third equation, namely
\begin{equation}\label{thirdeq}
\mathcal{L}_{\omega} \hat{w}=g_3+\partial_x K_{11}(\varphi) \hat{\eta}, \quad \mathcal{L}_{\omega}:=\omega\cdot \partial_{\varphi}-\partial_x K_{02} (\varphi).
\end{equation}
If we fix $\hat{\eta}$, then solving the equation \eqref{thirdeq} is tantamount to invert the operator $\mathcal{L}_{\omega}$. For the moment we assume the following hypotesis (that will be proved in Section $8$)
\begin{itemize}
\item[$\bullet$] \textbf{Inversion Assumption}. There exists a set $\Omega_{\infty}\subseteq \Omega_{\varepsilon}$ such that for all $\omega\in \Omega_{\infty}$, for every function $g\in H_{S^{\perp}}^{s+\mu}(\mathbb{T}^{\nu+1})$ there exists a solution $h:=\mathcal{L}_{\omega}^{-1} g$ of the linear equation $\mathcal{L}_{\omega} h=g$ which satisfies
\begin{equation}\label{InversionAssumption}
\lVert \mathcal{L}_{\omega}^{-1} g \rVert_s^{Lip(\gamma)}\le_s \gamma^{-1} (\lVert g \rVert_{s+\mu}^{Lip(\gamma)}+\varepsilon \gamma^{-1}\{ \lVert \mathfrak{I}_0 \rVert_{s+\mu}^{Lip(\gamma)}+\gamma^{-1}\lVert \mathfrak{I}_0 \rVert_{s+\mu}^{Lip(\gamma)} \lVert Z \rVert_{s+\mu}^{Lip(\gamma)}\}\lVert g \rVert_{s_0}^{Lip(\gamma)})
\end{equation}
for some $\mu:=\mu(\tau, \nu)$.
\begin{remark}
The term $\varepsilon \gamma^{-1} \lVert \mathfrak{I}_0 \rVert_{s+\mu}^{Lip(\gamma)}$ arises because the remainder $\mathcal{R}_6$ in Section $8$ contains the term $\varepsilon(\lVert \Theta_0 \rVert_{s+\mu}^{Lip(\gamma)}+\lVert y_{\delta} \rVert_{s+\mu}^{Lip(\gamma)})\le_s \varepsilon \lVert \mathfrak{I}_{0} \rVert_{s+\mu}^{Lip(\gamma)}$, see Lemma \ref{LemmaS}.\\
These big constants coming from the tame estimates for the inverse of the linearized operators at any approximate solution will be dominated by the quadraticity of the Nash-Moser scheme.
\end{remark}
By the above assumption, there exists a solution of \eqref{thirdeq}
\begin{equation}\label{w}
\hat{w}=\mathcal{L}_{\omega}^{-1}[g_3+\partial_x K_{11}(\varphi)\hat{\eta}].
\end{equation}
Now consider the first equation
\begin{equation}\label{firsteq}
\mathcal{D}_{\omega}\hat{\psi}=g_1+K_{20} \hat{\eta}-K_{11}^T(\varphi) \hat{w}.
\end{equation}
 Substituting \eqref{eta}, \eqref{w} in the equation \eqref{firsteq}, we get 
 \begin{equation}\label{firsteq2}
 \mathcal{D}_{\omega} \hat{\psi}=g_1+M_1(\varphi)M_{\varphi}[ \hat{\eta} ]+M_2(\varphi) g_2+M_3(\varphi) g_3-M_2(\varphi)[\partial_{\psi} \theta_0]^T M_{\varphi}[ g_2 ],
 \end{equation}
 where
 \begin{equation}
 M_1(\varphi):=K_{20}(\varphi)+K_{11}^T(\varphi)\mathcal{L}_{\omega}^{-1}\partial_x K_{11}(\varphi), \quad M_2(\varphi):=M_1(\varphi)\mathcal{D}_{\omega}^{-1}, \quad M_3(\varphi):=K_{11}^T(\varphi)\mathcal{L}_{\omega}^{-1}.
 \end{equation}
In order to solve the equation \eqref{firsteq2} we have to choose $M_{\varphi} [\hat{\eta}]$ such that the right hand side in \eqref{firsteq2} has zero $\varphi$-average.\\
 By Lemma \ref{Lemma6.6} and \eqref{Assumption}, the $\varphi$-averaged matrix $M_{\varphi}[ M_1 ]=\varepsilon^{2 b}M+O(\varepsilon^{10} \gamma^{-3})$. Therefore, for $\varepsilon$ small, $M_{\varphi}[ M_1 ]$ is invertible and $M_{\varphi}[ M_1 ]^{-1}=O(\varepsilon^{-2b})=O(\gamma^{-1})$. Thus we define
 \begin{equation}
M_{\varphi} [\hat{\eta}]:=-(M_{\varphi} [M_1])^{-1}\{M_{\varphi} [g_1]+M_{\varphi}[ M_2 g_2]+M_{\varphi} [M_3 g_3 ]-M_{\varphi}[ M_2(\partial_{\psi} \theta_0)^T]\,M_{\varphi}[ g_2 ]\}.
 \end{equation}
 With this choice of $M_{\varphi}[ \hat{\eta} ]$ the equation \eqref{firsteq2} has the solution
 \begin{equation}\label{psi}
 \hat{\psi}:=\mathcal{D}_{\omega}^{-1}\{g_1+M_1(\varphi)M_{\varphi} [\hat{\eta}]+M_2(\varphi) g_2+M_3(\varphi) g_3-M_2(\varphi)[\partial_{\psi} \theta_0]^T M_{\varphi}[ g_2 ]\}.
 \end{equation}
\end{itemize}
In conclusion, we have constructed a solution $(\hat{\psi}, \hat{\eta}, \hat{w}, \hat{\zeta})$ of the linear system \eqref{D}.  We resume this in the following proposition, giving also estimates on the inverse of the operator $\mathbb{D}$ defined in \eqref{D}.
\begin{prop}{(Proposition $6.9$ in \cite{KdVAut})}
Assume \eqref{Assumption} and \eqref{InversionAssumption}. Then, for all $\omega\in \Omega_{\infty}$, for all $g:=(g_1, g_2, g_3)$, the system \eqref{D} has a solution $\mathbb{D}^{-1} g:=(\hat{\psi}, \hat{\eta}, \hat{w}, \hat{\zeta})$ where $(\hat{\psi}, \hat{\eta}, \hat{w}, \hat{\zeta})$ are defined in \eqref{psi}, \eqref{eta}, \eqref{w}, \eqref{ValoreperZeta}. Moreover, we have
\begin{equation}\label{EstimateonD}
\lVert \mathbb{D}^{-1} g \rVert_s^{Lip(\gamma)}\le_s \gamma^{-1} (\lVert g \rVert_{s+\mu}^{Lip(\gamma)}+\varepsilon \gamma^{-1}\{ \mathfrak{I}_0 \rVert_{s+\mu}^{Lip(\gamma)}+\gamma^{-1}\lVert \mathfrak{I}_0 \rVert_{s_0+\mu}^{Lip(\gamma)}\lVert \mathcal{F}(i_0, \zeta_0)\rVert_{s+\mu}^{Lip(\gamma)}\}\lVert g \rVert_{s_0+\mu}^{Lip(\gamma)}).
\end{equation}
\end{prop}
Eventually we prove that the operator
\begin{equation}\label{ApproxInverse}
\mathbf{T}_0:=(D \tilde{G}_{\delta})(\varphi, 0, 0)\circ\mathbb{D}^{-1}\circ (D G_{\delta}(\varphi, 0, 0))^{-1}
\end{equation}
is an approximate right inverse of $d_{i, \zeta} \mathcal{F}(i_0)$ where $\tilde{G}_{\delta}((\psi, \eta, w), \zeta)$ is the identity on the $\zeta$-component. We denote the norm $\lVert (\psi, \eta, w, \zeta)\rVert_s^{Lip(\gamma)}:=\max \{ \lVert (\psi, \eta, w) \rVert, \lvert \zeta \rvert^{Lip(\gamma)} \}$.
\begin{teor}{(Theorem $6.10$ in \cite{KdVAut})}\label{TeoApproxInv}
Assume \eqref{Assumption} and the inversion assumption \eqref{InversionAssumption}. Then there exists $\mu:=\mu(\tau, \nu)$ such that, for all $\omega\in \Omega_{\infty}$, for all $g:=(g_1, g_2, g_3)$, the operator $\mathbf{T}_0$ defined in \eqref{ApproxInverse} satisfies
\begin{equation}\label{TameEstimateApproxInv}
\lVert \mathbf{T}_0 g \rVert_s^{Lip(\gamma)}\le_s\gamma^{-1}(\lVert g \rVert_{s+\mu}^{Lip(\gamma)}+\varepsilon\gamma^{-1}\{ \lVert \mathfrak{I}_0 \rVert_{s+\mu}^{Lip(\gamma)}+\gamma^{-1}\lVert \mathfrak{I}_0 \rVert_{s_0+\mu}^{Lip(\gamma)}\lVert \mathcal{F}(i_0, \zeta_0)\rVert_{s+\mu}^{Lip(\gamma)}\} \lVert g \rVert_{s_0+\mu}^{Lip(\gamma)}).
\end{equation}
It is an approximate inverse of $d_{i, \zeta} \mathcal{F}(i_0)$, namely
\begin{equation}\label{6.41}
\begin{aligned}
&\lVert (d_{i, \zeta} \mathcal{F}(i_0)\circ \mathbf{T}_0-\mathrm{I}) g \rVert_s^{Lip(\gamma)}\le_s \\
&\le_s \gamma^{-1} \left( \lVert \mathcal{F}(i_0, \zeta_0) \rVert_{s_0+\mu}^{Lip(\gamma)}\lVert g \rVert_{s+\mu}^{Lip(\gamma)}+\{\lVert \mathcal{F}(i_0, \zeta_0)\rVert_{s+\mu}^{Lip(\gamma)}+\varepsilon \gamma^{-1}\lVert \mathcal{F}(i_0, \zeta_0) \rVert_{s_0+\mu}^{Lip(\gamma)}\lVert \mathfrak{I}_0 \rVert_{s+\mu}^{Lip(\gamma)}\} \lVert g \rVert_{s_0+\mu}^{Lip(\gamma)} \right).
\end{aligned}
\end{equation}
\end{teor}
\section{The linearized operator in the normal directions}
In this section we give an explicit expression of the linearized operator 
\begin{equation}
\mathcal{L}_{\omega}:=\omega\cdot\partial_{\varphi}-\partial_x K_{0 2}(\varphi).
\end{equation}
To this aim we compute $\frac{1}{2}\,(K_{0 2}(\psi) w, w)_{L^2(\mathbb{T})}, w\in H_S^{\perp}$, which collects all the terms of $(H_{\varepsilon}\circ G_{\delta})(\psi, 0, w)$ that are quadratic in $w$.\\
First we recall some preliminary lemmata.
\begin{lem}{(Lemma $7.1$ in \cite{KdVAut})}\label{Egorov}
Let $H$ be a Hamiltonian function of class $C^2(H_0^1(\mathbb{T}_x), \mathbb{R})$ and consider a map $\Phi(u):=u+\Psi(u)$ satisfying $\Psi(u)=\Pi_E \Psi(\Pi_E u),$ for all $u$, where $E$ is a finite dimensional subspace as in \eqref{FinitedimensionalSubspace}. Then
\begin{equation}
\partial_u [\nabla (H\circ \Phi)](u)[h]=(\partial_u \nabla H)(\Phi(u))[h]+\mathcal{R}(u)[h],
\end{equation}
where $\mathcal{R}(u)$ has the ``finite dimensional'' form
\begin{equation}\label{FinitedimensionalRemainder}
\mathcal{R}(u)[h]=\sum_{\lvert j \rvert\le C} (h, g_j(u))_{L^2(\mathbb{T})}\chi_j(u)
\end{equation}
with $\chi_j(u)=e^{\mathrm{i} j x}$ or $g_j(u)=e^{\mathrm{i} j x}$. The remainder in \eqref{FinitedimensionalRemainder} is
\[
\mathcal{R}(u)=\mathcal{R}_0(u)+\mathcal{R}_1(u)+\mathcal{R}_2(u)
\]
with
\begin{equation}\label{iRestidiEgorov}
\begin{aligned}
&\mathcal{R}_0(u):=(\partial_u \nabla H) (\Phi(u)) \partial_u \Psi(u), \quad \mathcal{R}_1(u):=[\partial_u \{\Psi'(u)^T\}] [\cdot, \nabla H(\Phi(u))],\\
&\mathcal{R}_2(u):=[\partial_u \Psi(u)]^T (\partial_u \nabla H)(\Phi(u)) \partial_u \Phi(u).
\end{aligned}
\end{equation}
\end{lem}
\begin{lem}{(Lemma $7.3$ in \cite{KdVAut})}\label{Lemma7.3}
Let $\mathcal{R}$ be an operator of the form
\begin{equation}\label{VeraFinDimForm}
\mathcal{R} h=\sum_{\lvert j \rvert\le C} \int_0^1 (h, g_j(\tau))_{L^2(\mathbb{T})} \chi_j(\tau)\,d\tau,
\end{equation}
where the functions $g_j(\tau), \chi_j(\tau)\in H^s, \tau\in [0, 1]$ depend in a Lipschitz way on the parameter $\omega$. Then its matrix $s$-decay norm (see \eqref{decayNorm}-\eqref{decayNorm2}) satisfies
\begin{equation}
\lvert \mathcal{R} \rvert_s^{Lip(\gamma)}\le_s \sum_{\lvert j \rvert\le C} \,\,\sup_{\tau\in [0,1]} (\lVert \chi_j(\tau)\rvert_s^{Lip(\gamma)}\lVert g_j \rVert_{s_0}^{Lip(\gamma)}+\lVert \chi_j(\tau)\rVert_{s_0}^{Lip(\gamma)}\lVert g_j(\tau)\rVert_s^{Lip(\gamma)}).
\end{equation}
\end{lem}
\subsection{Composition with the map $G_{\delta}$}
In the sequel we use the fact that $\mathfrak{I}_{\delta}:=\mathfrak{I}_{\delta}(\varphi; \omega)=i_{\delta}(\varphi;\, \omega)-(\varphi,\, 0,\, 0)$ satisfies
\begin{equation}\label{IpotesiPiccolezzaIdelta}
\lVert \mathfrak{I}_{\delta} \rVert_{s_0+\mu}^{Lip(\gamma)}\le C\,\varepsilon^{6-2 b}\gamma^{-1}.
\end{equation}
We now study the Hamiltonian $K:=H_{\varepsilon}\circ G_{\delta}=\varepsilon^{-2 b} \mathcal{H}\circ A_{\varepsilon}\circ G_{\delta}$ (see \eqref{Hepsilon}). Recalling \eqref{Aepsilon}, $A_{\varepsilon}\circ G_{\delta}$ has the form
\begin{equation}\label{AepsilonGdelta}
A_{\varepsilon}(G_{\delta}(\psi, \eta, w))=\varepsilon v_{\varepsilon}(\theta_0(\psi), y_{\delta}(\psi)+L_1(\psi)\eta+L_2(\psi)w)+\varepsilon^b (z_0(\psi)+w)
\end{equation}
where
\begin{equation}\label{L1L2}
L_1(\Psi):=[\partial_{\psi} \theta_0(\psi)]^{-T}, \quad L_2(\psi):=[(\partial_{\theta} \tilde{z}_0)(\theta_0(\psi))]^T \partial_x^{-1}.
\end{equation}
By Taylor formula, we develop \eqref{AepsilonGdelta} in $w$ at $(\eta, w)=(0, 0)$, and we get 
\[
(A_{\varepsilon}\circ G_{\delta})(\psi, 0, w)=T_{\delta}(\psi)+T_1(\psi) w+T_2(\psi) [w, w]+T_{\geq 3}(\psi, w),
\]
where 
\begin{equation}\label{Tdelta}
T_{\delta}(\psi):=A_{\varepsilon}(G_{\delta}(\psi, 0, 0))=\varepsilon v_{\delta}(\psi)+\varepsilon^b z_0(\psi), \quad v_{\delta}(\psi):=v_{\varepsilon} (\theta_0(\psi), y_{\delta}(\psi))
\end{equation}
is the approximate isotropic torus in the phase space $H_0^1(\mathbb{T})$ (it corresponds to $i_{\delta}$),
\begin{align}
& T_1(\psi) w:=\varepsilon^{2 b-1} U_1(\psi) w+\varepsilon^b w; \quad T_2(\psi) [w, w]:=\varepsilon^{4 b-3} U_2(\psi) [w, w]\\[2mm] \label{U1}
&U_1(\psi) w:=\varepsilon\sum_{j\in S} \,\,\dfrac{\lvert j \rvert\,[L_2(\psi) w]_j\,e^{\mathrm{i} [\theta_0(\psi)]_j}}{2\sqrt{\lvert j \rvert}\sqrt{\xi_j+\varepsilon^{2(b-1)} [y_{\delta}(\psi)]_j}},\\\label{U2}
&U_2(\psi) [w, w]:=-\varepsilon\sum_{j\in S} \,\,\dfrac{j^2\,[L_2(\psi) w]_j^2\,e^{\mathrm{i}[\theta_0(\psi)]_j}}{8\lvert j \rvert^{\frac{3}{2}} \{ \xi_j+\varepsilon^{2(b-1)} [y_{\delta}(\psi)]_j \}^{\frac{3}{2}}},
\end{align}
and $T_{\geq 3}(\psi, w)$ collects all the terms of order at least cubic in $w$. In the notation of \eqref{Aepsilon}, the function $v_{\delta}(\Psi)$ in \eqref{Tdelta} is $v_{\delta}(\psi)=v_{\varepsilon}(\theta_0(\psi), y_{\delta}(\psi))$. The terms $U_1, U_2$ in \eqref{U1}, \eqref{U2} are $O(1)$ in $\varepsilon$. Moreover, using that $L_2(\psi)$ in \eqref{L1L2} vanishes at $z_0=0$, they satisfy
\begin{equation}
\begin{aligned}
&\lVert U_1 w \rVert_s\le_s \lVert \mathfrak{I}_{\delta} \rVert_s \lVert w \rVert_{s_0}+\lVert \mathfrak{I}_{\delta} \rVert_{s_0} \lVert w \rVert_s, \quad \lVert U_2 [w, w] \rVert_s \le_s \lVert \mathfrak{I}_{\delta} \rVert_s \lVert \mathfrak{I}_{\delta} \rVert_{s_0}\lVert w \rVert_{s_0}^2+\lVert \mathfrak{I}_{\delta} \rVert_{s_0}^2 \lVert w \rVert_{s_0} \lVert w \rVert_s
\end{aligned}
\end{equation}
and also in the norm $\lVert \cdot \rVert_s^{Lip(\gamma)}$. We expand $\mathcal{H}$ by Taylor formula
\begin{equation}
\mathcal{H}(u+h)=\mathcal{H}(u)+((\nabla\mathcal{H})(u), h)_{L^2(\mathbb{T})}+\frac{1}{2}((\partial_u \nabla \mathcal{H})(u)[h], h)_{L^2(\mathbb{T})}+O(h^3).
\end{equation}
Specifying at $u=T_{\delta}(\psi)$ and $h=T_1(\psi) w+T_2(\psi) [w, w]+T_{\geq 3} (\psi, w)$, we obtain that the sum of all components of $K=\varepsilon^{-2 b} (\mathcal{H}\circ A_{\varepsilon}\circ G_{\delta})(\psi, 0, w)$ that are quadratic in $w$ is
\begin{equation}\label{equality}
\begin{aligned}
& \frac{1}{2} (K_{0 2} w, w)_{L^2(\mathbb{T})}=\varepsilon^{-2 b} ( (\nabla \mathcal{H})(T_{\delta}) , T_2 [w, w] )_{L^2(\mathbb{T})}+\frac{\varepsilon^{-2 b}}{2} ((\partial_u \nabla \mathcal{H})(T_{\delta})[T_1 w] , T_1 w )_{L^2(\mathbb{T})}.
\end{aligned}
\end{equation}
Inserting the expressions \eqref{U1}, \eqref{U2} in the equality \eqref{equality}, we get
\begin{equation}
\begin{aligned}
K_{0 2} (\psi) w=& (\partial_u \nabla \mathcal{H})(T_{\delta})[w]+2 \varepsilon^{b-1} (\partial_u \nabla \mathcal{H})(T_{\delta})[U_1 w]+\\
&+\varepsilon^{2 (b-1)} U_1^T (\partial_u \nabla\mathcal{H})(T_{\delta}) [U_1 w]+2\,\varepsilon^{2 b-3} U_2 [w, \cdot]^T (\nabla\mathcal{H})(T_{\delta}).
\end{aligned}
\end{equation}

\begin{lem}\label{ApplicazioneEgorov}
The operator $K_{0 2}$ reads 
\begin{equation}
(K_{0 2} w, w)_{L^2(\mathbb{T})}=((\partial_u \nabla \mathcal{H})(T_{\delta})[w], w)_{L^2(\mathbb{T})}+(R(\psi) w, w)_{L^2(\mathbb{T})}
\end{equation}
where $R(\psi)$ has the ``finite dimensional'' form
\begin{equation}\label{FiniteDimForm}
R(\psi) w=\sum_{\lvert j \rvert\le C} (w, g_j(\psi))_{L^2(\mathbb{T})}\,\chi_j(\psi).
\end{equation}
The functions $g_j, \chi_j$ satisfy, for some $\sigma:=\sigma (\nu, \tau)>0$,
\begin{align}
&\lVert g_j \rVert_s^{Lip(\gamma)} \lVert \chi_j \rVert_{s_0}^{Lip(\gamma)}+\lVert g_j \rVert_{s_0}^{Lip(\gamma)}\le_s \varepsilon^{1+b} \lVert \mathfrak{I}_{\delta} \rVert_{s+\sigma}^{Lip(\gamma)}, \label{FiniteDimSize}\\
&\lVert \partial_i g_j [\hat{\imath}] \rVert_s \lVert \chi_j \rVert_{s_0}+\lVert \partial_i g_j [\hat{\imath}] \rVert_{s_0} \lVert \chi_j \rVert_{s}+\lVert g_j \rVert_s \lVert \partial_i \chi_j [\hat{\imath}] \rVert_{s_0}+\lVert g_j \rVert_{s_0} \lVert \partial_i \chi_j [\hat{\imath}] \rVert_{s} \notag\\
&\le_s \varepsilon^{1+b} \lVert \hat{\imath} \rVert_{s+\sigma}+\varepsilon^{2 b-1} \lVert \mathfrak{I}_{\delta} \rVert_{s+\sigma} \lVert \hat{\imath} \rVert_{s+\sigma}
\end{align}
\end{lem}
In conclusion,  the linearized operator to analyze after the composition with the action-angle variables, the rescaling and the transformation $G_{\delta}$ is
\[
w \mapsto (\partial_u \nabla \mathcal{H}) (T_{\delta}) [w],  \quad w\in H_S^{\perp}
\]
up to finite dimensional operators which have form \eqref{FiniteDimForm} and size \eqref{FiniteDimSize}.

\subsection{The linearized operator in the normal directions}
In this section we compute $((\partial_u \nabla \mathcal{H})(T_{\delta})[w], w)_{L^2(\mathbb{T})}, w\in H_S^{\perp}$, recalling that $\mathcal{H}=H\circ \Phi_B$ and $\Phi_B$ is the Birkhoff map of Proposition \ref{WBNF}. It is convenient to write separately the terms in
\begin{equation}
\mathcal{H}=H\circ \Phi_B=(H_2+H_3)\circ \Phi_B+H_4\circ \Phi_B+H_{\geq 5}\circ \Phi_B,
\end{equation}
where $H_2, H_3, H_4, H_{\geq 5}$ are defined in \eqref{Hamiltonians}.
First we consider $H_{\geq 5}\circ \Phi_B$. By \eqref{Hamiltonians} we get
\[
\nabla H_{\geq 5} (u)=\pi_0[(\partial_u f)(x, u, u_x)]-\partial_x \{ (\partial_{u_x} f) (x, u, u_x) \}.
\]
Since the Birkhoff transformation $\Phi_B$ has the form \eqref{IdpiuFiniteRank}, Lemma \ref{Egorov} (at $u=T_{\delta}$) implies that
\begin{equation}
\begin{aligned}
\partial_u \nabla (H_{\geq 5}\circ \Phi_B)(T_{\delta})[h]&=(\partial_u \nabla H_{\geq 5})(\Phi_B(T_{\delta}))[h]+\mathcal{R}_{H_{\geq 5}}(T_{\delta})[h]=\\
&=\partial_x (r_1 (T_{\delta})\,\partial_x h)+r_0(T_{\delta}) h+\mathcal{R}_{H_{\geq 5}}(T_{\delta})[h]
\end{aligned}
\end{equation}
where the multiplicative functions $r_0(T_{\delta}), r_1(T_{\delta})$ are 
\begin{align}
& r_0(T_{\delta}):=\sigma_0(\Phi_B(T_{\delta})), \quad \sigma_0(u):=(\partial_{uu} f)(x, u, u_x)-\partial_x\{(\partial_{u u_x} f)(x, u, u_x)\}, \label{r0}\\
& r_1(T_{\delta}):=\sigma_1(\Phi_B(T_{\delta})), \quad \sigma_1(u):=-(\partial_{u_x u_x} f) (x, u, u_x),\label{r1}
\end{align}
the remainder $\mathcal{R}_{H_{\geq 5}}(u)$ has the form \eqref{FinitedimensionalRemainder} with $\chi_j=e^{\mathrm{i} j x}$ or $g_j=e^{\mathrm{i} j x}$ and it satisfies, for some $\sigma:=\sigma(\nu, \tau)>0$,
\begin{align*}
&\lVert g_j \rVert_s^{Lip(\gamma)} \lVert \chi_j \rVert_{s_0}^{Lip(\gamma)}+\lVert g_j \rVert_{s_0}^{Lip(\gamma)}\le_s \varepsilon^4 (1+\lVert \mathfrak{I}_{\delta} \rVert_{s+2}^{Lip(\gamma)}),\\[2mm]
& \lVert \partial_i g_j [\hat{\imath}] \rVert_s \lVert \chi_j \rVert_{s_0}+\lVert \partial_i g_j [\hat{\imath}] \rVert_{s_0} \lVert \chi_j \rVert_{s}+\lVert g_j \rVert_s \lVert \partial_i \chi_j [\hat{\imath}] \rVert_{s_0}+\lVert g_j \rVert_{s_0} \lVert \partial_i \chi_j [\hat{\imath}] \rVert_{s}\le_s \varepsilon^4 (\lVert \hat{\imath} \rVert_{s+\sigma}+\lVert \mathfrak{I}_{\delta} \rVert_{s+2} \lVert \hat{\imath} \rVert_{s_0+2}).
\end{align*}
Now consider the contribution of $(H_2+H_3+H_4)\circ\Phi_B$. By Lemma \ref{Egorov} and \eqref{Hamiltonians} we have
\begin{equation}
\begin{aligned}
&\partial_u \nabla ((H_2+H_3+H_4)\circ \Phi_B) (T_{\delta}) [h]=-h_{xx}-6\,c_1\,\partial_x [\Phi_B(T_{\delta})_x\,h_x]-2\,c_2\,\partial_{xx} (\Phi_B(T_{\delta}) h)\\
&+2\,c_2\,\Phi_B(T_{\delta})_x\,h_x+6\,c_3\,\Phi_B(T_{\delta})\,h-12\,c_4\,\partial_x [(\Phi_B(T_{\delta}))_x^2\,h_x]-3\,c_5\,\partial_x [(\Phi_B(T_{\delta}))_x^2\,h]\\
&+3\,c_5\,(\Phi_B(T_{\delta}))_x^2\,h_x-2\,c_6\,\partial_x [\Phi_B(T_{\delta})^2\,h_x]-2\,c_6\,\partial_{xx} (\Phi_B(T_{\delta})^2)\,h+2\,c_6 \,\Phi_B(T_{\delta})_x^2\,h\\
&+12\,c_7\,\Phi_B(T_{\delta})^2\,h+\mathcal{R}_{H_2}(T_{\delta})+\mathcal{R}_{H_3}(T_{\delta})+\mathcal{R}_{H_4}(T_{\delta})[h],
\end{aligned}
\end{equation}
where $\Phi_B(T_{\delta})$ is a zero space average function, indeed $\Phi_B$ maps $H_{0}^1(\mathbb{T}_x)$ in itself by Proposition \eqref{WBNF}. The remainder $\mathcal{R}_{H_2}, \mathcal{R}_{H_3}, \mathcal{R}_{H_4}$ have the form \eqref{FinitedimensionalRemainder} and, by \eqref{iRestidiEgorov}, the size $(\mathcal{R}_{H_2}+\mathcal{R}_{H_3}+\mathcal{R}_{H_4})(T_{\delta})=O(\varepsilon)$. We develop this sum as
\begin{equation}
(\mathcal{R}_{H_2}+\mathcal{R}_{H_3}+\mathcal{R}_{H_4})(T_{\delta})=\varepsilon \mathcal{R}_1+\varepsilon^2 \mathcal{R}_2+\tilde{\mathcal{R}}_{>2},
\end{equation}
where $\tilde{\mathcal{R}}_{>2}$ has size $o(\varepsilon^2)$. Thus we get, for all $h\in H_S^{\perp}$,
\begin{equation}\label{Comparazione}
\begin{aligned}
&\Pi_S^{\perp} \partial_u \nabla ((H_2+H_3+H_4)\circ \Phi_B)(T_{\delta})[h]=-h_{xx}+\Pi_S^{\perp} \{-6\,c_1\,\partial_x [\Phi_B(T_{\delta})_x\,h_x]-2\,c_2\,\partial_{xx} (\Phi_B(T_{\delta}) h)\\
&+2\,c_2\,\Phi_B(T_{\delta})_x\,h_x+6\,c_3\,\Phi_B(T_{\delta})\,h-12\,c_4\,\partial_x [(\Phi_B(T_{\delta}))_x^2\,h_x]-3\,c_5\,\partial_x [(\Phi_B(T_{\delta}))_x^2\,h]+3\,c_5\,(\Phi_B(T_{\delta}))_x^2\,h_x\\&-2\,c_6\,\partial_x [\Phi_B(T_{\delta})^2\,h_x]-2\,c_6\,\partial_{xx} (\Phi_B(T_{\delta})^2)\,h+2\,c_6 \,\Phi_B(T_{\delta})_x^2\,h+12\,c_7\,\Phi_B(T_{\delta})^2\,h \}\\
&+\Pi_s^{\perp}(\varepsilon \mathcal{R}_1+\varepsilon^2 \mathcal{R}_2+\tilde{\mathcal{R}}_{>2})[h].
\end{aligned}
\end{equation}
Now we expand $\Phi_B(u)=u+\Psi_2(u)+\Psi_{\geq 3}(u)$, where $\Psi_2(u)$ is a quadratic function of $u$, $\Psi_{\geq 3}=O(u^3)$ and both map $H_{0}^1(\mathbb{T}_x)$ in itself. At $u=T_{\delta}=\varepsilon v_{\delta}+\varepsilon^b z_0$ we get
\begin{equation}\label{PhiB}
\Phi_B(T_{\delta})=T_{\delta}+\Psi_2(T_{\delta})+\Psi_{\geq 3}(T_{\delta})=\varepsilon v_{\delta}+\varepsilon^2 \Psi_2(v_{\delta})+\tilde{q},
\end{equation}
where $\tilde{q}=\varepsilon^b z_0+\Psi_2(T_{\delta})-\Psi_2(v_{\delta})+\Psi_{\geq 3}(T_{\delta})$ and it satisfies
\begin{equation}\label{Qtilda}
\lVert \tilde{q} \rVert_s^{Lip(\gamma)}\le_s \varepsilon^3+\varepsilon^b\lVert \mathfrak{I}_{\delta} \rVert_s^{Lip(\gamma)}, \quad \lVert \partial_i \tilde{q} [\hat{\imath}] \rVert_s\le_s \varepsilon^b (\lVert \hat{\imath} \rVert_s+\lVert \mathfrak{I}_{\delta} \rVert_s \lVert \hat{\imath} \rVert_{s_0}). 
\end{equation}
Note that also $\tilde{q}$ has zero space average, indeed $\tilde{q}=\Phi_B(T_{\delta})-\varepsilon v_{\delta}-\varepsilon^2 \Psi_2(v_{\delta})$ and $\Phi_B(T_{\delta}), v_{\delta}, \Psi_2(v_{\delta})$ belong to $ H_0^1(\mathbb{T}_x)$.\\
We observe that the terms $O(\varepsilon)$ come from the monomials $R(v\,z^2)$ of $\mathcal{H}_3$ and the ones of size $O(\varepsilon^2)$ from $H_2+\mathcal{H}_{4,2}$ (see \eqref{HamiltonianeMathcal}). Thus, we compare \eqref{Comparazione} with $\Pi_S^{\perp}(\partial_u \nabla (H_2+\mathcal{H}_3+\mathcal{H}_{4, 2}))(T_{\delta})[h]$, using \eqref{HamiltonianeMathcal}, and, by \eqref{PhiB}, we obtain $\mathcal{R}_1=0$,
\begin{align}
&\Psi_2(v_{\delta})=-c_1\,\partial_x (v_{\delta}^2)-\frac{c_2}{3} \partial_{xx} [(\partial_x^{-1} v_{\delta})^2]+\frac{c_2}{3} \pi_0[v_{\delta}^2]+c_3 \pi_0[(\partial_x^{-1} v_{\delta})^2]
\end{align}
and
\begin{equation}\label{R2}
\begin{aligned}
\mathcal{R}_2 [h]=& -6 c_1^2 \{v_{\delta} \partial_{xx} (\Pi_S[(v_{\delta})_x h_x]) - \partial_x ((v_{\delta})_x \partial_{xx} \Pi_S[v_{\delta} h])\}\\
&+2 c_1 c_2\,v_{\delta}\, \partial_x (\Pi_S[(v_{\delta})_x\,h_x])+ 2 c_1 c_2\, \partial_x ((v_{\delta})_x \,\partial_x \Pi_S[v_{\delta}\,h])\\
&-2 c_1 c_2\,(\partial_x^{-1} v_{\delta})\,\partial_{xx}\Pi_S[(v_{\delta})_x\,h_x]+2 c_1 c_2\,\partial_{x} \{(v_{\delta})_x\,\partial_{xx} \Pi_S[(\partial_x^{-1} v_{\delta})\,h]\}\\
&-\frac{2\,c_2^2}{3} \, (\partial_x^{-1} v_{\delta})\,\partial_{xxx} \Pi_S[v_{\delta}\,h]+\frac{2\,c_2^2}{3} \,(\partial_x^{-1} v_{\delta})\,\partial_x\Pi_S[(v_{\delta})_x\,h_x]\\
&+\frac{2\,c_2^2}{3}\,\partial_x \{ (v_{\delta})_x\,\partial_x \Pi_S[(\partial_x^{-1} v_{\delta})\,h]+2 c_2 c_3\,(\partial_x^{-1} v_{\delta})\,\partial_x\Pi_S[v_{\delta}\,h]\\
&-2 c_2 c_3\,v_{\delta}\,\partial_x \Pi_S[(\partial_x^{-1} v_{\delta}) h]+2 c_1 c_2\,\partial_x^{-1}\{v_{\delta}\,\partial_{xx}\Pi_S[(v_{\delta})_x\,h_x]  \}\\
&+2 c_1 c_2\,\partial_x \{ (v_{\delta})_x\,\partial_{xx} \Pi_S[v_{\delta}\,(\partial_x^{-1} h)] \}+\frac{2\,c_2^2}{3} \,\partial_x^{-1} \{ v_{\delta}\,\partial_{xxx} \Pi_S[v_{\delta} h]\}\\
&+\frac{2\,c_2^2}{3}\,v_{\delta}\,\partial_{xxx} \Pi_S[v_{\delta}\,(\partial_x^{-1} h)]-\frac{2\,c_2^2}{3} \,(\partial_x^{-1} \{ v_{\delta})\,\partial_x\Pi_S[(v_{\delta})_x h_x] \}\\
&+\frac{2\,c_2^2}{3}\,\partial_x\{ (v_{\delta})_x\,\partial_x\Pi_S[v_{\delta} (\partial_x^{-1} h)]-2 c_2 c_3\,\partial_x^{-1} \{ v_{\delta}\,\partial_x\Pi_S[v_{\delta}\,h]\}\\
&-2 c_2 c_3 v_{\delta}\,\partial_x \Pi_S[v_{\delta}\,(\partial_x^{-1} h)]-2 c_1 c_2\,v_{\delta}\,\partial_x\Pi_S[(v_{\delta})_x\,h_x]\\
&-2 c_1 c_2 \,\partial_x \{ (v_{\delta})_x\,\partial_x \Pi_S[v_{\delta}\,h] \}-\frac{4\,c_2^2}{3}\,v_{\delta}\,\partial_{xx}\Pi_S[v_{\delta}\,h]\\
&+\frac{2\,c_2^2}{3}\,v_{\delta}\,\Pi_S[(v_{\delta})_x\,h_x]-\frac{2\,c_2^2}{3}\,\partial_x \{ (v_{\delta})_x \,\Pi_S[v_{\delta}\,h]\}\\
&+4 c_2 c_3\,v_{\delta}\,\Pi_S[ v_{\delta}\,h]+6 c_1 c_3\,\partial_x^{-1} \{ (\partial_x^{-1} v_{\delta})\,\partial_x \Pi_S[(v_{\delta})_x\,h_x] \}\\
&-6 c_1 c_3\,\partial_x \{ (v_{\delta})_x\,\partial_x \Pi_S[(\partial_x^{-1} v_{\delta})(\partial_x^{-1} h)]  \}+2 c_2 c_3 \partial_x^{-1} \{ (\partial_x^{-1} v_{\delta})\,\partial_{xx} \Pi_S[v_{\delta} h] \}\\
&-2 c_2 c_3 \,v_{\delta}\,\partial_{xx}\Pi_S[(\partial_x^{-1} v_{\delta})(\partial_x^{-1} h)]-2 c_2 c_3\,\partial_x^{-1} \{ (\partial_x^{-1} v_{\delta})\,\Pi_S[(v_{\delta})_x h_x] \}\\
&-2 c_2 c_3\,\partial_x \{ (v_{\delta})_x\,\Pi_S[(\partial_x^{-1} v_{\delta})(\partial_x^{-1} h)]\}-6\,c_3^2\,\partial_x^{-1} \{ (\partial_x^{-1} v_{\delta})\,\Pi_S[v_{\delta}\,h]\}\\
&+6 c_3^2\,v_{\delta}\,\Pi_S[(\partial_x^{-1} v_{\delta})(\partial_x^{-1} h)]+\frac{2}{3}c_2^2\,v_{\delta}\,\partial_{xxx}\Pi_S[(\partial_x^{-1} v_{\delta})\,h].
\end{aligned}
\end{equation}
In conclusion, we have the following proposition.
\begin{prop}
Assume \eqref{IpotesiPiccolezzaIdelta}. Then the Hamiltonian operator $\mathcal{L}_{\omega}$, for all $h\in H^s_{S^{\perp}}(\mathbb{T}^{\nu+1})$, has the form
\begin{equation}\label{Lomega}
\mathcal{L}_{\omega} h:=\omega\cdot \partial_{\varphi} h-\partial_x K_{0 2} h=\Pi_S^{\perp} (\omega\cdot \partial_{\varphi} h+\partial_{xx} (a_1\,h_x)+\partial_x(a_0 h)-\varepsilon^2 \partial_x \mathcal{R}_2 h-\partial_x \mathcal{R}_* h)
\end{equation}
where $\mathcal{R}_2$ is defined in \eqref{R2},
\begin{equation}\label{Rstar}
{R}_*:=\tilde{\mathcal{R}}_{>2}+R_{H_{\geq 5}}(T_{\delta})+R(\psi),
\end{equation}
 with $R(\psi)$ defined in Lemma \ref{ApplicazioneEgorov}, the functions
\begin{align}
a_1:=&1+6 c_1 \,(\Phi_B(T_{\delta}))_x+2\,c_2\,\Phi_B(T_{\delta})+12 c_4\,(\Phi_B(T_{\delta}))_x^2+3 c_5\,\partial_x[\Phi_B(T_{\delta})^2]+\label{DefA1}\\ \notag
&+2 c_6\,\Phi_B(T_{\delta})^2-r_1(T_{\delta}),\\
a_0:=&2 c_2\,(\Phi_B(T_{\delta}))_{xx}-6 c_3\,\Phi_B(T_{\delta})+3 c_5\,\partial_x [(\Phi_B(T_{\delta}))_x^2]+2 c_6\, \{\Phi_B(T_{\delta})_x^2+\label{DefA0}\\ \notag
&+2\Phi_B(T_{\delta})\,(\Phi_B(T_{\delta}))_{xx} \}-12 c_7\,\Phi_B(T_{\delta})^2-r_0(T_{\delta})
\end{align}
the function $r_1$ is defined in \eqref{r1}, $r_0$ in \eqref{r0}, $T_{\delta}$ and $v_{\delta}$ in \eqref{Tdelta}.\\
Furthermore, we have, for some $\sigma:=\sigma(\nu, \tau)>0$,
\begin{align}
&\lVert a_1-1 \rVert_s^{Lip(\gamma)}\le_s \varepsilon\, (1+\lVert \mathfrak{I}_{\delta} \rVert^{Lip(\gamma)}_{s+\sigma}), &\lVert \partial_i a_1 [\hat{\imath}] \rVert_s\le_s \varepsilon(\lVert i\rVert_{s+\sigma}+ \lVert \mathfrak{I}_{\delta} \rVert_{s+\sigma}\lVert i \rVert_{s_0+\sigma}),\label{a1}\\
&\lVert a_0 \rVert^{Lip(\gamma)}_s \le_s  \varepsilon\, (1+\lVert \mathfrak{I}_{\delta} \rVert^{Lip(\gamma)}_{s+\sigma}),  &\lVert \partial_i a_0 [\hat{\imath}] \rVert_s\le_s \varepsilon(\lVert i\rVert_{s+\sigma}+ \lVert \mathfrak{I}_{\delta} \rVert_{s+\sigma}\lVert i \rVert_{s_0+\sigma}),\label{a0}
\end{align}
where $\mathfrak{I}_{\delta}(\varphi):=(\theta_0(\varphi)-\varphi, y_{\delta}(\varphi), z_0(\varphi))$ corresponds to $T_{\delta}$. The remainder $\mathcal{R}_2$ has the form \eqref{FinitedimensionalRemainder} with
\begin{equation}\label{DecayR2}
\lVert g_j \rVert_s^{Lip(\gamma)}+\lVert \chi_j \rVert_s^{Lip(\gamma)}\le_s 1+\lVert \mathfrak{I}_{\delta} \rVert_{s+\sigma}^{Lip(\gamma)}, \quad \lVert \partial_i g_j [\hat{\imath}] \rVert_s+\lVert \partial_i\chi_j [\hat{\imath}] \rVert_s\le_s \lVert \hat{\imath} \rVert_{s+\sigma}+\lVert \mathfrak{I}_{\delta} \rVert_{s+\sigma} \lVert \hat{\imath} \rVert_{s_0+\sigma}
\end{equation}
and also $\mathcal{R}_*$ has the form \eqref{FinitedimensionalRemainder} with
\begin{align}
& \lVert g^*_j \rVert_s^{Lip(\gamma)}\lVert \chi_j^* \rVert_{s_0}^{Lip(\gamma)}+\lVert g^*_j \rVert_{s_0}^{Lip(\gamma)}\lVert \chi_j^* \rVert_{s}^{Lip(\gamma)}\le_s \varepsilon^3+\varepsilon^{1+b}\lVert \mathfrak{I}_{\delta} \rVert_{s+\sigma}^{Lip(\gamma)}, \label{DecayR*}\\[2mm]
&\lVert \partial_i g^*_j [\hat{\imath}] \rVert_s\lVert \chi_j^* \rVert_{s_0}+\lVert \partial_i g^*_j [\hat{\imath}] \rVert_{s_0}\lVert \chi_j^* \rVert_{s}+\lVert g^*_j \rVert_{s_0} \lVert \partial_i \chi^*_j \rVert_{s}+\lVert g^*_j \rVert_{s} \lVert \partial_i \chi^*_j \rVert_{s_0} \\ \notag
&\le_s \varepsilon^{1+b} \lVert \hat{\imath} \rVert_{s+\sigma}+\varepsilon^{2 b-1}\lVert \mathfrak{I}_{\delta} \rVert_{s+\sigma} \lVert \hat{\imath} \rVert_{s_0+\sigma}.
\end{align}
\end{prop}
The bounds \eqref{DecayR2} and \eqref{DecayR*} imply, by Lemma \ref{Lemma7.3}, estimates for the $s$-decay norms of $\mathcal{R}_2$ and $\mathcal{R}_*$.\\
The linearized operator $\mathcal{L}_{\omega}:=\mathcal{L}_{\omega}(\omega, i_{\delta}(\omega))$ depends on the parameter $\omega$ both directly and also through the dependence on the embedded torus $i_{\delta}(\omega)$.
The estimates on the partial derivative respect to $i$ (see \eqref{i}) allow us to control, along the Nash-Moser iteration, the Lipschitz variation of the eigenvalues of $\mathcal{L}_{\omega}$ with respect to $\omega$ and the approximate solution $i_{\delta}$.

\section{Reduction of the linearized operator in the normal\\ directions}

The goal of this section is to conjugate the Hamiltonian linear operator $\mathcal{L}_{\omega}$ in \eqref{Lomega} to a constant coefficients linear operator $\mathcal{L}_{\infty}$. For this purpose, we shall apply the same kind of symplectic transformations used in \cite{KdVAut}, whose aim is to diagonalize the operator $\mathcal{L}_{\omega}$ up to a bounded remainder $\mathcal{R}_6$ (see \eqref{L6}). This one has to satisfy the smallness condition \eqref{PiccolezzaperKamred} in order to initialize the KAM reducibility scheme of Theorem \ref{Reducibility}, that completes the diagonalization procedure.\\
The size of all these transformations will be greater than the ones used in \cite{KdVAut} (see Section $8$ in \cite{KdVAut}) and, as a consequence, some non perturbative terms will be modified by them. Thus, in order to prove \eqref{PiccolezzaperKamred} we will have to overcome two main difficulties: (a) computing the terms of order $\varepsilon$ and $\varepsilon^2$ after each transformation, since we need to normalize them through the Birkhoff steps of Section $8.5$ and $8.6$,
(b) providing optimal estimates for the transformations and, consequently, for the remainder $\mathcal{R}_6$ (see \eqref{L6}).\\

Consider 
\begin{equation}\label{vSegnato}
\overline{v}(\varphi, x):=\sum_{j\in S} \sqrt{\lvert j \rvert \xi_j}\,e^{\mathrm{i} \mathtt{l}(j)\cdot \varphi}\,e^{\mathrm{i} j x}
\end{equation}
and $\mathtt{l}\colon S\to \mathbb{Z}^{\nu}$ is the odd injective map
\begin{equation}\label{mathttL}
\mathtt{l}\colon S \to \mathbb{Z}^{\nu}, \quad \mathtt{l}(\overline{\jmath}_i):=\mathtt{e}_i, \quad \mathtt{l}(-\overline{\jmath}_i)=-\mathtt{l}(\overline{\jmath}_i)=-\mathtt{e}_i, \quad i=1,\dots, \nu,
\end{equation}
denoting by $\mathtt{e}_i=(0, \dots, 1, \dots, 0)$ the $i$-th vector of the canonical basis of $\mathbb{R}^{\nu}$. We observe that 
\begin{equation}\label{vdeltamenovsegn}
\lVert v_{\delta}-\overline{v}\rVert_s^{Lip(\gamma)}\le_s \lVert \mathfrak{I}_{\delta} \rVert_s^{Lip(\gamma)}, \qquad \lVert \partial_i (v_{\delta}-\overline{v}) [\hat{\imath}] \rVert_s\le_s \lVert \hat{\imath} \rVert_s+\lVert \mathfrak{I}_{\delta} \rVert_{s} \lVert \hat{\imath} \rVert_{s_0}.
\end{equation}
\begin{remark}\label{KerL}
The function $\overline{v}(\varphi, x)$ in \eqref{vSegnato} corresponds to the torus $(\varphi, 0, 0)$ after the transformation $A_{\varepsilon}$ defined in \eqref{Aepsilon}. In particular, this torus is invariant under the flow of the integrable Hamiltonian $\varepsilon^{-2 b} \tilde{h}\circ A_{\varepsilon}$ (recalling \eqref{HamBif}), which preserves the momentum. Hence, the square of the $L^2$ norm of $\overline{v}$ is independent of the time $\varphi$, as we can deduce by the properties of the map $\mathtt{l}$ defined in \eqref{mathttL}.\\  
We shall expand the coefficients of the linearized operator at $y=z=0$ to get the bounds on the transformations defined along this section, thus we will frequently use the inequalities \eqref{vdeltamenovsegn} and the assumption \eqref{IpotesiPiccolezzaIdelta}. Moreover, we will use the fact that $\overline{v}$ satisfies the equation $L_{\overline{\omega}}=0$, where $\overline{\omega}$ is the vector of the linear frequencies (see \eqref{LinearFreq}) and $L_{\omega}:=\omega\cdot\partial_{\varphi}+\partial_{xxx}$.
\end{remark}
\begin{remark}\label{Domegamenoomegasegnato}
We recall that $\omega=\overline{\omega}+O(\varepsilon^2)$, see for instance \eqref{Frequency-AmplitudeMAP}. Moreover, note that $\mathcal{D}_{\omega}\overline{v}=\mathcal{D}_{\overline{\omega}}\overline{v}+\mathcal{D}_{\omega-\overline{\omega}}\overline{v}$ and
\[
\mathcal{D}_{\omega-\overline{\omega}} \overline{v} = \sum_{j\in S} \mathrm{i} ( \omega-\overline{\omega} ) \cdot \mathtt{l}(j)\, \sqrt{\lvert j \rvert \xi_j} \,e^{\mathrm{i} \mathtt{l}(j)\cdot \varphi}\,e^{\mathrm{i} j x}.
\]
Then $\lVert \mathcal{D}_{\omega-\overline{\omega}} \overline{v}\rVert_s^{Lip(\gamma)}\le C \varepsilon^2$ and $\mathcal{D}_{\omega-\overline{\omega}} \overline{v}$ has zero spatial average.
\end{remark}

We expand in powers of $\varepsilon$ the coefficients $a_0$ and $a_1$ in \eqref{DefA0} and \eqref{DefA1} as
\begin{align}\label{SviluppoA1A0}
&a_0=\varepsilon a_{0, 1}+\varepsilon^2 a_{0, 2}+\mathtt{R}_{a_0},\qquad a_1-1=\varepsilon a_{1, 1}+\varepsilon^2 a_{1, 2}+\mathtt{R}_{a_1},
\end{align}
where
\begin{align*}
&a_{0, 1}:=2 c_2\,\overline{v}_{xx}-6 c_3\,\overline{v}, \qquad a_{1, 1}:=6 c_1 \overline{v}_x+2 c_2 \overline{v}_{xx},\\
&a_{0, 2}:=2 c_2\,(\Psi_2(\overline{v}))_{xx}-6 c_3\,\Psi_2(\overline{v})+3 c_5 \partial_x(\overline{v}_x^2)+2 c_6\{ \overline{v}_x^2+2 \overline{v}\overline{v}_{xx}\}-12 c_7 \overline{v}^2,\\
&a_{1, 2}:=6 c_1 (\Psi_2(\overline{v}))_x+2 c_2 \Psi_2(\overline{v})+12 c_4 \overline{v}_x^2+3 c_5 \partial_x (\overline{v}^2)+2 c_6 \overline{v}^2
\end{align*}
and, by \eqref{vdeltamenovsegn}, $\lVert \mathtt{R}_{a_k} \rVert^{Lip(\gamma)}_s\le \varepsilon^3+\varepsilon \lVert \mathfrak{I}_{\delta} \rVert_{s+\sigma}$, for some $\sigma>0$.

\subsection{Space reduction at the order $\partial_{xxx}$}
First we conjugate $\mathcal{L}_{\omega}$ in \eqref{Lomega} to an operator $\mathcal{L}_1$ whose coefficient in front of $\partial_{xxx}$ is independent on the space variable $x$. Because of the Hamiltonian structure, the terms $O(\partial_{xx})$ will be simultaneously eliminated.\\
We look for a $\varphi$-dependent family of symplectic diffeomorphisms $\Phi(\varphi)$ of $H_S^{\perp}$ which differ from
\begin{equation}\label{Aperp}
\mathcal{A}_{\perp}:=\Pi_S^{\perp} \mathcal{A} \Pi_S^{\perp}, \quad (\mathcal{A} h)(\varphi, x):=(1+\beta_x(\varphi, x))\,h(\varphi, x+\beta(\varphi, x)),
\end{equation}
up to a small ``finite dimensional'' remainder, see \eqref{FormaRestoTrasp}.\\
If $\lVert \beta \rVert_{W^{1, \infty}}<\frac{1}{2}$ then $\mathcal{A}$ is invertible and its inverse and adjoin map are
\begin{equation}
(\mathcal{A}^{-1} h)(\varphi, y):=(1+\tilde{\beta}_y (\varphi, y))\,h(\varphi, y+\tilde{\beta}(\varphi, y)), \quad (\mathcal{A}^T h)(\varphi, y)=h(\varphi, y+\tilde{\beta}(\varphi, y))
\end{equation}
For each $\varphi\in\T^{\nu}$, $\mathcal{A}(\varphi)$ is a symplectic transformation of the phase space, see Remark $3.3$ in \cite{Airy}, but the restricted map $\mathcal{A}_{\perp}(\varphi)$ is not.\\
In order to find a symplectic diffeomorphism near $\mathcal{A}_{\perp}$ first we observe that $\mathcal{A}_{\perp}$ is the time$-1$ flow map of the linear Hamiltonian PDE
\begin{equation}\label{transportEq}
\partial_{\tau} u=\partial_x (b(\varphi, \tau, x) u), \quad b(\varphi, \tau, x):=\frac{\beta(\varphi, x)}{1+\tau \beta_x(\varphi, x)}.
\end{equation}
The equation \eqref{transportEq} is a linear transport equation, whose characteristic curves are the solutions of the ODE
\[
\frac{d}{d \tau} x=-b(\varphi, \tau, x).
\]
As in \cite{KdVAut}, we define a symplectic map $\Phi$ of $H_S^{\perp}$ as the time$-1$ flow of the Hamiltonian PDE
\begin{equation}\label{HPDEtransport}
\partial_{\tau} u=\Pi_S^{\perp} \partial_x (b(\tau, x) u)=\partial_x (b(\tau, x) u)-\Pi_S \partial_x (b(\tau, x) u), \quad u\in H_S^{\perp}
\end{equation}
generated by the quadratic Hamiltonian $\frac{1}{2}\int_{\T} b(\tau, x) u^2\,dx$ restricted to $H_S^{\perp}$. The flow of \eqref{HPDEtransport} is well defined in the Sobolev spaces $H_{S^{\perp}}^s(\mathbb{T}_x)$ for $b(\tau, x)$ smooth enough, by standard theory of linear hyperbolic PDE's. We obtained a symplectic diffeomorphism $\Phi$ that differs from $\mathcal{A}_{\perp}$ by a  ``finite dimensional'' remainder of small size, more precisely, of size $O(\beta)$.
\begin{lem}{(Lemma $8.2$ in \cite{KdVAut})}\label{LemmaTrasporto}
For $\lVert \beta \rVert_{W^{s_0+1,\infty}}$ small, there exists an invertible symplectic transformation $\Phi=\mathcal{A}_{\perp}+\mathcal{R}_{\Phi}$ of $H^s_{S^{\perp}}$, where $\mathcal{A}_{\perp}$ is defined in \eqref{Aperp} and $\mathcal{R}_{\Phi}$ is a ``finite dimensional'' remainder
\begin{equation}\label{FormaRestoTrasp}
\mathcal{R}_{\Phi} h=\sum_{j\in S}\int_0^1 (h, g_j(\tau))_{L^2(\mathbb{T})} \chi_j(\tau)\,d\tau+\sum_{j\in S} (h, \psi_j)_{L^2(\mathbb{T})} e^{\mathrm{i} j x}
\end{equation}
for some functions $\chi_j(\tau), g_j (\tau), \psi_j(\tau)\in H^s$ satysfying for all $\tau\in [0, 1]$
\begin{equation}\label{JohnQ}
\lVert \psi_j \rVert_s+\lVert g_j(\tau) \rVert_s\le_s \lVert \beta \rVert_{W^{s+2,\infty}}, \quad \lVert \chi_j (\tau) \rVert_s\le_s 1+\lVert \beta \rVert_{W^{s+1, \infty}}.
\end{equation}
Moreover 
\begin{equation}
\lVert \Phi h \rVert_s+\lVert \Phi^{-1} h \rVert_s\le_s \lVert h \rVert_s+\lVert \beta \rVert_{W^{s+2, \infty}} \lVert h \rVert_{s_0} \quad \forall h\in H^s_{S^{\perp}}.
\end{equation}
\end{lem}
We conjugate $\mathcal{L}_{\omega}$ in \eqref{Lomega} via the symplectic map $\Phi=\mathcal{A}_{\perp}+\mathcal{R}_{\Phi}$ of Lemma \eqref{LemmaTrasporto}. Using the splitting $\Pi_S^{\perp}=\mathrm{I}-\Pi_S$, we compute 
\begin{equation}\label{8.22}
\mathcal{L}_{\omega} \Phi=\Phi \mathcal{D}_{\omega}+\Pi_S^{\perp} \mathcal{A} (b_3 \partial_{yyy}+b_2 \partial_{yy}+b_1 \partial_y+b_0)\Pi_S^{\perp}+\mathcal{R}_{\mathit{I}},
\end{equation}
where the coefficients are
\begin{align}
&b_3(\varphi, y):=\mathcal{A}^T [a_1\,(1+\beta_x)^3]  \qquad b_2(\varphi, y):=\mathcal{A}^T [2 (a_1)_x(1+\beta_x)^2+6\,a_1\,\beta_{xx}(1+\beta_x)]\\
&b_1(\varphi, y):=\mathcal{A}^T \left[ (\mathcal{D}_{\omega} \beta)+3\,a_1\,\frac{\beta_{xx}^2}{1+\beta_x}+4\,a_1\,\beta_{xxx}+6\,(a_1)_x \beta_{xx}+(a_1)_{xx}(1+\beta_x)+a_0 (1+\beta_x)  \right]\\
&b_0(\varphi, y):=\mathcal{A}^T \left[ \frac{(\mathcal{D}_{\omega} \beta_x)}{1+\beta_x}+a_1\,\frac{\beta_{xxxx}}{1+\beta_x}+2 (a_1)_x \frac{\beta_{xxx}}{1+\beta_x}+(a_1)_{xx}\,\frac{\beta_{xx}}{1+\beta_x}+a_0\,\frac{\beta_{xx}}{1+\beta_x}+(a_0)_x \right]
\end{align}
and the remainder 
\begin{equation}\label{R_I}
\begin{aligned}
\mathcal{R}_{\mathit{I}}:=&-\Pi_S^{\perp} \partial_x (\varepsilon^2 \mathcal{R}_2+\mathcal{R}_*)\,\mathcal{A}_{\perp}-\Pi_S^{\perp} (a_1 \partial_{xxx}+2 (a_1)_x \partial_{xx}+((a_1)_{xx}+a_0) \partial_x+(a_0)_x)\Pi_S \mathcal{A} \Pi_S^{\perp}+\\
&+[\mathcal{D}_{\omega}, \mathcal{R}_{\Phi}]+(\mathcal{L}_{\omega}-\mathcal{D}_{\omega}) \mathcal{R}_{\Phi}.
\end{aligned}
\end{equation}
The commutator $[\mathcal{D}_{\omega}, \mathcal{R}_{\Phi}]$ has the form \eqref{FormaRestoTrasp} with $\mathcal{D}_{\omega} g_j$ or $\mathcal{D}_{\omega} \chi_j, \mathcal{D}_{\omega} \psi_j$ instead of $\chi_j, g_j, \psi_j$ respectively. Also the last term $(\mathcal{L}_{\omega}-\mathcal{D}_{\omega}) \mathcal{R}_{\Phi}$ in \eqref{R_I} has the form \eqref{FormaRestoTrasp} (note that $\mathcal{L}_{\omega}-\mathcal{D}_{\omega}$ does not contain derivatives with respect to $\varphi$). By \eqref{8.22}, and decomposing $\mathrm{I}=\Pi_S+\Pi_S^{\perp}$, we get
\begin{align}
&\mathcal{L}_{\omega} \Phi=\Phi (\mathcal{D}_{\omega}+b_3\partial_{yyy}+b_2 \partial_{yy}+b_1 \partial_y +b_0) \Pi_S^{\perp}+\mathcal{R}_{\mathit{II}},\label{mapPHI}\\
&\mathcal{R}_{\mathit{II}}:=\{\Pi_S^{\perp}(\mathcal{A}-\mathrm{I})\Pi_S-\mathcal{R}_{\Phi}\}(b_3 \partial_{yyy}+b_2 \partial_{yy}+b_1 \partial_y +b_0)\Pi_S^{\perp}+\mathcal{R}_{\mathit{I}}.\label{R_II}
\end{align}
In order to solve the equation
\[
b_{3}(\varphi, y)=b_3(\varphi)
\]
for some function $b_3(\varphi)$, so that the coefficient in front of $\partial_{xxx}$ depends only on $\varphi$, we choose the function $\beta=\beta(\varphi, x)$ such that
\begin{equation}\label{EqperB3}
a_1(\varphi, x) (1+\beta_x(\varphi, x))^3=b_3(\varphi),
\end{equation}
where we used that $\mathcal{A}^T[b_3(\varphi)]=b_3(\varphi)$. The only solution of \eqref{EqperB3} with zero space average is 
\begin{equation}\label{ChoiceforB3}
\beta:=\partial_x^{-1} \rho_0, \quad \rho_0:=b_3(\varphi)^{\frac{1}{3}} (a_1(\varphi, x))^{-\frac{1}{3}}-1, \quad b_3(\varphi):=\left( \frac{1}{2\pi} \int_{\T} (a_1(\varphi, x))^{-\frac{1}{3}}\,dx \right)^{-3}.
\end{equation}
Applying the symplectic map $\Phi^{-1}$ in \eqref{mapPHI} we obtain the Hamiltonian operator
\begin{equation}\label{L1}
\mathcal{L}_1:=\Phi^{-1} \mathcal{L}_{\omega} \Phi=\Pi_S^{\perp} (\omega\cdot \partial_{\varphi} + b_3(\varphi) \partial_{yyy}+b_1 \partial_y+b_0)\Pi_S^{\perp}+\mathfrak{R}_1
\end{equation}
where $\mathfrak{R}_1:=\Phi^{-1} \mathcal{R}_{\mathit{II}}$. We used that, by the Hamiltonian nature of $\mathcal{L}_1$, the coefficient $b_2=2\,(b_3)_y$ and so, by the choice \eqref{ChoiceforB3}, we have $b_2=2\,(b_3)_y=0$. 
\begin{lem}{(Lemma $8.3$ in \cite{KdVAut})}
The operator $\mathfrak{R}_1$ in \eqref{L1} has the form \eqref{VeraFinDimForm}.
\end{lem}

In the proofs of the estimates for the transformations and the coefficients, we will always use the index $\sigma$ to denote a certain loss of derivatives, since we do not need to know exactly the total amount of this loss. This, in fact, involves only the regularity required for the Hamiltonian nonlinearity $f(x, u, u_x)$ in \eqref{KdVHamiltonian}.
\begin{lem}\label{StimaPasso1}
There is $\sigma:=\sigma (\tau, \nu)>0$ such that, for $k=0, 1$,
\begin{align}
& \lVert \beta \rVert_s^{Lip(\gamma)}\le_s \varepsilon\,(1+\lVert \mathfrak{I}_{\delta} \rVert_{s+\sigma}^{Lip(\gamma)})  &\lVert \partial_i \beta [\hat{\imath}] \rVert_s\le_s \varepsilon (\lVert \hat{\imath} \rVert_{s+\sigma}+\lVert \mathfrak{I}_{\delta} \rVert_{s+\sigma} \lVert \hat{\imath} \rVert_{s_0+\sigma}) \label{beta}\\
& \lVert b_3-1 \rVert_s^{Lip(\gamma)}\le_s \varepsilon^2\,(1+\lVert \mathfrak{I}_{\delta} \rVert_{s+\sigma}^{Lip(\gamma)})  &\lVert \partial_i b_3 [\hat{\imath}] \rVert\le_s \varepsilon^2(\lVert \hat{\imath} \rVert_{s+\sigma}+\lVert \mathfrak{I}_{\delta} \rVert_{s+\sigma} \lVert \hat{\imath} \rVert_{s_0+\sigma}) \label{b3}\\
& \lVert b_k \rVert^{Lip(\gamma)}_s\le_s \varepsilon (1+\lVert \mathfrak{I}_{\delta} \rVert_{s+\sigma}^{Lip(\gamma)}) & \lVert \partial_ i b_k [\hat{\imath}] \rVert_s\le_s \varepsilon (\lVert \hat{\imath} \rVert_{s+\sigma}+\lVert \mathfrak{I}_{\delta} \rVert_{s+\sigma} \lVert \hat{\imath} \rVert_{s_0+\sigma}).\label{b0b1}
\end{align}
The transformations $\Phi, \Phi^{-1}$ satisfy
\begin{align}
&\lVert \Phi^{\pm 1} h \rVert_s^{Lip(\gamma)}\le_s \lVert h \rVert_{s+1}^{Lip(\gamma)}+\lVert \mathfrak{I}_{\delta} \rVert^{Lip(\gamma)}_{s+\sigma} \lVert h \rVert_{s_0+1}^{Lip(\gamma)}\label{8.39}\\
&\lVert \partial_i (\Phi^{\pm 1} h) [\hat{\imath}] \rVert_s\le_s \lVert h \rVert_{s+\sigma} \lVert \hat{\imath} \rVert_{s_0+\sigma}+\lVert h \rVert_{s_0+\sigma} \lVert \hat{\imath} \rVert_{s+\sigma}+\lVert \mathfrak{I}_{\delta} \rVert_{s+\sigma} \lVert h \rVert_{s_0+\sigma} \lVert \hat{\imath} \rVert_{s_0+\sigma}.\label{8.40}
\end{align}
Moreover the remainder $\mathcal{R}_*$ has the form \eqref{VeraFinDimForm} where the functions $\chi_j(\tau), g_j(\tau)$ satisfy the estimates \eqref{DecayR*} uniformly in $\tau\in [0, 1]$.
\end{lem}

\begin{proof}
\textit{Estimate} \eqref{b3}: Consider the functions $g(t)=(1+t)^{-\frac{1}{3}}$ and $\Upsilon(t)=(1+t)^{-3}$, analytic in a small neighbourhood of the origin.
Then we have
\begin{equation}\label{Michela}
b_3-1=\Upsilon(M_x[g(a_1-1)-g(0)])-\Upsilon(0).
\end{equation}
By the mean value theorem, $\lVert b_3-1\rVert_s\le_s \lVert M_x[g(a_1-1)-g(0)]\rVert_s$. By Taylor expansion, we get
\begin{equation}\label{LeIene}
M_x[g(a_1-1)-g(0)]=g'(0)M_x[a_1-1]+\int_{\T}\int_0^1 (1-s) \,g''(s(a_1-1))\,(a_1-1)^2\,ds\,dx
\end{equation}
and we note that, by Remark \ref{SviluppoA1A0},
\[
M_x[a_1-1]=\varepsilon^2 M_x[a_{1, 2}]+M_x[\mathtt{R}_{a_1}].
\]
Moreover, $\lVert M_x[\mathtt{R}_{a_1}]\rVert_s\le_s \varepsilon^3+\varepsilon^{2 b}\lVert \mathfrak{I}_{\delta}\rVert_{s+\sigma}$, because $M_x[v_{\delta}-\overline{v}]=M_x[\tilde{q}]=0$ and $\mathtt{R}_{a_1}$ contains terms like $\varepsilon^2(v_{\delta}^2-\overline{v}^2)$ and cubic in the $x$-derivatives of $v_{\delta}$.\\
The second addend in the right hand side of \eqref{LeIene} can be estimated by $\varepsilon^2 (1+\lVert \mathfrak{I}_{\delta}\rVert_{s+\sigma})$. Hence
\begin{equation}\label{PulpFiction}
\lVert b_3-1 \rVert_s\le_s \varepsilon^2(1+\lVert \mathfrak{I}_{\delta}\rVert_{s+\sigma}).
\end{equation} 
Now we consider the partial derivative respect to the variable $i$ (see \eqref{i}) of $b_3$, namely
\[
\partial_i b_3[\hat{\imath}]=\Upsilon'(M_x[g(a_1-1)-g(0)])\,M_x[g'(a_1-1)\,\partial_i a_1[\hat{\imath}]].
\]
The derivatives of the functions $g$ and $\Upsilon$, for $\varepsilon$ small enough, are approximately $1$. Therefore, the estimate
\begin{equation}\label{JackieBrown}
\lVert \partial_i b_3 [\hat{\imath}] \rVert_s\le_s \varepsilon^2 (\lVert \hat{\imath} \rVert_{s+\sigma}+\lVert \mathfrak{I}_{\delta}\rVert_{s+\sigma}\lVert \hat{\imath}\rVert_{s_0+\sigma})
\end{equation}
derived from the estimate on $M_x[\partial_i a_1[\hat{\imath}]]$ and the fact that $M_x[\partial_i \overline{v}[\hat{\imath}]]=0$. By \eqref{PulpFiction} and \eqref{JackieBrown} we conclude.\\

\noindent\textit{Estimate \eqref{beta}}: Consider the functions $\phi(t):=(1+t)^{-1}$ and $g(t):=(1+t)^{-\frac{1}{3}}$. Recalling that $\beta_x=(b_3^{-1} a_1)^{\frac{1}{3}}-1$, we have
\[
\beta_x=g^{-1}(b_3^{-1}\,a_1-1)-g^{-1}(0)
\quad \mbox{and} \quad
b_3^{-1}\,a_1-1=a_1\, (\phi(b_3-1)-\phi(0))+(a_1-1).
\]
Then, by \eqref{TameProduct},
\begin{align*}
\lVert \beta_x \rVert_s &\le_s \lVert \phi(b_3-1)-\phi(0) \rVert_s \lVert a_1 \rVert_{s_0}+\lVert\phi(b_3-1)-\phi(0) \rVert_{s_0} \lVert a_1 \rVert_{s}+\lVert a_1-1 \rVert_s\\
&\le_s \lVert b_3-1 \rVert_{s+\sigma}+\lVert b_3-1 \rVert_{s_0+\sigma}\lVert a_1 \rVert_s+\lVert a_1-1 \rVert_s\le_s \varepsilon (1+\lVert \mathfrak{I}_{\delta} \rVert_{s+\sigma}).
\end{align*}

\noindent\textit{Estimate \eqref{b0b1}}: By \eqref{a0}, \eqref{a1}, \eqref{beta} we get the estimates \eqref{b0b1}.\\
For the estimates \eqref{8.39}, \eqref{8.40} on $\Phi, \Phi^{-1}$ we apply Lemma \ref{LemmaTrasporto} and the estimate \eqref{beta} for $\beta$. We estimate the remainder $\mathcal{R}_*$ using \eqref{R_I}, \eqref{R_II} and \eqref{DecayR*}.
\end{proof}

\subsection{Terms of order $\varepsilon$ and $\varepsilon^2$}
The diffeomorphism of the torus $\Phi=\mathcal{A}_{\perp}+\mathcal{R}_{\Phi}$ defined in Lemma \ref{LemmaTrasporto} is, by \eqref{JohnQ} and \eqref{beta}, of the form $\mathrm{I}+O(\varepsilon)$, hence, the terms $O(\varepsilon^2)$ of $\mathcal{L}_{\omega}$ are modified by it.\\ 
From now on, the transformations we shall apply to reduce the linearized operator $\mathcal{L}_{\omega}$ to a constant coefficient operator will be $\mathrm{I}+O(\varepsilon^d)$ with $d>1$, hence the terms of order $\varepsilon, \varepsilon^2$ will not be changed anymore.\\
In this section, our goal is to identify them in view of the linear Birkhoff steps of Section $8.5$ and $8.6$.\\

We have to put in evidence the terms $O(\varepsilon), O(\varepsilon^2)$ of $b_0, b_1, b_3$ in \eqref{L1} and the ones in the remainder $\mathfrak{R}_1$ defined in \eqref{8.34}.

\subsubsection*{Coefficients $b_k$}
First, we note that $b_k=\mathcal{A}^T \alpha_k=\alpha_k+(\mathcal{A}^T-\mathrm{I})\alpha_k,\, k=0, 1$, where
\begin{align}
&\mathit{\alpha}_1:= (\mathcal{D}_{\omega} \beta)+3\,a_1\,\frac{\beta_{xx}^2}{1+\beta_x}+4\,a_1\,\beta_{xxx}+6\,(a_1)_x \beta_{xx}+(a_1)_{xx}(1+\beta_x)+a_0 (1+\beta_x),\\
&\mathcal{\alpha}_0:=\frac{(\mathcal{D}_{\omega} \beta_x)}{1+\beta_x}+a_1\,\frac{\beta_{xxxx}}{1+\beta_x}+2 (a_1)_x \frac{\beta_{xxx}}{1+\beta_x}+(a_1)_{xx}\,\frac{\beta_{xx}}{1+\beta_x}+a_0\,\frac{\beta_{xx}}{1+\beta_x}+(a_0)_x.
\end{align}
By \eqref{DefA1}, \eqref{ChoiceforB3}, we have
\begin{equation}\label{BetaVero}
\begin{aligned}
\beta=&-2\,c_1 \Phi_B(T_{\delta})-\frac{2}{3}\,c_2 \partial_x^{-1}[\Phi_B(T_{\delta})]-4\,c_4 \partial_x^{-1}[\Phi_B(T_{\delta})_x^2]-c_5\,\pi_0[\Phi_B(T_{\delta})^2]-\frac{2}{3}\,c_6 \partial_x^{-1}[\Phi_B(T_{\delta})^2]\\
&+8\,c_1^2 \partial_x^{-1}[\Phi_B(T_{\delta})_x^2]+\frac{8}{9}\,c_2^2 \partial_x^{-1}[\Phi_B(T_{\delta})^2]+\frac{8}{3}\,c_1 c_2 \pi_0 [\Phi_B(T_{\delta})^2]+\mathtt{R}
\end{aligned}
\end{equation}
where, by \eqref{Qtilda}, $\lVert \mathtt{R} \rVert_s^{Lip(\gamma)}\le_s \varepsilon^3+\varepsilon^{b}\lVert \mathfrak{I}_{\delta} \rVert^{Lip(\gamma)}_{s+\sigma}$.
Then we write $\beta=\varepsilon \,\beta_1+\varepsilon^2 \,\beta_2+\mathtt{R}_{\beta}$, where
\begin{equation}\label{betaOrdiniEpsilon}
\begin{aligned}
\beta_1:&=-2 c_1 \overline{v}-\frac{2}{3} c_2 \partial_x^{-1}(\overline{v}),\\
\beta_2:&=-2 c_1 \Psi_2(\overline{v})-\frac{2}{3} c_2 \partial_x^{-1}(\Psi_2(\overline{v}))-4 c_4 \partial_x^{-1}(\overline{v}_x^2)-c_5 \pi_0 [\overline{v}^2]\\
&-\frac{2}{3} c_6 \partial_x^{-1}[\overline{v}^2]+8 c_1^2 \partial_x^{-1}[\overline{v}_x^2]+\frac{8}{9} c_2^2 \partial_x^{-1}[\overline{v}^2]+\frac{8}{3} c_1 c_2 \pi_0 [\overline{v}^2]
\end{aligned}
\end{equation}
and $\mathtt{R}_{\beta}$ is defined by difference and satisfies, by \eqref{vdeltamenovsegn},
\[
\lVert \mathtt{R}_{\beta} \rVert_s^{Lip(\gamma)}\le_s\varepsilon^3+\varepsilon \lVert \mathfrak{I}_{\delta} \rVert_{s+\sigma}^{Lip(\gamma)}, \quad \lVert \partial_i \mathtt{R}_{\beta} [\hat{\imath}] \rVert_s\le_s \varepsilon (\lVert \hat{\imath} \rVert_{s+\sigma}+\lVert \mathfrak{I}_{\delta} \rVert_{s+\sigma} \lVert \hat{\imath} \rVert_{s_0+\sigma}).
\]

Now we can develop $\alpha_0$ and $\alpha_1$ in powers of $\varepsilon$.
By \eqref{DefA1}, \eqref{DefA0}, \eqref{BetaVero} and by Remark \ref{KerL} we obtain $\alpha_1:=\varepsilon \alpha_{1,1}+\varepsilon^2 \alpha_{1,2}+\mathtt{R}_1$ and $\alpha_0=\varepsilon \alpha_{0,1}+\varepsilon^2 \alpha_{0, 2}+\mathtt{R}_0$, where
\begin{equation}\label{alpha1}
\begin{aligned}
&\alpha_{1,1}=2\,c_2 \overline{v}_{xx}-6\,c_3 \overline{v},\\
&\alpha_{1, 2}=L_{\overline{\omega}}[\beta_2]+\frac{8}{3} (\beta_2)_{xxx}-\frac{41}{3}\,\partial_x [ (\beta_1)_x\,(\beta_1)_{xx}]+a_{0, 2}+a_{0, 1}\,(\beta_1)_x,
\end{aligned}
\end{equation}
and
\begin{equation}\label{alpha0}
\begin{aligned}
&\alpha_{0, 1}=2\,c_2 \overline{v}_{xxx}-6\,c_3 \overline{v}_x,\\
&\alpha_{0, 2}=\partial_x L_{\overline{\omega}}[\beta_2]-3 \partial_x [(\beta_1)_x \,(\beta_1)_{xxx}]-3 \partial_x [(\beta_1)_{xx}^2]+a_{0, 1} (\beta_1)_{xx}+(a_{0, 2})_x.
\end{aligned}
\end{equation}
The functions $\mathtt{R}_0$ and $\mathtt{R}_1$ are defined by difference and
satisfy the following estimates
\begin{equation}
\lVert \mathtt{R}_k \rVert_s^{Lip(\gamma)}\le_s\varepsilon^3+ \varepsilon \lVert \mathfrak{I}_{\delta} \rVert_{s+\sigma}^{Lip(\gamma)}, \quad \lVert \partial_i \mathtt{R}_k [\hat{\imath}] \rVert_s\le_s \varepsilon (\lVert \hat{\imath} \rVert_{s+\sigma}+\lVert \mathfrak{I}_{\delta} \rVert_{s+\sigma} \lVert \hat{\imath} \rVert_{s_0+\sigma}), \quad k=0,1.
\end{equation}
\begin{remark}\label{ordiniEpsilonCanc}
We note that the terms $O(\varepsilon)$ generated by the Hamiltonian $ \int_{\T}(3 c_1 v_x+c_2 v)\,z_x^2 \,dx$
(see \eqref{HamiltonianeMathcal}) are cancelled by the diffeomorphism of the torus $\Phi$.
\end{remark}
\begin{remark}\label{Averages}
The averages of $\alpha_{j, k}, j=0, 1$ for $k=1$ are zero and, for $k=2$, we have
\begin{align*}
& M_x[\alpha_{1, 2}]=M_x[a_{0, 2}]+M_x[a_{0, 1}\,(\beta_1)_x]=-2 c_6\,M_x[\overline{v}_x^2]-12 c_7 M_x[\overline{v}^2]+\frac{4}{3} c_2^2 M_x[\overline{v}_x^2]+4 c_2 c_3 M_x[\overline{v}^2],\\[2mm]
&M_x[\alpha_{0, 2}]=M_x[a_{0, 1}\,(\beta_1)_{xx}]=-4 c_1 c_2 M_x[\overline{v}_{xx}^2]-12 c_1 c_3 M_x[\overline{v}_x^2].
\end{align*}
We used the fact that $\partial_{\varphi} M_x[\overline{v}^{2}]=0$, see Remark \ref{KerL}. Moreover, we note that, for a similar argument, $M_{\varphi, x}[\alpha_{k, 2}]=M_x[\alpha_{k, 2}]$, for $k=0, 1$.
\end{remark}


The transformation $\mathcal{A}^T-\mathrm{I}$ (see Section $8.1$) is of order $O(\varepsilon)$, hence it generates new terms of order $O(\varepsilon^2)$ when it is applied to ones of order $\varepsilon$. In particular, by the regularity of the function $\overline{v}(\varphi, x)$, that is at least $C^2$, we have, for $k=0, 1$, by Taylor expansion
\[
\varepsilon(\mathcal{A}^T-\mathrm{I})\alpha_{k, 1}(\varphi, y)=\varepsilon (\alpha_{k, 1}(\varphi, y+\tilde{\beta}(\varphi, y))-\alpha_{k, 1}(\varphi, y))=\varepsilon \partial_y (\alpha_{k, 1})(\varphi, y)\,\tilde{\beta}(\varphi, y)+\mathtt{R}_{\tilde{\beta}}, 
\]
where $\lVert \mathtt{R}_{\tilde{\beta}}\rVert_s\le_s \varepsilon^3 (1+\lVert \mathfrak{I}_{\delta}\rVert_{s+\sigma})$ for some $\sigma>0$.\\ 
We observe that $\tilde{\beta}(\varphi, y)=-(\mathcal{A}^T\,\beta)(\varphi, y)$ and by \eqref{betaOrdiniEpsilon} we get, for $k=0, 1$,
\begin{equation}\label{NewTermVarquadro}
\varepsilon(\mathcal{A}^T-\mathrm{I})\alpha_{k, 1}(\varphi, y)=-\varepsilon^2\,\partial_y (\alpha_{k, 1})(\varphi, y)\,\beta_1(\varphi, y)+\mathtt{R}_{\tilde{\beta}},
\end{equation}
where we have renamed $\mathtt{R}_{\tilde{\beta}}$ the terms of order $o(\varepsilon^2)$.

\subsubsection*{Remainder $\mathfrak{R}_1$}
The remaining terms of order $\varepsilon^2$ generated by the diffeomorphism of the torus $\Phi$ have the form \eqref{VeraFinDimForm}  and originate from $\mathcal{R}_{II}=\Phi \mathfrak{R}_1$ (see \eqref{R_II}). Thus we analyze the expression
\begin{equation}\label{Rii}
\begin{aligned}
\mathcal{R}_{II}:&=\Pi_S^{\perp} (\mathcal{A}-\mathrm{I}) \Pi_S [b_3 \partial_{yyy}+b_1 \partial_y+b_0]-R_{\Phi} (b_3 \partial_{yyy}+b_1 \partial_y+b_0)\\[2mm]
&-\Pi_S^{\perp} \partial_x (\varepsilon^2 \mathcal{R}_2+\mathcal{R}_*) \mathcal{A}_{\perp}-\Pi_S^{\perp} [\partial_{xx} (a_1 \partial_x)+\partial_x (a_0 \cdot)]\Pi_S \mathcal{A} \Pi_S^{\perp}+[\mathcal{D}_{\omega}, \mathcal{R}_{\Phi}]\\[2mm]
&+(\mathcal{L}_{\omega}-\mathcal{D}_{\omega}) \mathcal{R}_{\Phi}.
\end{aligned}
\end{equation}
We start from the first term in \eqref{Rii}. As we said above, the transformation $\mathcal{A}-\mathrm{I}$ has size $O(\varepsilon)$. Hence, we look for the terms $O(\varepsilon)$ of $b_3 \partial_{yyy}+b_1 \partial_y+b_0$. We have, by \eqref{b3}, $b_3=1+O(\varepsilon^2)$ and $b_k=\alpha_k+(\mathcal{A}^T-\mathrm{I})\alpha_k$ for $k=0, 1$. Thus
\[
b_3 \partial_{yyy}+b_1 \partial_y+b_0=\partial_{yyy}+\varepsilon \partial_y( \alpha_{1, 1}\, \cdot)+O(\varepsilon^2).
\]
By Taylor expansion at the point $\beta=0$, we get, for a function $u(\varphi, x)$ 
\begin{equation}\label{AmenoI}
\begin{aligned}
(\mathcal{A}-\mathrm{I}) u(\varphi, x)&=(1+\beta_x) u(\varphi, x+\beta)-u (\varphi, x)=u(\varphi, x+\beta)-u (\varphi, x)+\beta_x u(\varphi, x+\beta)=\\
&=u_x(\varphi, x) \beta(\varphi, x)+ \beta_x(\varphi, x) u(\varphi, x)+O(\beta^2)=\\
&=\varepsilon\partial_x (\beta_1 (\varphi, x)\,u(\varphi, x))+O(\varepsilon^2).
\end{aligned}
\end{equation}
Therefore we have
\begin{equation}\label{first}
\Pi_S^{\perp} (\mathcal{A}-\mathrm{I}) \Pi_S [b_3 \partial_{yyy}+b_2 \partial_{yy}+b_1 \partial_y+b_0]=\varepsilon^2\Pi_S^{\perp}[\partial_x (\beta_1\,\partial_x(\alpha_{1, 1}\,\cdot) )]+o(\varepsilon^2)
\end{equation}
Now we extract the homogeneous terms of order $\varepsilon$ from $\mathcal{R}_{\Phi}$ (see \eqref{FormaRestoTrasp}). We recall the exact expressions of $g_k$ and $\chi_k$ in \eqref{FormaRestoTrasp} refering to the proof of Lemma $8.2$ in \cite{KdVAut}.
We have
\begin{align}\label{gk}
g_k(\tau, x):=-(\Phi^{\tau})^T[b(\tau)\partial_x e^{\mathrm{i} k x}],
\end{align}
where $(\Phi^{\tau})^T$ is the flow of the adjoint PDE
\begin{equation}\label{adjointPDE}
\partial_{\tau} z=\Pi_S^{\perp} \{ b(\tau, x) \partial_x z \}, \qquad b(\tau, x)=\frac{\beta(x)}{1+\tau \beta_x(x)}=\varepsilon\beta_1+O(\varepsilon^2).
\end{equation}
This equation is well defined on $H^s_{S^{\perp}}(\T_x)$, because the function $b$ is smooth enough. By \eqref{gk} we have
\[
g_k(\tau, x)=-b(\tau)\partial_x e^{\mathrm{i} k x}+(\mathrm{I}_{H_S^{\perp}}-(\Phi^{\tau})^T)[b(\tau)\partial_x e^{\mathrm{i} k x}]
\]
and, for $z\in H^s_{S^{\perp}}(\T_x)$, by \eqref{beta} and \eqref{adjointPDE}, $\lVert (\Phi^{\tau})^T z-z \rVert_s\le_s \varepsilon C(\lVert z \rVert_{s+1}+\lVert \mathfrak{I}_{\delta}\rVert_{s+\sigma}\lVert z \rVert_{s+1})$, where $C$ is the Lipschitz constant, in time, on the interval $[0, 1]$ of the flow $(\Phi^{\tau})^T$. Hence, by \eqref{gk}, 
\begin{equation}\label{kg}
g_k=-\varepsilon\beta_1\,\partial_x e^{\mathrm{i} k x}+O(\varepsilon^2).
\end{equation}
Now consider
\[
\chi_k:=-\frac{1+\beta_x}{1+\tau \beta_x}\,\, \exp(\mathrm{i} k \gamma^{\tau}(x+\beta(x))),
\]
where $\gamma^{\tau}$ is the flow of the characteristic ODE
\begin{equation}\label{charactODE}
\frac{d}{d \tau} x=-b(\tau, x).
\end{equation}
By \eqref{adjointPDE}, the vector field of \eqref{charactODE} has size $O(\varepsilon)$ and, by similar arguments used above for the flow of \eqref{adjointPDE}, we have $\gamma^{\tau}(x)-x=O(\varepsilon)$. By Taylor expansion of the function $\exp(\mathrm{i} k \gamma^{\tau}(x+\beta(x)))$ at $\beta=0$ we have
\begin{equation}\label{kchi}
\chi_k=e^{\mathrm{i} k x}+O(\varepsilon).
\end{equation}
Recalling \eqref{AmenoI} we have
\begin{equation}\label{kpsi}
\psi_k=(\mathcal{A}^T-\mathrm{I}) e^{\mathrm{i} k x}=\varepsilon\partial_x (\beta_1 e^{\mathrm{i} k x})+O(\varepsilon^2)=\varepsilon(\beta_1)_x e^{\mathrm{i} k x}+\varepsilon\beta_1\,\partial_x e^{\mathrm{i} k x}+O(\varepsilon^2).
\end{equation}
Eventually, by \eqref{kg}, \eqref{kchi} and \eqref{kpsi}, we have $\mathcal{R}_{\Phi}=\varepsilon\mathrm{R}_{\Phi}+O(\varepsilon^2)$, where
\begin{equation}\label{kd}
\begin{aligned}
\mathrm{R}_{\Phi} (h):&=-\sum_{k\in S} (h, \beta_1 \partial_x e^{\mathrm{i} k x})_{L^2(\mathbb{T})} e^{\mathrm{i} k x}+\sum_{k\in S} (h, (\beta_1)_x e^{\mathrm{i} k x})_{L^2(\mathbb{T})} e^{\mathrm{i} k x}+\sum_{k\in S} (h, \beta_1 \partial_x e^{\mathrm{i} k x})_{L^2(\mathbb{T})} e^{\mathrm{i} k x}\\[2mm]
&=\Pi_S[(\beta_1)_x \,h].
\end{aligned}
\end{equation}
By \eqref{kd} the range of $\mathrm{R}_{\Phi}$ is orthogonal to the subspace $H_S^{\perp}$, hence the term $\Phi^{-1}\, R_{\Phi} (b_3 \partial_{yyy}+b_1 \partial_y+b_0)$ will have size at least $O(\varepsilon^3)$, indeed $\Phi=\mathrm{I}_{H_S^{\perp}}+O(\varepsilon)$.\\
We ignore the terms $\varepsilon^2 \mathcal{R}_2$ and $\mathcal{R}_*$ because are too small. Then, we can consider 
\[
(\mathcal{L}_{\omega}-\mathcal{D}_{\omega})\mathrm{R}_{\Phi}=\Pi_S^{\perp}[\partial_{xx} (a_1\partial_x )+\partial_x (a_0 \cdot)]\Pi_S^{\perp}\mathrm{R}_{\Phi}=0.
\]
By \eqref{AmenoI} we have
\begin{equation}\label{second}
\Pi_S^{\perp} [\partial_{xx} (a_1 \partial_x)+\partial_x (a_0 \cdot)]\Pi_S (\mathcal{A}-\mathrm{I}) \Pi_S^{\perp}=\varepsilon^2 \Pi_S^{\perp}[\partial_{xx} (a_{1, 1}\partial_{xx}\Pi_S[\beta_1 \,\cdot])+\partial_x (a_{0, 1}\,\partial_x\Pi_S[\beta_1\,\cdot])]+o(\varepsilon^2).
\end{equation}
It remains to study the commutator $[\mathcal{D}_{\omega}, \mathcal{R}_{\Phi}]=[D_{\overline{\omega}}, \mathcal{R}_{\Phi}^{\varepsilon}]+O(\varepsilon^3)$. We have
\begin{align*}
[D_{\overline{\omega}}, \mathcal{R}_{\Phi}^{\varepsilon}] h=\varepsilon D_{\overline{\omega}}\Pi_S[(\beta_1)_x h]-\varepsilon \Pi_S[(\beta_1)_x D_{\overline{\omega}} h]=\varepsilon\Pi_S[(D_{\overline{\omega}} (\beta_1)_x) h]
\end{align*}
and so $\Phi^{-1} [\mathcal{D}_{\omega}, \mathcal{R}_{\Phi}]=o(\varepsilon^2)$.\\
Finally, by \eqref{first}, \eqref{second}, we obtained $\mathcal{R}_{II}=\varepsilon^2 \mathrm{R}_2+o(\varepsilon^2)$, where, for $h\in H_S^{\perp}$,
\begin{equation}
\begin{aligned}
\mathrm{R}_2[h]&=\Pi_S^{\perp} \{\partial_x (\beta_1\, \Pi_S [\partial_x (\alpha_{1, 1} h)])-\partial_{xx} (a_{1,1} \partial_x \Pi_S[\partial_x (\beta_1 h)])-\partial_x (\alpha_{1, 1} \Pi_S[\partial_x(\beta_1 h)])\}\\
&=4 \,c_1 c_2\, \Pi_S^{\perp} \{ -\partial_x(v_{\delta}\,\partial_x \Pi_S[(v_{\delta})_{xx}\,h]) +\partial_{xx} ((v_{\delta})_x \partial_{xx} \Pi_S[(\partial_x^{-1} v_{\delta}) h]) 
\\
&+\partial_{xx} (v_{\delta} \,\partial_{xx}\Pi_S[ v_{\delta} h])+\partial_x ((v_{\delta})_{xx} \,\partial_x \Pi_S[v_{\delta} h]) \}\\
&+\frac{4}{3} \,c_2^2 \,\Pi_S^{\perp} \{ -\partial_x((\partial_x^{-1} v_{\delta})\,\partial_x \Pi_S[(v_{\delta})_{xx}\,h]) +\partial_{xx} (v_{\delta}\partial_{xx}\, \Pi_S[(\partial_x^{-1} v_{\delta}) h])
\\
&+\partial_x ((v_{\delta})_{xx} \,\partial_x \Pi_S[(\partial_x^{-1} v_{\delta}) h]) \}\\
&+12 \,c_1 c_3\,\Pi_S^{\perp}\{ \partial_x(v_{\delta}\,\partial_x \Pi_S[v_{\delta}\,h])-\partial_x (v_{\delta}\,\partial_x \Pi_S[v_{\delta} h]) \}\\
&+4\,c_2 c_3\,\Pi_S^{\perp}\{ \partial_x((\partial_x^{-1} v_{\delta})\,\partial_x \Pi_S[v_{\delta}\,h])-\partial_x (v_{\delta} \,\partial_x \Pi_S[(\partial_x^{-1} v_{\delta}) h])\}\\
&+12 c_1^2 \Pi_S^{\perp}\{ \partial_{xx} ((v_{\delta})_x \,\partial_{xx}\Pi_S[v_{\delta} h]) \}
\end{aligned}
\end{equation}

Using \eqref{R_I}, \eqref{R_II} we get
\begin{equation}\label{8.34}
\mathfrak{R}_1:=\Phi^{-1} \mathcal{R}_{\mathrm{I}\mathrm{I}}=-\varepsilon^2 \Pi_S^{\perp} \partial_x \mathcal{R}_2+\mathcal{R}_*
\end{equation}
where $\mathcal{R}_2$, defined in \eqref{R2}, has been renamed as 
\begin{equation}\label{NewR2}
\mathcal{R}_2:=\mathcal{R}_2-\partial_x^{-1}\mathrm{R}_2
\end{equation}
and we have renamed $\mathcal{R}_*$ the term $o(\varepsilon^2)$. Note that $\mathcal{R}_{II}^{\varepsilon^2}[h]$ has zero spatial average for every $h$ belonging to $H^s_{S^{\perp}}(\T^{\nu+1})$ and the remainder $\mathcal{R}_*$ has the form \eqref{VeraFinDimForm}.

\subsection{Time reduction at the order $\partial_{xxx}$}
The goal of this section is to make constant the coefficient of the highest order spatial derivative operator $\partial_{yyy}$ by a quasi-periodic reparametrization of time. We consider the change of variable
\begin{equation}\label{RepofTime}
(Bw)(\varphi, y):=w(\varphi+\omega \alpha(\varphi), y), \quad (B^{-1} h)(\vartheta, y):=h(\vartheta+\omega \tilde{\alpha}(\vartheta), y),
\end{equation}
where $\varphi=\vartheta+\omega \tilde{\alpha}(\vartheta)$ is the inverse diffeomorphism of $\vartheta=\varphi+\omega \alpha(\varphi)$ in $\T^{\nu}$. By conjugation, the differential operators transform into
\begin{equation}\label{Ro}
B^{-1}\omega\cdot \partial_{\varphi} B=\rho(\vartheta) \omega\cdot \partial_{\vartheta}, \quad B^{-1}\partial_y B=\partial_y, \quad \rho:=B^{-1}(1+\omega\cdot\partial_{\varphi}\alpha).
\end{equation}
By \eqref{L1}, using also that $B$ and $B^{-1}$ commute with $\Pi_S^{\perp}$, we get
\begin{equation}\label{8.42}
B^{-1}\mathcal{L}_1 B=\Pi_S^{\perp} [\rho\,\omega\cdot\partial_{\vartheta}+(B^{-1}b_3)\partial_{yyy}+(B^{-1} b_1)\partial_y+(B^{-1} b_0)]\Pi_S^{\perp}+B^{-1} \mathfrak{R}_1 B.
\end{equation}
We choose $\alpha$ such that the new coefficient at order $\partial_{yyy}$ is proportional to the function $\rho(\vartheta)$, namely
\begin{equation}\label{8.43}
(B^{-1}b_3)(\vartheta)=m_3\,\rho(\vartheta), \quad m_3\in\mathbb{R} \quad  \Longrightarrow \quad b_3(\varphi)=m_3 (1+\omega\cdot \partial_{\varphi}\alpha(\varphi)).
\end{equation}
The unique solution with zero average of \eqref{8.43} is
\begin{equation}\label{Defalpha}
\alpha(\varphi):=\frac{1}{m_3} (\omega\cdot\partial_{\varphi})^{-1} (b_3-m_3)(\varphi), \quad m_3:=\frac{1}{(2\pi)^{\nu}} \int_{\T^{\nu}} b_3(\varphi)\,d\varphi.
\end{equation}
Hence, by \eqref{8.42} we have
\begin{align}
& B^{-1}\mathcal{L}_1 B=\rho\,\mathcal{L}_2, \quad \mathcal{L}_2:=\Pi_S^{\perp}(\omega\cdot\partial_{\vartheta}+m_3 \partial_{yyy}+c_1 \partial_y+c_0)\Pi_S^{\perp}+\mathfrak{R}_2,\label{8.45}\\
&c_1:=\rho^{-1} (B^{-1} b_1), \quad c_0:=\rho^{-1} (B^{-1} b_0), \quad \mathfrak{R}_2:=\rho^{-1} B^{-1} \mathfrak{R}_1 B.\label{8.46}
\end{align}

In order to control the corrections to the normal frequencies also at lower orders of size, we expand the constant coefficient $m_3$, defined in \eqref{Defalpha}, in powers of $\varepsilon$. We have
\begin{equation}
m_3=1+\varepsilon^2 d(\xi)+\mathtt{r}_{m_3}
\end{equation}
where
\begin{equation}\label{OriginalDxi}
\begin{aligned}
d(\xi):&=(12 c_4-24\,c_1^2)\,M_{\varphi, x} [\overline{v}_x^2]+\varepsilon^2(2 c_6-\frac{8}{3} c_2^2) M_{\varphi, x}[\overline{v}^2]\\
&=(24 c_4-48 c_1^2) v_3\cdot \xi+(4 c_6-\frac{16}{3} c_2^2) v_1\cdot \xi
\end{aligned}
\end{equation}
and $\lvert \mathtt{r}_{m_3}\rvert^{Lip(\gamma)}\le \varepsilon^3$.
The transformed operator $\mathcal{L}_2$ in \eqref{8.45} is still Hamiltonian, since the reparametrization of time preserves the Hamiltonian structure (see Section $2.2$ and Remark $3.7$ in \cite{Airy}).\\
We note that, by \eqref{8.46}, for $k=0, 1$, we have
\[
c_k=b_k+(B^{-1}-\mathrm{I}) b_k+(\rho^{-1}-1)\,B^{-1}\,b_k
\]
and $b_k=O(\varepsilon)$ is the biggest term in the expression above. We define, for $k=0, 1$,
\begin{equation}\label{TildecK}
\tilde{c}_k:=c_k-b_k=(B^{-1}-\mathrm{I}) b_k+(\rho^{-1}-1) B^{-1} b_k
\end{equation}
and we estimate them in Lemma \ref{LemmaFondamentale}.
The remainder $\mathfrak{R}_2$ in \eqref{8.46} has still the form \eqref{VeraFinDimForm} and, by \eqref{8.34},
\begin{equation}
\mathfrak{R}_2:=-\rho^{-1} B^{-1} \mathfrak{R}_1 B=-\varepsilon^{2} \Pi_S^{\perp} \partial_x \mathcal{R}_2+\mathcal{R}_*
\end{equation}
where $\mathcal{R}_2$ is defined in \eqref{NewR2} and we have renamed $\mathcal{R}_*$ the term of order $o(\varepsilon^2)$ in $\mathfrak{R}_2$.

\begin{remark}\label{NoS1}
In the proof of the estimates for the transformations $B$ and $\mathcal{T}$, respectively defined in \eqref{RepofTime} and \eqref{TransloftheSpaceVar}, we have to give a bound to the inverse of the operator $\mathcal{D}_{\omega}$ applied to the difference of a spatial and total (in space and time) average of some function in $H_{S^{\perp}}^s(\T^{\nu+1})$.\\
The main problem is that the estimate \eqref{stimaDomega} is too rough to deal with functions $h(\varphi, x)$ of size greater or equal than $\varepsilon^3$, indeed, the terms $O(\varepsilon^3 \gamma^{-1})$ are just not perturbative.

 
\noindent In the proofs of Lemma \ref{LemmaFondamentale} and \ref{LemmaFond2}, we exploit the fact that if $h(\varphi, x)$ is a function supported on few harmonics, then we do not need to use the diophantine inequality \eqref{0diMelnikov} to give a bound to the divisors appearing in the Fourier coefficients of $\mathcal{D}_{\omega}^{-1} h$.\\
In this way, we overcome the problem discussed in Remark $8.11$ in \cite{KdVAut} and we can drop the hypotesis \eqref{S1} on the tangential sites assumed in \cite{KdVAut}.\\
\end{remark}

\begin{lem}\label{LemmaFondamentale}
There is $\sigma=\sigma(\nu, \tau)>0$ (possibly larger than the one in Lemma \ref{StimaPasso1}) such that
\begin{align}
\lvert m_3-1 \rvert^{Lip(\gamma)}\le C\,\varepsilon^2, \quad &\lvert \partial_i m_3 [\hat{\imath}] \rvert\le \varepsilon^2 \lVert \hat{\imath} \rVert_{s_0+\sigma},\label{m3}\\
 \lVert \alpha \rVert_s^{Lip(\gamma)}\le_s \varepsilon^4 \gamma^{-1}+\lVert \mathfrak{I}_{\delta} \rVert_{s+\sigma}^{Lip(\gamma)}, \quad &\lVert \partial_i \alpha [\hat{\imath}] \rVert_s\le_s \lVert \hat{\imath} \rVert_{s+\sigma}+\lVert \mathfrak{I}_{\delta} \rVert_{s+\sigma} \lVert \hat{\imath} \rVert_{s_0+\sigma} \label{alpha}\\
\lVert \rho-1 \rVert_s^{Lip(\gamma)}\le_s \varepsilon^3+\varepsilon^{2 b}\lVert \mathfrak{I}_{\delta} \rVert_{s+\sigma}^{Lip(\gamma)}, &\quad \lVert \partial_i \rho [\hat{\imath}] \rVert_s\le_s \varepsilon^{2 b} (\lVert \hat{\imath} \rVert_{s+\sigma}+\lVert \mathfrak{I}_{\delta} \rVert_{s+\sigma} \lVert \hat{\imath} \rVert_{s_0+\sigma}), \label{rho}\\
\lVert \tilde{c}_k \rVert_s^{Lip(\gamma)}\le_s \varepsilon^{3-2 a}+\varepsilon\lVert \mathfrak{I}_{\delta} \rVert_{s+\sigma}^{Lip(\gamma)}, &\quad \lVert \partial_i \tilde{c}_k [\hat{\imath}] \rVert_s\le_s \varepsilon (\lVert \hat{\imath} \rVert_{s+\sigma}+\lVert \mathfrak{I}_{\delta} \rVert_{s+\sigma} \lVert \hat{\imath} \rVert_{s_0+\sigma}).\label{cK}
\end{align}
\end{lem}
\begin{proof}
\textit{Estimate \eqref{m3}}: We have $m_3-1=\int_{\T^{\nu}} (b_3-1)\,d\varphi$,
then, by \eqref{b3},
\begin{align*}
&\lvert m_3-1 \rvert\le \int_{\T^{\nu}} \lvert b_3-1 \rvert d\varphi\le \lVert b_3-1 \rVert_{s_0}\le C \varepsilon^2,
&\lvert \partial_i m_3 [\hat{\imath}] \rvert\le \int_{\T^{\nu}} \partial_i b_3[\hat{\imath}] d\varphi\le \lVert\partial_i b_3[\hat{\imath}] \rVert_{s_0}\le \varepsilon^2 \lVert \hat{\imath} \rVert_{s_0+2}.
\end{align*}

\noindent\textit{Estimate \eqref{alpha}}: By \eqref{Defalpha} and the fact that $m_3$ is a constant near to $1$, it is sufficient to give a bound to $b_3-m_3$.\\
Consider the functions $g(t)=(1+t)^{-\frac{1}{3}},\Upsilon(t)=(1+t)^{-3}$, defined in a small neighbourhood of the origin.\\
We have
\begin{equation}\label{KillBill}
\begin{aligned}
b_3-m_3=(b_3-1)-M_{\varphi}[b_3-1]\stackrel{\eqref{Michela}}{=}\Upsilon[M_x[g(a_1-1)-g(0)]]-M_{\varphi}[\Upsilon[M_x[g(a_1-1)-g(0)]]].
\end{aligned}
\end{equation}
By the analiticity of $\Upsilon$ 
\[
\Upsilon(t)-\Upsilon(0)=\Upsilon'(0)\,t+\Upsilon_{\geq 2}[t], \qquad \Upsilon_{\geq 2}[t]:=\sum_{k\geq 2} \frac{\Upsilon^{(k)}(0)}{k!}\,t^k,
\]
for $\lvert t \rvert$ small enough. Hence, by \eqref{KillBill},
\begin{equation}\label{KillBill2}
\begin{aligned}
b_3-m_3&=\Upsilon'(0)\{ M_x[g(a_1-1)-g(0)]-M_{\varphi, x}[g(a_1-1)-g(0)] \}\\[1.5mm]
&+\Upsilon_{\geq 2}[M_x[g(a_1-1)-g(0)]]-M_{\varphi}[\Upsilon_{\geq 2}[M_x[g(a_1-1)-g(0)]]].
\end{aligned}
\end{equation}
The difference of the last two terms in the right hand side of \eqref{KillBill2} can be estimated by
\begin{align*}
&\lVert  \Upsilon_{\geq 2}[M_x[g(a_1-1)-g(0)]]-M_{\varphi}[\Upsilon_{\geq 2}[M_x[g(a_1-1)-g(0)]]] \rVert_s\\[1.5mm]
&\le_s \lVert M_x[g(a_1-1)-g(0)] \rVert_{s_0}\lVert M_x[g(a_1-1)-g(0)] \rVert_s\stackrel{\eqref{b3}}{\le_s} \varepsilon^4(1+\lVert \mathfrak{I}_{\delta}\rVert_s).
\end{align*}
Now we prove a bound for the difference $M_x[g(a_1-1)-g(0)]-M_{\varphi, x}[g(a_1-1)-g(0)]$.\\
By Taylor expansion
\begin{equation*}
g(a_1-1)-g(0)=g'(0)(a_1-1)+\frac{g''(0)}{2} (a_1-1)^2+\frac{g'''(0)}{3!}(a_1-1)^3+\frac{(a_1-1)^4}{6}\int_0^1 (1-s)^3\,g^{(4)}(s(a_1-1))\,ds
\end{equation*}
and the last term of the right hand side can be estimated by $\varepsilon^4(1+\lVert \mathfrak{I}_{\delta}\rVert_{s+\sigma})$.\\
The function $a_1$ in \eqref{a1} is a linear combination of  $\Phi_B(T_{\delta}), \Phi_B(T_{\delta})^2$ (and their derivatives in the $x$-variable) and $r_1(T_{\delta})$, whose coefficients depend on $c_1, \dots, c_7$ and other real constants. Without loss of generality, to simplify the notations, we can write $a_1=1+\Phi_B(T_{\delta})+\Phi_B(T_{\delta})^2+r_1(T_{\delta})$ (recall \eqref{Tdelta}, \eqref{PhiB} and \eqref{r1}). Thus, we have
\begin{align*}
M_x[a_1-1]&=M_x[\Phi_B(T_{\delta})^2]+M_x[r_1(T_{\delta})], \quad M_x[(a_1-1)^2]=M_x[\Phi_B(T_{\delta})^2]+2 M_x[\Phi_B(T_{\delta})^3]+\mathtt{Q}_2(T_{\delta}),\\[1.5mm]
M_x[(a_1-1)^3]&=4 M_x[\Phi_B(T_{\delta})^3]+\mathtt{Q}_3(T_{\delta}),
\end{align*}
where $\lVert \mathtt{Q}_i(T_{\delta})\rVert_s\le_s \varepsilon^4 +\varepsilon^{2+b}\lVert \mathfrak{I}_{\delta} \rVert_{s+\sigma}$ for $i=2, 3$. By \eqref{r1} and the fact that $\Phi_B(T_{\delta})$ has size $O(\varepsilon)$, $r_1(T_{\delta})$ is a polynomial of degree three in the variables $(\Phi_B(T_{\delta}), \Phi_B(T_{\delta})_x)$, up to a remainder that is bounded in $H^s$ norm by $\varepsilon^4 (1+\lVert \mathfrak{I}_{\delta}\rVert_{s+\sigma})$. Thus, we reduced to study the differences 
\[
M_x[\Phi_B(T_{\delta})^2]-M_{\varphi, x}[\Phi_B(T_{\delta})^2], \qquad M_x[\Phi_B(T_{\delta})^3]-M_{\varphi, x}[\Phi_B(T_{\delta})^3].
\]
We have, up to constants,
\[
\Phi_B(T_{\delta})^2=\varepsilon^2 v_{\delta}^2+\varepsilon v_{\delta}\tilde{q}+\varepsilon^3 v_{\delta}\Psi_2(v_{\delta})+\tilde{\mathtt{Q}}_2(T_{\delta}), \qquad \Phi_B(T_{\delta})^3=\varepsilon^3 v_{\delta}^3+\tilde{\mathtt{Q}}_3(T_{\delta}),
\]
where $\lVert \tilde{\mathtt{Q}}_i(T_{\delta}) \rVert_s\le_s \varepsilon^4+\varepsilon^{2+b}\lVert \mathfrak{I}_{\delta}\rVert_{s+\sigma}$ for $i=2, 3$.
By the definition of $\tilde{q}$ and the fact that $v_{\delta}$ and $z_0$ are orthogonal in $L^2(\T)$, we have
\begin{equation}\label{PlanetTerror}
\varepsilon M_x[v_{\delta} \tilde{q}]=\varepsilon^{2+b} M_x[\Psi'_2(v_{\delta}) v_{\delta} z_0]+\varepsilon M_x[v_{\delta}\,\Psi_3(T_{\delta})],
\end{equation}
thus $\lVert M_x[\varepsilon v_{\delta} \tilde{q}]-M_{\varphi, x}[\varepsilon v_{\delta} \tilde{q}]\rVert_s\le_s \varepsilon^4+ \varepsilon^{2+b}\lVert \mathfrak{I}_{\delta}\rVert_{s+\sigma}$. It remains to estimate the differences of the averages of polynomial of degree two and three in the variables $v_{\delta}$ and its derivatives. These functions are of order $\varepsilon^2$ and $\varepsilon^3$, respectively, and supported on not many harmonics, because $v_{\delta}$ is not.\\
By \eqref{Tdelta} we get
\[
M_x[v_{\delta}^2]-M_{\varphi, x}[v_{\delta}^2]=\varepsilon^{2 (b-1)} \sum_{j\in S} \lvert j \rvert ((y_{\delta})_j-M_{\varphi}[(y_{\delta})_j]).
\]
We gain an extra smallness factor $\varepsilon^{2(b-1)}$ by the fact that $M_x[\overline{v}^2]$ is independent of $\varphi$ (see Remark \ref{KerL}).
Thus, we obtain $\varepsilon^2\lVert M_x[v_{\delta}^2]-M_{\varphi, x}[v_{\delta}^2] \rVert_s\le_s \varepsilon^{2 b}\lVert \mathfrak{I}_{\delta}\rVert_{s}$.\\
For the cubic terms in $v_{\delta}$ we use the following equality
\begin{equation}\label{GrindHouse}
M_x[v_{\delta}^3]-M_{\varphi, x}[v_{\delta}^3]=(M_x[\overline{v}^3]-M_{\varphi, x}[\overline{v}^3])+M_x[v_{\delta}^3-\overline{v}^3]-M_{\varphi, x}[v_{\delta}^3-\overline{v}^3],
\end{equation}
where $\lVert M_x[v_{\delta}^3-\overline{v}^3]-M_{\varphi, x}[v_{\delta}^3-\overline{v}^3]\rVert_s\le_s \varepsilon^3 \lVert \mathfrak{I}_{\delta}\rVert_s$.\\
We now analyze the first difference in the right hand side of \eqref{GrindHouse}. We cannot roughly bound it by $\varepsilon^3$ (see Remark \ref{NoS1}). But we have 
\begin{equation}\label{Hateful8}
\mathcal{D}_{\omega}^{-1}\left( M_x[\overline{v}^3]-M_{\varphi, x}[\overline{v}^3]\right)=\sum_{\substack{j_1, j_2, j_3\in S,\\j_1+j_2+j_3=0\\\mathtt{l}(j_1)+\mathtt{l}(j_2)+\mathtt{l}(j_3)\neq 0}} \frac{\sqrt{\xi_{j_1}\xi_{j_2}\xi_{j_3}}}{\mathrm{i}\,\omega\cdot (\mathtt{l}(j_1)+\mathtt{l}(j_2)+\mathtt{l}(j_3))}\,e^{\mathrm{i}(\mathtt{l}(j_1)+\mathtt{l}(j_2)+\mathtt{l}(j_3))\cdot\varphi}.
\end{equation}
We recall that $\omega=\overline{\omega}+O(\varepsilon^2)$, hence the denominator in \eqref{Hateful8} can be written as
\[
\omega\cdot (\mathtt{l}(j_1)+\mathtt{l}(j_2)+\mathtt{l}(j_3))=\overline{\omega}\cdot (\mathtt{l}(j_1)+\mathtt{l}(j_2)+\mathtt{l}(j_3))+(\omega-\overline{\omega})\cdot (\mathtt{l}(j_1)+\mathtt{l}(j_2)+\mathtt{l}(j_3)) =j_1^3+j_2^3+j_3^3+O(\varepsilon^2)
\]
and it is greater or equal than $1$, indeed, if $j_1+j_2+j_3=0$, then $\lvert j_1^3+j_2^3+j_3^3\rvert=3 \lvert j_1\,j_2\,j_3\rvert\geq 3$. Thus, actually,
\[
\lVert \mathcal{D}_{\omega}^{-1}\left( M_x[\overline{v}^3]-M_{\varphi, x}[\overline{v}^3]\right) \rVert_s\le \varepsilon^3.
\]
Finally, we get
\begin{equation}\label{Memento}
\lVert b_3-m_3 \rVert_s\le_s \varepsilon^3+\varepsilon^{2 b}\lVert \mathfrak{I}_{\delta}\rVert_{s+\sigma} \qquad \mbox{and} \qquad\lVert \mathcal{D}_{\omega}^{-1} (b_3-m_3)\rVert_s\le_s \varepsilon^4\gamma^{-1}+\lVert \mathfrak{I}_{\delta}\rVert_{s+\sigma},
\end{equation}
so $\lVert \alpha \rVert_s\le_s \varepsilon^4\gamma^{-1}+\lVert \mathfrak{I}_{\delta}\rVert_{s+\sigma}$.\\
Now we look to the partial derivative
\begin{equation}\label{ToroScatenato}
\partial_i \left( \frac{b_3-m_3}{m_3} \right)[\hat{\imath}]=\frac{1}{m_3^2}\left[m_3\partial_i (b_3-m_3)[\hat{\imath}]  -(b_3-m_3)\partial_i m_3[\hat{\imath}]\right].
\end{equation}
By \eqref{m3} $m_3-1$ and $\partial_i m_3[\hat{\imath}]$ are of order $\varepsilon^2$, hence the estimate for $\partial_i \alpha [\hat{\imath}]$ comes from $\mathcal{D}_{\omega}^{-1}(\partial_i (b_3-m_3)[\hat{\imath}])$. By \eqref{KillBill2} we have
\begin{align*}
\partial_i (b_3-m_3)[\hat{\imath}]&=\Upsilon'(0)\{ M_x[\partial_i (g(a_1-1)-g(0))[\hat{\imath}]]-M_{\varphi, x}[\partial_i (g(a_1-1)-g(0))[\hat{\imath}]] \}\\
&+\partial_i \{\Upsilon_{\geq 2}[M_x[g(a_1-1)-g(0)]]-M_{\varphi}[\Upsilon_{\geq 2}[M_x[g(a_1-1)-g(0)]]]\} [\hat{\imath}]
\end{align*}
As before, the bigger terms are the partial derivatives of $M_x[g(a_1-1)-g(0)]-M_{\varphi, x}[g(a_1-1)-g(0)]$. We have
\[
\partial_i (g(a_1-1)-g(0)) [\hat{\imath}]=g'(0)\partial_i a_1[\hat{\imath}]+g''(0)(a_1-1)\,\partial_i a_1[\hat{\imath}]+\frac{g'''(0)}{2} (a_1-1)^2\,\partial_i a_1[\hat{\imath}]+\mathtt{T}(i_{\delta}, \hat{\imath})
\]
where $\lVert \mathtt{T}(i_{\delta}, \hat{\imath}) \rVert_s\le_s \varepsilon^4(\lVert \hat{\imath}\rVert_{s+\sigma}+\lVert \mathfrak{I}_{\delta}\rVert_{s+\sigma}\lVert \hat{\imath}\rVert_{s_0+\sigma})$ and
\[
\partial_i a_1[\hat{\imath}]=\partial_i \Phi_B(T_{\delta})[\hat{\imath}]+2 \Phi_B(T_{\delta})\partial_i \Phi_B(T_{\delta})[\hat{\imath}]+\partial_i \tilde{q}[\hat{\imath}].
\]
We note that $M_x[\partial_i \Phi_B(T_{\delta})[\hat{\imath}]]=M_x[\partial_i \tilde{q}[\hat{\imath}]]=0$. Thus, we focus on the terms
$$\Phi_B(T_{\delta})\partial_i \Phi_B(T_{\delta}), \quad\Phi_B(T_{\delta})^2\partial_i \Phi_B(T_{\delta}),\quad\Phi_B(T_{\delta})\,\partial_i \tilde{q}[\hat{\imath}].$$
Further terms have Sobolev norm bounded by $\varepsilon^{2+b}(\lVert \hat{\imath}\rVert_{s+\sigma}+\lVert \mathfrak{I}_{\delta}\rVert_{s+\sigma}\lVert \hat{\imath}\rVert_{s_0+\sigma})$.
We have
\begin{align*}
\Phi_B(T_{\delta})\partial_i \Phi_B(T_{\delta})&=\varepsilon^2 v_{\delta}\, \partial_i v_{\delta}[\hat{\imath}]+\varepsilon^3 (\Psi_2(v_{\delta})+\Psi'_2(v_{\delta}) v_{\delta}) \partial_i v_{\delta}[\hat{\imath}]+\varepsilon \partial_i(\tilde{q}\, v_{\delta})[\hat{\imath}]+\tilde{\mathtt{T}}(i_{\delta}, \hat{\imath}),\\
\Phi_B(T_{\delta})^2 \partial_i \Phi_B(T_{\delta})[\hat{\imath}]&=\varepsilon^3 v_{\delta}^2\,\partial_i v_{\delta}[\hat{\imath}]+\tilde{\mathtt{T}}(i_{\delta}, \hat{\imath}),
\end{align*}
where $\lVert \mathtt{T}(i_{\delta}, \hat{\imath})\rVert_s\le_s \varepsilon^{2+b}(\lVert \hat{\imath}\rVert_{s+\sigma}+\lVert \mathfrak{I}_{\delta}\rVert_{s+\sigma}\lVert \hat{\imath}\rVert_{s_0+\sigma})$. We start from the average of the partial derivative of $v_{\delta}\tilde{q}$. By \eqref{PlanetTerror} we get $\varepsilon\lVert \partial_i M_x[v_{\delta}\tilde{q}] [\hat{\imath}]\rVert_s\le_s\varepsilon^{2+b}(\lVert \hat{\imath}\rVert_{s+\sigma}+\lVert \mathfrak{I}_{\delta}\rVert_{s+\sigma}\lVert \hat{\imath}\rVert_{s_0+\sigma})$. Then, we reduce to study
\[
M_x[v_{\delta}\partial_i v_{\delta}[\hat{\imath}]]-M_{\varphi, x}[v_{\delta}\partial_i v_{\delta}[\hat{\imath}]], \qquad M_x[v_{\delta}^2\partial_i v_{\delta}[\hat{\imath}]]-M_{\varphi, x}[v_{\delta}^2\partial_i v_{\delta}[\hat{\imath}]].
\]
If we call $G(i_0(\varphi)):=y_{\delta}-y_0$, then we have
\[
\partial_i v_{\delta}[\hat{\imath}]=\sum_{j\in S}\sqrt{\lvert j \rvert}\sqrt{\xi_j+\varepsilon^{2(b-1)} (y_{\delta})_j}e^{\mathrm{i}(\theta_0)_j}\left(\mathrm{i}\hat{\Theta}_j+\varepsilon^{2(b-1)}\frac{\hat{y}_j+(\partial_i G(i_0(\varphi))[\hat{\imath}])_j}{2\,\lvert j \rvert\,(\xi_j+\varepsilon^{2(b-1)}(y_{\delta})_j)}  \right)\,e^{\mathrm{i} j x}
\]
and
\begin{align*}
M_x[v_{\delta}\partial_i v_{\delta}[\hat{\imath}]]-M_{\varphi, x}[v_{\delta}\partial_i v_{\delta}[\hat{\imath}]]&=\varepsilon^{2 (b-1)}\sum_{j\in S} \mathrm{i} \hat{\Theta}_j ((y_{\delta})_j-M_{\varphi}[(y_{\delta})_j])\\
&+\frac{\varepsilon^{2(b-1)}}{2}\sum_{j\in S} \{ (\partial_i G(i_0(\varphi))[\hat{\imath}])_j-M_{\varphi}[ (\partial_i G(i_0(\varphi))[\hat{\imath}])_j] \}.
\end{align*}
Therefore, $\varepsilon^2\lVert M_x[v_{\delta}\partial_i v_{\delta}[\hat{\imath}]]-M_{\varphi, x}[v_{\delta}\partial_i v_{\delta}[\hat{\imath}]] \rVert_s\le_s\varepsilon^{2 b}(\lVert \hat{\imath}\rVert_{s+\sigma}+\lVert \mathfrak{I}_{\delta}\rVert_{s+\sigma}\lVert \hat{\imath}\rVert_{s_0+\sigma})$. Moreover, we have $\lVert \varepsilon^3 M_x[v_{\delta}^2 \partial_i v_{\delta}[\hat{\imath}]] \rVert_s\le_s \varepsilon^3(\lVert \hat{\imath}\rVert_{s+\sigma}+\lVert \mathfrak{I}_{\delta}\rVert_{s+\sigma}\lVert \hat{\imath}\rVert_{s_0+\sigma})$. Hence, we get
\[
\lVert \partial_i (b_3-m_3)[\hat{\imath}]\rVert_s\le_s\varepsilon^{2 b}(\lVert \hat{\imath}\rVert_{s+\sigma}+\lVert \mathfrak{I}_{\delta}\rVert_{s+\sigma}\lVert \hat{\imath}\rVert_{s_0+\sigma})
\]
and $\lVert \partial_i \alpha [\hat{\imath}]\rVert\le_s\lVert \hat{\imath}\rVert_{s+\sigma}+\lVert \mathfrak{I}_{\delta}\rVert_{s+\sigma}\lVert \hat{\imath}\rVert_{s_0+\sigma}$.By Lemma \ref{lemmaLip} we deduce the inequality \eqref{alpha}.\\

\noindent\textit{Estimate \eqref{rho}}: Note that $\rho-1=B^{-1}\, \left((b_3-m_3)/m_3\right)$. Thus, by Lemma \ref{ChangeofVariablesLemma}, \eqref{alpha}, \eqref{Memento} we get
\begin{align*}
\lVert B^{-1}\, \left((b_3-m_3)/m_3\right) \rVert^{Lip(\gamma)}_s&\le_s \lVert b_3-m_3 \rVert^{Lip(\gamma)}_{s+1}+\lVert \alpha \rVert^{Lip(\gamma)}_{s+s_0} \lVert b_3-m_3 \rVert^{Lip(\gamma)}_2\\
&\le_s \varepsilon^3+\varepsilon^{2 b}\lVert \mathfrak{I}_{\delta} \rVert^{Lip(\gamma)}_{s+s_0+\sigma}.
\end{align*}

\noindent\textit{Estimate \eqref{cK}}: Note that $\lVert \rho^{-1}-1 \rVert_s\le_s \lVert \rho-1 \rVert_s$. By Lemma \ref{ChangeofVariablesLemma} and \eqref{TameProduct}, \eqref{b0b1}, we get, for $k=0,1$,
\begin{align*}
\lVert (B^{-1}-\mathrm{I})\,b_k \rVert^{Lip(\gamma)}_s &\le_s \varepsilon^7\,\gamma^{-2}+\varepsilon \lVert \mathfrak{I}_{\delta} \rVert_{s+\sigma}^{Lip(\gamma)}, \quad\lVert (\rho^{-1}-1) b_k \rVert_s^{Lip(\gamma)}\le_s \varepsilon^4+\varepsilon^{1+2 b}\lVert \mathfrak{I}_{\delta}\rVert^{Lip(\gamma)}_{s+\sigma}.
\end{align*}
\end{proof}

\subsection{Translation of the space variable}
The goal of this section is to remove the space average from the coefficient in front of $\partial_y$.  This is a preliminary step for the descent method that we apply at Section $8.7$.\\
Consider the change of variable
\begin{equation}\label{TransloftheSpaceVar}
(\mathcal{T} w)(\vartheta, y)=w(\vartheta, y+p(\vartheta)), \quad (\mathcal{T}^{-1} h)(\vartheta, z)=h(\vartheta, z-p(\vartheta)).
\end{equation}
The differential operators in $\mathcal{L}_2$ (see \eqref{8.45}) transform into
\[
\mathcal{T}^{-1} \omega\cdot\partial_{\vartheta} \mathcal{T}=\omega\cdot\partial_{\vartheta}+\{\omega\cdot\partial_{\vartheta} p(\vartheta)\}\partial_z, \quad \mathcal{T}^{-1} \partial_y \mathcal{T}=\partial_z.
\]
Since $\mathcal{T}, \mathcal{T}^{-1}$ commute with $\Pi_S^{\perp}$, we get
\begin{align}\label{OriginalL3}
& \mathcal{L}_3:=\mathcal{T}^{-1} \mathcal{L}_2 \mathcal{T}=\Pi_S^{\perp} (\omega\cdot\partial_{\vartheta}+m_3\,\partial_{zzz}+D_S\,\partial_z+d_0)\Pi_S^{\perp}+\mathfrak{R}_3,\\
&d_1:=(\mathcal{T}^{-1} c_1)+\omega\cdot\partial_{\vartheta} p, \quad d_0:=\mathcal{T}^{-1} c_0, \quad \mathfrak{R}_3:=\mathcal{T}^{-1}\mathfrak{R}_2 \mathcal{T}
\end{align}
and we choose
\begin{equation}\label{Defp}
m_1:=\frac{1}{(2\pi)^{\nu+1}} \int_{\T^{\nu+1}} c_1\,d\vartheta\,dy, \quad p:=(\omega\cdot\partial_{\vartheta})^{-1} \left(m_1-\frac{1}{2\pi} \int_{\T} c_1\,dy\right)
\end{equation}
so that
\begin{equation}
\frac{1}{2\pi} \int_{\T} d_1(\vartheta, z)\,dz=m_1\quad \forall \vartheta\in\T^{\nu}.
\end{equation}
We define
\begin{align}\label{DefdK}
&\tilde{d}_k:=d_k-\varepsilon\,\alpha_{k, 1}-\varepsilon^2 (\alpha_{k, 2}-\alpha_{k, 1}\, (\beta_1)_x), \qquad k=0, 1
\end{align}
and we split $\mathfrak{R}_3=-\varepsilon^2 \partial_x \overline{\mathcal{R}}_2+\tilde{\mathcal{R}}_*$, where $\overline{\mathcal{R}}_2$ is obtained replacing $v_{\delta}$ with $\overline{v}$ in $\mathcal{R}_2$ and
\begin{equation}\label{RtildeStar}
\tilde{\mathcal{R}}_*:=\mathcal{T}^{-1} \mathcal{R}_* \mathcal{T}+\varepsilon^2 \Pi_S^{\perp} \partial_x (\mathcal{R}_2-\mathcal{T}^{-1} \mathcal{R}_2 \mathcal{T})+\varepsilon^2\Pi_S^{\perp} \partial_x (\overline{\mathcal{R}}_2-\mathcal{R}_2),
\end{equation}
where $\mathcal{R}_*$ has been defined in \eqref{Rstar} and modified along this section by adding terms $o(\varepsilon^2)$. We used that $\mathcal{T}^{-1}$ commutes with $\partial_x$ and $\Pi_S^{\perp}$.\\
We define 
\begin{equation}\label{OriginalCdiXi}
c(\xi):=M_{\varphi, x}[\alpha_{1, 2}+\alpha_{1, 1}\,(\beta_1)_x].
\end{equation}
This quantity is a correction at order $\varepsilon^2$ to the eigenvalues of the linear operator $\mathcal{L}_{\omega}$, see \eqref{Lomega}. In particular, we have
\[
m_1=\varepsilon^2 c(\xi)+\mathtt{r}_{m_1}, \quad \mbox{with}\quad \lvert \mathtt{r}_{m_1}\rvert^{Lip(\gamma)}\le \varepsilon^{3-2 a}.
\]
\begin{lem}\label{LemmaFond2}
There is $\sigma:=\sigma(\tau, \nu)$ (possibly larger than in Lemma \ref{LemmaFondamentale}) such that
\begin{align}
\lvert m_1-\varepsilon^2 c(\xi) \rvert^{Lip(\gamma)}\le \varepsilon^7\,\gamma^{-2}, &\quad \lvert \partial_i (m_1-\varepsilon^2 c(\xi)) [\hat{\imath}]\rvert\le \varepsilon^7\,\gamma^{-2} \lVert \hat{\imath} \rVert_{s_0+\sigma}, \label{estimateM1}\\
 \lVert p \rVert_s^{Lip(\gamma)}\le_s \varepsilon^4 \gamma^{-1}+\lVert \mathfrak{I}_{\delta}\rVert_{s+\sigma}^{Lip(\gamma)}, &\quad \lVert \partial_i p [\hat{\imath}] \rVert_s\le_s \lVert \hat{\imath} \rVert_{s+\sigma}+\lVert \mathfrak{I}_{\delta} \rVert_{s+\sigma} \lVert \hat{\imath} \rVert_{s_0+\sigma}, \label{p}\\
\lVert \tilde{d}_k \rVert_s^{Lip(\gamma)}\le_s \varepsilon^{3-2 a}+\varepsilon \lVert \mathfrak{I}_{\delta} \rVert_{s+\sigma}^{Lip(\gamma)}, &\quad \lVert \partial_i \tilde{d}_k [\hat{\imath}]\rVert_s\le_s \varepsilon (\lVert \hat{\imath} \rVert_{s+\sigma}+\lVert \mathfrak{I}_{\delta} \rVert_{s+\sigma} \lVert \hat{\imath} \rVert_{s_0+\sigma} )\label{dK}
\end{align}
for $k=0, 1$. Moreover the matrix s-decay norm (see \eqref{decayNorm})
\begin{equation}
\lvert \tilde{\mathcal{R}}_* \rvert_s^{Lip(\gamma)}\le_s \varepsilon^3+\varepsilon^2 \lVert \mathfrak{I}_{\delta} \rVert_{s+\sigma}^{Lip(\gamma)}, \quad \lvert \partial_i \tilde{\mathcal{R}}_*[\hat{\imath}] \rvert_s\le_s \varepsilon^2 \lVert \hat{\imath} \rVert_{s+\sigma}+\varepsilon^{2 b-1} \lVert \mathfrak{I}_{\delta} \rVert_{s+\sigma} \lVert \hat{\imath} \rVert_{s_0+\sigma}. 
\end{equation}
The transformations $\mathcal{T}, \mathcal{T}^{-1}$ satisfy \eqref{8.39}, \eqref{8.40}.
\end{lem}
\begin{proof}
\textit{Estimate \eqref{p}}: By \eqref{8.46} and \eqref{Defp} we have
\begin{equation}\label{HeatLaSfida}
\begin{aligned}
m_1-M_x[c_1]&=(M_{\varphi, x}[b_1]-M_x[b_1])+(M_{\varphi, x}[(\rho^{-1}-1) b_1]-M_x[(\rho^{-1}-1) b_1])\\
&+(M_{\varphi, x}[(B^{-1}-\mathrm{I}) b_1]-M_x[(B^{-1}-\mathrm{I}) b_1])\\
&+(M_{\varphi, x}[(\rho^{-1}-1)(B^{-1}-\mathrm{I}) b_1]-M_x[(\rho^{-1}-1)(B^{-1}-\mathrm{I}) b_1]).
\end{aligned}
\end{equation}
By \eqref{alpha}, \eqref{rho} and Lemma \ref{TameProduct}, we get $\lVert (\rho^{-1}-1)(B^{-1}-\mathrm{I}) b_1 \rVert_s\le_s  \varepsilon^{9}\gamma^{-2}+\varepsilon^6\gamma^{-1}\lVert \mathfrak{I}_{\delta}\rVert_{s+\sigma}$.
Thus, by \eqref{stimaDomega}
\begin{equation}
\lVert \mathcal{D}_{\omega}^{-1}\{ M_{\varphi, x}[(\rho^{-1}-1)(B^{-1}-\mathrm{I}) b_1]-M_x[(\rho^{-1}-1)(B^{-1}-\mathrm{I}) b_1] \} \rVert_s\le_s \varepsilon^9\gamma^{-3}+\varepsilon^6\gamma^{-2}\lVert \mathfrak{I}_{\delta}\rVert_{s+\sigma}.
\end{equation}
We note that $\rho^{-1}-1$ is independent of $x$, hence $M_x[(\rho^{-1}-1) b_1]=(\rho^{-1}-1)M_x[b_1]$ and we can estimate the difference between the averages of $(\rho^{-1}-1) b_1$ with
\begin{equation}
\lVert (\rho^{-1}-1)M_x[b_1] \rVert_s\le_s \varepsilon^5+\varepsilon^{2(b+1)}\lVert \mathfrak{I}_{\delta}\rVert_{s+\sigma}
\end{equation}
and use again \eqref{stimaDomega} for $\lVert \mathcal{D}_{\omega}^{-1} (M_{\varphi, x}[(\rho^{-1}-1) b_1]-M_x[(\rho^{-1}-1) b_1])\rVert_s\le_s \varepsilon^5 \gamma^{-1}+\varepsilon\lVert \mathfrak{I}_{\delta}\rVert_{s+\sigma}$.\\
By Taylor expansion and the fact that $\tilde{\alpha}=-\alpha+(B-\mathrm{I})\alpha$ (see \eqref{RepofTime}), we have
\[
b_1(\vartheta+\omega\tilde{\alpha}(\vartheta), x)=b_1(\vartheta, x)-\omega\cdot\partial_{\vartheta} b_1(\vartheta, x)\,\alpha(\vartheta)+\mathtt{R}_{\tilde{\alpha}}(\vartheta, x)
\]
where $\lVert \mathtt{R}_{\tilde{\alpha}}\rVert_s\le_s \varepsilon^8\gamma^{-2}+\varepsilon^4 \gamma^{-1}\lVert \mathfrak{I}_{\delta}\rVert_{s+\sigma}$. Moreover, by a change of variable
\begin{equation}
\int_{\T^{\nu+1}} (B^{-1}-\mathrm{I}) b_1\,d\vartheta\,dx=\int_{\T^{\nu+1}} \omega\cdot\partial_{\varphi}\alpha(\varphi)\,b_1(\varphi, x)\,d\varphi\,dx.
\end{equation}
From these facts and an integration by parts, we obtain
\begin{align*}
M_x[(B^{-1}-\mathrm{I}) b_1]-M_{\varphi, x}[(B^{-1}-\mathrm{I}) b_1]=\mathcal{D}_{\omega}\alpha\, M_x[b_1]-M_{\varphi, x}[(\mathcal{D}_{\omega} \alpha)\,b_1]+M_x[\mathtt{R}_{\tilde{\alpha}}]-M_{\varphi, x}[\mathtt{R}_{\tilde{\alpha}}]
\end{align*}
and, by the estimate above for $\mathtt{R}_{\tilde{\alpha}}$ and the bound given by \eqref{Memento} for $\mathcal{D}_{\omega}\alpha$, we have
\begin{equation}
\lVert M_x[(B^{-1}-\mathrm{I}) b_1]-M_{\varphi, x}[(B^{-1}-\mathrm{I}) b_1] \rVert_s\le_s  \varepsilon^5+\varepsilon^{2(b+1)}\lVert \mathfrak{I}_{\delta}\rVert_{s+\sigma}.
\end{equation}
As before, we can use \eqref{stimaDomega}.
We remark that
\[
\int_{\T} b_1 (\varphi, y)\,dy=\int_{\T} (\mathcal{A}^T \alpha_1)(\varphi, y)\,dy=\int_{\T} \alpha_1(\varphi, y+\tilde{\beta}(\varphi, y))\,dy=\int_{\T} \alpha_1(\varphi, x)(1+\beta_x(\varphi, x))\,dx,
\]
hence, it remains to estimate
\begin{equation}
M_{\varphi, x}[b_1]-M_x[b_1]=(M_{\varphi, x}[\alpha_1]-M_x[\alpha_1])+(M_{\varphi, x}[\alpha_1\beta_x]-M_x[\alpha_1\beta_x]).
\end{equation}
The functions $\alpha_1$ and $\alpha_1\beta_x$ are linear combinations of powers of $\Phi_B(T_{\delta})$ (and its derivatives in the $x$-variable), $r_1(T_{\delta})$, $r_0(T_{\delta})$, whose coefficients depend on $c_1, \dots, c_7$ and other real constants. Hence, using the same reasoning adopted in the proof of the estimates \eqref{alpha}, we get
\begin{equation}
\lVert \mathcal{D}_{\omega}^{-1}\{M_{\varphi, x}[\alpha_1]-M_x[\alpha_1]\} \rVert_s\le_s \varepsilon^4\gamma^{-1}+\lVert \mathfrak{I}_{\delta}\rVert_{s+\sigma}
\end{equation}
and the same estimate holds for $\mathcal{D}_{\omega}^{-1}\{M_{\varphi, x}[\beta_x\alpha_1]-M_x[\beta_x \alpha_1]\}$. By following analogous arguments used in the proof of the estimate \eqref{alpha} we conclude.\\

\noindent\textit{Estimate \eqref{estimateM1}}:
By \eqref{8.46} and \eqref{Defp}
\begin{align*}
m_1&=\int_{\T^{\nu+1}} b_1\,dx\,d\varphi+\int_{\T^{\nu+1}} \tilde{c}_1\,dx\,d\varphi.
\end{align*}
Moreover, 
$$\int_{\T^{\nu+1}} b_1\,dx\,d\varphi=\int_{\T^{\nu+1}} (\varepsilon^2\alpha_{1,2}+\mathtt{R}_{a_1})\,dx\,d\varphi+\int_{\T^{\nu+1}} (\mathcal{A}^T-\mathrm{I})\alpha_1\,dx\,d\varphi.$$
Thus, the bound \eqref{estimateM1} comes from taking the maximum between
\[
\left\lvert\int_{\T^{\nu+1}} b_1 \,d\varphi\,dy-\varepsilon^2\int_{\T^{\nu+1}} \left(\alpha_{1, 2}+\alpha_{1, 1}\,(\beta_1)_x\right) d\varphi\,dx\right\rvert\le \lVert \mathtt{R}_{a_1}\rVert_{s_0}+\lVert (\mathcal{A}^T-\mathrm{I})(\alpha_1-\varepsilon \alpha_{1, 1}) \rVert_{s_0}\le \varepsilon^3
\]
and $\lVert \tilde{c}_1 \rVert_{s_0}\le \varepsilon^7\gamma^{-2}=\varepsilon^{3-2 a}$.\\

\noindent\textit{Estimate \eqref{dK}}: We observe that, by \eqref{NewTermVarquadro},
\[
\tilde{d}_0:=\varepsilon^2 (\mathcal{A}^T-I) \alpha_{0, 2}+\mathcal{A}^T \mathtt{R}_0+\mathcal{R}_{\tilde{\beta}}+(\mathcal{T}^{-1}-\mathrm{I}) b_0+\mathcal{T}^{-1} \tilde{c}_0.
\]
By Lemma \ref{lemmacomp}, \ref{ChangeofVariablesLemma} we have the following bounds
\begin{align*}
\lVert \varepsilon^2 (\mathcal{A}^T-I) \alpha_{0, 2} \rVert_s\le_s \varepsilon^3 (1+\lVert \mathfrak{I}_{\delta}\rVert_{s+\sigma}), &\qquad \lVert \mathcal{A}^T \mathtt{R}_0 \rVert_s^{Lip(\gamma)}\le_s \varepsilon^3+ \varepsilon \lVert \mathfrak{I}_{\delta} \rVert_{s+\sigma}^{Lip(\gamma)},\\
\lVert \mathcal{R}_{\tilde{\beta}} \rVert_s^{Lip(\gamma)}\le_s \varepsilon^3+\varepsilon^{1+b} \lVert \mathfrak{I}_{\delta} \rVert_{s+\sigma}^{Lip(\gamma)}, &\qquad \lVert \mathcal{T}^{-1} \tilde{c}_0 \rVert_s^{Lip(\gamma)}\le_s \varepsilon^7 \gamma^{-2}+\varepsilon \lVert \mathfrak{I}_{\delta} \rVert_{s+\sigma}^{Lip(\gamma)},\\
\lVert (\mathcal{T}^{-1}-\mathrm{I}) b_0 \rVert_s\le_s \varepsilon^7 \gamma^{-2}+\varepsilon\lVert \mathfrak{I} \rVert_{s+\sigma}&.
\end{align*}
From these estimates we get \eqref{dK} for $k=0$. The estimate for $k=1$ can be obtained in the same way, considering that $\omega\cdot \partial_{\vartheta} p=O(\varepsilon^6\gamma^{-1})$ in low norm by \eqref{p}.
\end{proof}

\subsection{Linear Birkhoff Normal Form (Step one)}
Let us collect all the terms of order $\varepsilon$ and $\varepsilon^2$ of $\mathcal{L}_3$ (see \eqref{OriginalL3}) in the operators
\begin{equation}\label{B1B2}
\begin{aligned}
&\mathfrak{B}_1[h]:=\alpha_{1, 1}\,\partial_x h+\alpha_{0, 1}\,h=\partial_x \{(2 c_2 v_{xx}-6 c_3 v)\,h\},\\
&\mathfrak{B}_2[h]:=\{\alpha_{1, 2} -(\alpha_{1,1})_x\, \beta_1\}\,\partial_x h+\{\alpha_{0, 2} -(\alpha_{0, 1})_x\, \beta_1\}\,h-\partial_x \overline{\mathcal{R}}_2[h].
\end{aligned}
\end{equation}
Note that $\mathfrak{B}_1$ and $\mathfrak{B}_2$ are not the linear Hamiltonian vector fields of $H_S^{\perp}$ generated, respectively, by the Hamiltonians $R(v^2 z)$ and $R(v^2 z^2)$ in \eqref{Hamiltonians} at $v=\overline{v}$, as expected. Indeed, as we said in Remark \ref{ordiniEpsilonCanc}, some Hamiltonians of type $R(v^2 z)$ have been eliminated by the diffeomorphism of the torus $\Phi$ defined in Section $8.1$, and also the Hamiltonians $R(v^2 z^2)$ have been modified by that.\\
Renaming $\vartheta=\varphi, z=x$ we have
\begin{equation}\label{L3}
\mathcal{L}_3=\Pi_S^{\perp} (\omega\cdot\partial_{\varphi}+m_3 \partial_{xxx}+\varepsilon \mathfrak{B}_1+\varepsilon^2 \mathfrak{B}_2+\tilde{d}_1 \partial_x+\tilde{d}_0)\Pi_S^{\perp}+\tilde{\mathcal{R}}_*
\end{equation}
where $\tilde{d}_1, \tilde{d}_0, \tilde{\mathcal{R}}_*$ are defined in \eqref{DefdK} and \eqref{RtildeStar}.\\
The aim of this section is to eliminate $\mathfrak{B}_1$ from \eqref{L3}. In the next section we shall normalize the term $\mathfrak{B}_2$.\\
We conjugate $\mathcal{L}_3$ with a symplectic operator $\Phi_1\colon H^s_{S^{\perp}}(\T^{\nu+1})\to H^s_{S^{\perp}}(\T^{\nu+1})$ of the form
\begin{equation}\label{Fi1}
\Phi_1:=\exp(\varepsilon A_1)=\mathrm{I}_{H_S^{\perp}}+\varepsilon A_1+\varepsilon^2 \frac{A_1^2}{2}+\varepsilon^3 \hat{A}_1, \quad \hat{A}_1:=\sum_{k\geq 3} \frac{\varepsilon^{k-3}}{k!}\,A_1^k,
\end{equation}
where $A_1(\varphi) h=\sum_{j, j'\in S^c} (A_1)_j^{j'}(\varphi)\, h_{j'} \,e^{\mathrm{i} j x}$ is a Hamiltonian vector field. The map $\Phi_1$ is symplectic, because it is the time$-1$ flow of a Hamiltonian vector field. Therefore
\begin{equation}\label{8.79}
\begin{aligned}
&\mathcal{L}_3 \Phi_1-\Phi_1 \Pi_S^{\perp} (\mathcal{D}_{\omega}+m_3 \partial_{xxx}) \Pi_S^{\perp}=\\
&=\Pi_S^{\perp} (\varepsilon\{\mathcal{D}_{\omega} A_1+m_3 [\partial_{xxx}, A_1]+\mathfrak{B}_1\}+\varepsilon^2\{\mathfrak{B}_1 A_1+\mathfrak{B}_2+\frac{1}{2} m_3 [\partial_{xxx}, A_1^2]+\frac{1}{2} (\mathcal{D}_{\omega}A_1^2)\}+\tilde{d}_1\partial_x+R_3) \Pi_S^{\perp}
\end{aligned}
\end{equation}
where 
\begin{equation}
R_3:=\tilde{d}_1 \partial_x (\Phi_1-\mathrm{I})+\tilde{d}_0 \Phi_1+\tilde{\mathcal{R}}_* \Phi_1+\varepsilon^2 \mathfrak{B}_2 (\Phi_1-\mathrm{I})+\varepsilon^3 \{ \mathcal{D}_{\omega}\hat{A}_1+m_3 [\partial_{xxx}, \hat{A}_1]+\frac{1}{2} \mathfrak{B}_1 A_1^2+\varepsilon \mathfrak{B}_1 \hat{A}_1\}.
\end{equation}
\begin{remark}
$R_3$ has no longer the form \eqref{VeraFinDimForm}. However $R_3=O(\partial_x^0)$ because $A_1=O(\partial_x^{-1})$ and therefore $\Phi_1-\mathrm{I}_{H_S^{\perp}}=O(\partial_x^{-1})$. Moreover the matrix decay norm of $R_3$ is $o(\varepsilon^2)$.
\end{remark}
In order to eliminate the order $\varepsilon$ from \eqref{8.79}, we choose
\begin{equation}\label{DefdiA1jj}
(A_1)_j^{j'}(l)=\begin{cases}
-\dfrac{(\mathfrak{B}_1)_j^{j'}(l)}{i(\omega\cdot l+m_3(j'^3-j^3))} \qquad \mbox{if}\,\,\overline{\omega}\cdot l+j'^3-j^3\neq 0, \qquad j, \,j'\in S^c, \,l\in\mathbb{Z}^{\nu}\\[3mm]
0 \qquad \qquad \qquad \qquad \qquad \qquad \mbox{otherwise}
\end{cases}
\end{equation}
This definition is well posed. Indeed, by \eqref{vSegnato} and \eqref{B1B2} 
\begin{equation}\label{DefdiB1jj}
(\mathfrak{B}_1)_j^{j'}(l):=\begin{cases}
-2 i j\,c_2\,(j-j')^2\,\sqrt{\lvert j-j'\rvert \xi_{j-j'}}-6 i j\,c_3\,\sqrt{\lvert j-j'\rvert \xi_{j-j'}} \qquad \mbox{if} \,\,\,j-j'\in S,\quad l=\mathtt{l}(j-j')\\
0 \qquad \qquad \qquad \qquad \qquad\qquad\qquad\qquad\qquad\qquad\qquad\qquad\quad\mbox{otherwise}.
\end{cases}
\end{equation}
In particular $(\mathfrak{B}_1)_j^{j'}(l)=0$ unless $\lvert l \rvert\le 1$. Thus, for $(l, j, j')$ such that $\overline{\omega}\cdot l+j'^3-j^3\neq 0$, the denominators in \eqref{DefdiA1jj} satisfy
\begin{equation}\label{8.83}
\begin{aligned}
\lvert \omega\cdot l +m_3 (j'^3-j^3) \rvert&=\lvert m_3 (\overline{\omega}\cdot l+j'^3-j^3)+(\omega-m_3 \overline{\omega})\cdot l \rvert\geq \\
&\geq \lvert m_3 \rvert\,\lvert \overline{\omega}\cdot l+j'^3-j^3 \rvert-\lvert \omega-m_3 \overline{\omega} \rvert\,\lvert l \rvert\geq 1/2, \,\,\,\forall \lvert l \rvert\le 1
\end{aligned}
\end{equation}
for $\varepsilon$ small enough, since $m_3-1$ and $\omega-\overline{\omega}$ are $O(\varepsilon^2)$.
$A_1$ defined in \eqref{DefdiA1jj} is a Hamiltonian vector field as $\mathfrak{B}_1$.
\begin{lem}{(Lemma $8.16$ in \cite{KdVAut})}\label{Lemma8.16}
If $j, j'\in S^c, j-j'\in S, l=\mathtt{l}(j-j')$, then 
$$\overline{\omega}\cdot l+j'^3-j^3=3\,j\, j'\,(j'-j)\neq 0.$$
\end{lem}
\begin{cor}{(Corollary $8.17$ in \cite{KdVAut})}
Let $j, j'\in S^c$. If $\overline{\omega}\cdot l+j'^3-j^3=0$ then $(\mathfrak{B}_1)_j^{j'}=0$.
\end{cor}
By \eqref{DefdiA1jj} and the previous corollary, the term of order $\varepsilon$ in \eqref{8.79} is
\begin{equation}\label{HomologicalEquation}
\Pi_S^{\perp} (\mathcal{D}_{\omega} A_1+m_3 [\partial_{xxx}, A_1]+\mathfrak{B}_1)\Pi_S^{\perp}=0.
\end{equation}
We now prove that $A_1$ is a bounded transformation.
\begin{lem}{(Lemma $8.18$ in \cite{KdVAut})}
\begin{itemize}
\item[(i)] For all $l\in\mathbb{Z}^{\nu}, j, j'\in S^c$,
\begin{equation}
\lvert (A_1)_j^{j'}(l)\rvert\le C (\lvert j \rvert+\lvert j' \rvert)^{-1}, \quad \lvert (A_1)_j^{j'}(l) \rvert^{lip}\le \varepsilon^{-2} (\lvert j \rvert+\lvert j'\rvert)^{-1}.
\end{equation}
\item[(ii)] $(A_1)_j^{j'}(l)=0$ for all $l\in\mathbb{Z}^{\nu}, j, j'\in S^c$ such that $\lvert j-j'\rvert>C_S$, where $C_S:=\max\{ \lvert j \rvert : j\in S\}$.
\end{itemize}
\end{lem}
The previous lemma means that $A=O(\partial_x^{-1})$. More precisely, we deduce that
\begin{lem}{(Lemma $8.19$ in \cite{KdVAut})}
$\lvert A_1 \partial_x \rvert_s^{Lip(\gamma)}+\lvert \partial_x A_1 \rvert_s^{Lip(\gamma)}\le C(s)$.
\end{lem}
It follows that the symplectic map $\Phi_1$ in \eqref{Fi1} is invertible for $\varepsilon$ small, with inverse
\begin{equation}
\Phi_1^{-1}=\exp(-\varepsilon A_1)=\mathrm{I}_{H_S^{\perp}}+\varepsilon \check{A}_1, \,\, \check{A}_1:=\sum_{n\geq 1} \frac{\varepsilon^{n-1}}{n!} (-A_1)^n,\,\,\lvert \check{A}_1 \partial_x \rvert_s^{Lip(\gamma)}+\lvert \partial_x \check{A}_1 \rvert_s^{Lip(\gamma)}\le C(s).
\end{equation}
Since $A_1$ solves the homological equation \eqref{HomologicalEquation}, the $\varepsilon$-term in \eqref{L3} is zero, and, with a straightforward calculation, the $\varepsilon^2$-term simplifies to $\mathfrak{B}_2+\frac{1}{2}[\mathfrak{B}_1, A_1]$. We obtain the Hamiltonian operator
\begin{align}
&\mathcal{L}_4:=\Phi_1^{-1} \mathcal{L}_3 \Phi_1=\Pi_S^{\perp} (\mathcal{D}_{\omega}+m_3 \partial_{xxx}+\tilde{d}_1 \partial_x+\varepsilon^2\{ \mathfrak{B}_2+\frac{1}{2}[\mathfrak{B}_1, A_1] \}+\tilde{R}_4)\Pi_S^{\perp},\label{L4}\\
&\tilde{R}_4:=(\Phi_1^{-1}-\mathrm{I}) \Pi_S^{\perp} [\varepsilon^2(\mathfrak{B}_2+\frac{1}{2}[\mathfrak{B}_1, A_1])+\tilde{d}_1 \partial_x ]+\Phi_1^{-1}\Pi_S^{\perp} R_3.
\end{align}
We split $A_1$ defined in \eqref{DefdiA1jj}, \eqref{DefdiB1jj} into $A_1=\overline{A}_1+\tilde{A}_1$ where, for all $j, j'\in S^c, l\in\mathbb{Z}^{\nu}$,
\begin{equation}
(\overline{A}_1)_j^{j'}(l):=-\frac{2  j\,c_2\,(j-j')^2\,\sqrt{\lvert j-j'\rvert \xi_{j-j'}}+6 j\,c_3\,\sqrt{\lvert j-j'\rvert \xi_{j-j'}} }{\overline{\omega}\cdot l+j'^3-j^3}
\end{equation}
if $
\overline{\omega}\cdot l+j'^3-j^3\neq 0$, $j-j'\in S,\,\,l=\mathtt{l}(j-j')$,
and $(\overline{A}_1)_j^{j'}(l):=0$ otherwise.\\ By Lemma \ref{Lemma8.16}, for all $j, j'\in S^c, l\in\mathbb{Z}^{\nu}$,
\begin{equation}
(\overline{A}_1)_j^{j'}(l)=\begin{cases}
-\dfrac{2}{3}\,c_2\,\left(\dfrac{j-j'}{j'}\right) \sqrt{\lvert j-j'\rvert \xi_{j-j'}}-2\,c_3\,\dfrac{1}{j' (j'-j)}\,\sqrt{\lvert j-j'\rvert \xi_{j-j'}} \qquad \mbox{if}\,\,j-j'\in S,\\[3mm]
0 \qquad\qquad\qquad\qquad\qquad  \qquad \qquad \qquad \qquad \qquad\qquad \qquad\qquad\qquad\mbox{otherwise},
\end{cases}
\end{equation}
namely
\begin{equation}\label{A1segnato}
\overline{A}_1 h=-\frac{2}{3}c_2 \Pi_S^{\perp}[\overline{v}_x\,(\partial_x^{-1} h)] +2\,c_3 \Pi_S^{\perp}[(\partial_x^{-1} \overline{v}) (\partial_x^{-1} h)], \quad \forall h\in H^s_{S^{\perp}}(\T^{\nu+1}).
\end{equation}
The difference is 
\begin{equation}
(\tilde{A}_1)_j^{j'}(l):=-\frac{(2 c_2\, j\, (j-j')^2+6\,c_3\, j) \sqrt{\lvert j-j'\rvert \xi_{j-j'}}\{ (\omega-\overline{\omega})\cdot l+(m_3-1) (j'^3-j^3) \}}{(\omega\cdot l+m_3 (j'^3-j^3))(\overline{\omega}\cdot l+j'^3-j^3)}
\end{equation}
for $j, j'\in S^c, j-j'\in S, l=\mathtt{l}(j-j')$, and $(\tilde{A}_1)_j^{j'}(l)=0$ otherwise. Then, by \eqref{L4},
\begin{equation}
\mathcal{L}_4=\Pi_S^{\perp} (\mathcal{D}_{\omega}+ m_3 \partial_{xxx}+\tilde{d}_1\,\partial_x+\varepsilon^2 T+R_4)\Pi_S^{\perp},
\end{equation}
where
\begin{equation}\label{T}
T:=\mathfrak{B}_2+\frac{1}{2} [\mathfrak{B}_1, \overline{A}_1], \quad R_4:=\frac{\varepsilon^2}{2} [\mathfrak{B}_1, \tilde{A}_1]+\tilde{R}_4.
\end{equation}
The operator $T$ is Hamiltonian as $\mathfrak{B}_1, \mathfrak{B}_2, \overline{A}_1$, because the commutator of two Hamiltonian vector fields is Hamiltonian.
\begin{lem}\label{PossiblyLarg}
There is $\sigma=\sigma(\nu, \tau)>0$ (possibly larger than in Lemma \ref{LemmaFond2}) such that
\begin{equation}\label{EstimateOnR4}
\lvert R_4 \rvert_s^{Lip(\gamma)}\le_s \varepsilon^7 \gamma^{-2}+\varepsilon \lVert \mathfrak{I}_{\delta} \rVert_{s+\sigma}^{Lip(\gamma)}, \quad \lvert \partial_i R_4 [\hat{\imath}] \rvert_s\le_s \varepsilon (\lVert \hat{\imath} \rVert_{s+\sigma}+\lVert \mathfrak{I}_{\delta} \rVert_{s+\sigma} \lVert \hat{\imath} \rVert_{s_0+\sigma}).
\end{equation}
\end{lem}
\begin{proof}
The proof follows the one of Lemma $8. 20$ in \cite{KdVAut}. The only difference is the estimate on the coefficient $\tilde{d}_0$ (see \eqref{dK}), that gives the term of size $\varepsilon^{7} \gamma^{-2}$ in \eqref{EstimateOnR4}, instead of $\varepsilon^5 \gamma^{-1}$ in the inequality $(8. 95)$ in \cite{KdVAut}. 
\end{proof}

\subsection{Linear Birkhoff Normal form (Step two)}
The goal of this section is to normalize the term $\varepsilon^2 T$ from the operator $\mathcal{L}_4$ defined in \eqref{L4}. We cannot eliminate the terms $O(\varepsilon^2)$ at all, because some harmonics of $\varepsilon^2 T$, which correspond to null divisors, are not naught.\\
We conjugate the Hamiltonian operator $\mathcal{L}_4$ via a symplectic map
\begin{equation}
\Phi_2:=\exp (\varepsilon^2 A_2)=\mathrm{I}_{H_S^{\perp}}+\varepsilon^2 A_2+\varepsilon^4 \hat{A}_2, \quad \hat{A}_2:=\sum_{k\geq 2} \frac{\varepsilon^{2(k-2)}}{k!}\,A_2^k
\end{equation}
where $A_2(\varphi)=\sum_{j, j'\in S^c}(A_2)_j^{j'}(\varphi) h_{j'} e^{\mathrm{i} j x} $ is a Hamiltonian vector field. We compute
\begin{align}
&\mathcal{L}_4 \Phi_2-\Phi_2 \Pi_S^{\perp} (\mathcal{D}_{\omega}+m_3 \partial_{xxx})\Pi_S^{\perp}=\Pi_S^{\perp} (\varepsilon^2\{\mathcal{D}_{\omega} A_2+m_3 [\partial_{xxx}, A_2]+T  \}+\tilde{d}_1 \partial_x+\tilde{R}_5)\Pi_S^{\perp},\\
&\tilde{R}_5:=\Pi_S^{\perp} \{ \varepsilon^4 (\mathcal{D}_{\omega} \hat{A}_2+m_3 [\partial_{xxx}, \hat{A}_2])+(\tilde{d}_1 \partial_x+\varepsilon^2 T)(\Phi_2-\mathrm{I})+R_4 \Phi_2\}\Pi_S^{\perp}.
\end{align}
We define
\begin{equation}\label{A2jj}
(A_2)_j^{j'}(l):=\begin{cases}
-\dfrac{T_j^{j'}(l)}{\mathrm{i}(\omega\cdot l+m_3 (j'^3-j^3))} \quad \mbox{if}\,\,\overline{\omega}\cdot l+j'^3-j^3\neq 0,\\[3mm]
0 \qquad\qquad\qquad\qquad\qquad\quad\mbox{otherwise}.
\end{cases}
\end{equation}
The definition is well posed. Indeed the matrix entries $T_j^{j'}(l)=0$ for all $\lvert j-j' \rvert>2 C_S, l\in\mathbb{Z}^{\nu}$, where $C_S:=\max\{ \lvert j \rvert : j\in S\}$. Also $T_j^{j'}(l)=0$ for all $j, j'\in S^c, \lvert l \rvert>2$. Thus, arguing as in \eqref{8.83}, if $\overline{\omega}\cdot l+j'^3-j^3\neq 0$, then $\lvert \omega\cdot l+m_3 (j'^3-j^3)\rvert\geq 1/2$. The operator $A_2$ is a Hamiltonian vector field because $T$ is Hamiltonian.
\subsubsection*{Resonant terms}
Now we compute the terms of $\varepsilon^2 T$ that cannot be removed by the Birkhoff map $\Phi_2$.\\
By \eqref{A1segnato}, \eqref{T} we get, for $h\in H^s_{S^{\perp}}$,
\begin{align*}
 \mathfrak{B}_1\,\overline{A}_1 [h]&=-\frac{4}{3} c_2^2\,\partial_x \Pi_S^{\perp} [\overline{v}_{xx}\,\Pi_S^{\perp}[\overline{v}_x\,(\partial_x^{-1} h)]]+4 c_2 c_3 \partial_x \Pi_S^{\perp}[\overline{v}_{xx}\Pi_S^{\perp}[(\partial_x^{-1}\overline{v})(\partial_x^{-1} h)]]\\
 &+4 c_2 c_3 \partial_x \Pi_S^{\perp}[\overline{v}\Pi_S^{\perp}[\overline{v}_x \,(\partial_x^{-1} h)]]-12 c_3^2 \partial_x\Pi_S^{\perp}[\overline{v}\Pi_S^{\perp}[(\partial_x^{-1} \overline{v})(\partial_x^{-1} h)]]\\
\overline{A}_1 \mathfrak{B}_1 [h]&=-\frac{4}{3} c_2^2\,\Pi_S^{\perp}[\overline{v}_x \,\Pi_S^{\perp}[\overline{v}_{xx} h]]+4 c_2 c_3 \Pi_S^{\perp}[\overline{v}_x \,\Pi_S^{\perp}[\overline{v} \,h]]\\
&+4 c_2 c_3 \Pi_S^{\perp}[(\partial_x^{-1} \overline{v})\Pi_S^{\perp}[\overline{v}_{xx}\,h]]-12 c_3^2 \Pi_S^{\perp}[(\partial_x^{-1}\overline{v})\Pi_S^{\perp}[\overline{v}\,h]]
\end{align*}
whence, for all $j, j'\in S^c, l\in\mathbb{Z}^{\nu}$,
\begin{equation}
\begin{aligned}
([\mathfrak{B}_1, \overline{A}_1])_j^{j'}(l)&=\frac{4}{3} c_2^2\,\mathrm{i} \sum_{\substack{j_1, j_2\in S, j_1+j_2=j-j',\\j'+j_2\in S^c, \mathtt{l}(j_1)+\mathtt{l}(j_2)=l}} \left( \frac{j\,j_1^2\,j_2-j_1\,j_2^2\,j'}{j'} \right)\,\sqrt{\lvert j_1\,j_2\rvert\xi_{j_1} \xi_{j_2}}\\[2mm]
&+4 c_2 c_3 \,\mathrm{i}\,\sum_{\substack{j_1, j_2\in S, j_1+j_2=j-j',\\j'+j_2\in S^c, \mathtt{l}(j_1)+\mathtt{l}(j_2)=l}} \left( \frac{-j\, j_1^3+j\,j_1\,j_2^2-j_1^2\, j_2\,j'-j_2^3\,j'}{j' j_1 j_2} \right)\,\sqrt{\lvert j_1\,j_2\rvert \xi_{j_1} \xi_{j_2}}\\[2mm]
&+12\,c_3^2\,\mathrm{i}\,\sum_{\substack{j_1, j_2\in S, j_1+j_2=j-j',\\j'+j_2\in S^c, \mathtt{l}(j_1)+\mathtt{l}(j_2)=l}} \left( \frac{j j_1-j' j_2}{j' j_1 j_2} \right) \,\sqrt{\lvert j_1 \,j_2\rvert \xi_{j_1} \xi_{j_2}}.
\end{aligned}
\end{equation}
If $([\mathfrak{B}_1, \overline{A}_1])_j^{j'}(l)\neq 0$ there are $j_1, j_2\in S$ such that $j_1+j_2=j-j', j'+j_2\in S^c, \mathtt{l}(j_1)+\mathtt{l}(j_2)=l$. Then
\begin{equation}\label{risonanza}
\overline{\omega}\cdot l+j'^3-j^3=\overline{\omega}\cdot \mathtt{l}(j_1)+\overline{\omega}\cdot \mathtt{l}(j_2)+j'^3-j^3=j_1^3+j_2^3+j'^3-j^3.
\end{equation}
Thus, if $\overline{\omega}\cdot l+j'^3-j^3=0$, Lemma \eqref{AlgebraicLemma} implies that $(j_1+j_2)(j_1+j')(j_2+j')=0$. Now $j_1+j', j_2+j'\neq 0$ because $j_1, j_2\in S, j'\in S^c$ and $S$ is symmetric. Hence $j_1+j_2=0$, which implies $j=j'$ and $l=0$. In conclusion, if $\overline{\omega}\cdot l+j'^3-j^3=0$, the only nonzero matrix entry $([\mathfrak{B}_1, \overline{A}_1])_j^{j'}(l)$ is
\begin{equation}\label{Commutatore}
\begin{aligned}
\frac{1}{2}([\mathfrak{B}_1, \overline{A}_1])_j^j(0)&=\frac{4}{3} c_2^2\,\mathrm{i}\,\sum_{j_2\in S, j_2+j\in S^c} j_2^3\,\lvert j_2 \rvert \xi_{j_2}+8 c_2 c_3\,\mathrm{i} \sum_{j_2\in S, j_2+j\in S^c} j_2\,\lvert j_2 \rvert \xi_{j_2}\\
& +12\,c_3^2 \mathrm{i} \sum_{j_2\in S, j_2+j\in S^c} j_2^{-1}\,\lvert j_2 \rvert\,\xi_{j_2}.
\end{aligned}
\end{equation}
Now consider $\mathfrak{B}_2$ defined in \eqref{B1B2}. We split $\mathfrak{B}_2=B_1+B_2+B_3+B_4+B_5$, where
\begin{equation}\label{iB's}
\begin{aligned}
&B_1 [h]:=\alpha_{1, 2} \,h_x, \quad B_2 [h]:=\alpha_{0, 2}\,h, \quad B_3 [h]:=-(\alpha_{1,1})_x\,\beta_1\,h_x,\\[2mm]
& B_4 [h]:=-(\alpha_{0, 1})_x\,\beta_1, \quad B_5[h]:=-\partial_x \overline{\mathcal{R}}_2 [h].
\end{aligned}
\end{equation}
We denote by $(\alpha)_{j, l}$ the $(j, l)$-th Fourier coefficient of $\alpha(\varphi, x)$ as function of time and space. The Fourier representation of $B_i, i=1, \dots, 4$ in \eqref{iB's} is
\begin{align*}
(B_1)_j^{j'}(l)&=\mathrm{i}\,j'\,(\alpha_{1, 2})_{\substack{j-j', \mathtt{l}(j-j')}}, \qquad (B_2)_j^{j'}(l)=(\alpha_{0, 2})_{\substack{j-j', \mathtt{l}(j-j')}}\\
(B_3)_j^{j'}(l)&=4 c_1 c_2 \mathrm{i}\,j' (\overline{v}_{xxx} \overline{v})_{j-j', \mathtt{l}(j-j')}+\frac{4}{3} c_2^2 \mathrm{i} j' (\overline{v}_{xxx} (\partial_x^{-1} \overline{v}))_{j-j', \mathtt{l}(j-j')}\\
&-12 c_1 c_3 \mathrm{i}\,j'\,(\overline{v} \,\overline{v}_x)_{j-j', \mathtt{l}(j-j')}-4 c_2 c_3 \mathrm{i}\,j' (\overline{v}_x (\partial_x^{-1} \overline{v}))_{j-j', \mathtt{l}(j-j')} ,\\
(B_4)_j^{j'}(l)&=4 c_1 c_2 (\overline{v}_{xxxx} \overline{v})_{j-j', \mathtt{l}(j-j')}+\frac{4}{3} c_2^2 (\overline{v}_{xxxx} (\partial_x^{-1} \overline{v}))_{j-j', \mathtt{l}(j-j')}\\
&-12 c_1 c_3 (\overline{v} \overline{v}_{xx})_{j-j', \mathtt{l}(j-j')}-4 c_2 c_3 (\overline{v}_{xx} (\partial_x^{-1} \overline{v}))_{j-j', \mathtt{l}(j-j')}
\end{align*}
If $(B_k)_j^{j'}(l)\neq 0$, $k=1, \dots, 4$ there are $j_1, j_2\in S$ such that $j_1+j_2=j-j', l=\mathtt{l}(j_1)+\mathtt{l}(j_2)$ and \eqref{risonanza} holds. Thus, if $\overline{\omega}\cdot l+j'^3-j^3=0$, Lemma \eqref{AlgebraicLemma} implies that $(j_1+j_2)(j_1+j')(j_2+j')=0$, and, since $j'\in S^c$ and $S$ is symmetric, the only possibility is $j_1+j_2=0$. Hence $j=j'$, $l=0$. In conclusion, if $\overline{\omega}\cdot l+j'^3-j^3=0$, the only nonzero matrix element $(B_i)_j^{j'}(l), i=1,\dots, 4$, by \eqref{Averages}, is
\begin{equation}
\begin{aligned}
&(B_1)_j^j(0)=\mathrm{i} j\,\sum_{k\in S} (-2 c_6 k^2-12 c_7+\frac{4}{3} c_2^2 k^2+4 c_2 c_3 )\,\lvert k \rvert \xi_k, \,\,\, (B_2)_j^j(0)=\sum_{k\in S} (-4 c_1 c_2\,k^4-12 c_1 c_3\,k^2)\,\lvert k \rvert \xi_k,\\
&(B_3)_j^j(0)=\mathrm{i} j\,\sum_{k\in S}(\frac{4}{3} c_2^2\,k^2+4 c_2 c_3)\lvert k\rvert \xi_k, \quad (B_4)_j^j(0)=\sum_{k\in S} (4 c_1 c_2\,k^4+12 c_1 c_3\,k^2)\,\lvert k \rvert\xi_k
\end{aligned}
\end{equation} 
We note that $c(\xi)$ defined in \eqref{OriginalCdiXi} is equal to $-\mathrm{i}\sum_{i=1}^4j^{-1}\,(B_i)_j^j(0)$ (observe that the term $j^{-1}\,(B_i)_j^j(0)$ is independent of $j$) and we write
\begin{equation}\label{cXi}
\begin{aligned}
c(\xi)&=\sum_{k\in S^+} (- 4 c_6\,k^3-24 c_7 k+\frac{16}{3} c_2^2\,k^3+16 c_2 c_3 k )\,\xi_k\\[2mm]
&=(\frac{16}{3} c_2^2-4 c_6) v_3\cdot \xi+(16 c_2 c_3-24 c_7) v_1\cdot \xi,
\end{aligned}
\end{equation}
where $v_3\cdot \xi=\sum_{j\in S^+} j^3\,\xi_j$ and $v_1\cdot \xi=\sum_{j\in S^+} j\,\xi_j$.\\
As before, the only possibility to get a zero at the denominator of \eqref{A2jj} is $j_1+j_2=0$. Therefore
\begin{equation}\label{LaCosa}
\begin{aligned}
(B_5)_j^j(0)&=\frac{4}{3} c_2^2 \mathrm{i} \sum_{j_2\in S, j_2+j\in S} j_2^3\,\lvert j_2 \rvert \xi_{j_2}+8 c_2 c_3 \mathrm{i} \sum_{j_2\in S, j_2+j\in S} j_2 \,\lvert j_2 \rvert \xi_{j_2}\\
&+12 c_3^2 \mathrm{i}\sum_{j_2\in S, j_2+j\in S} j_2^{-1} \,\lvert j_2 \rvert \xi_{j_2}.
\end{aligned}
\end{equation}
We note that for every odd function $f\colon S\to \mathbb{Z}$, by the simmetry of $S$, we have $\sum_{j_2\in S} f(j_2)\,\xi_{j_2}=0$.
Thus, by \eqref{Commutatore} and \eqref{LaCosa}, we get 
\begin{equation*}
(B_5)_j^j(0)+\frac{1}{2}([\mathfrak{B}_1, \overline{A}_1])_j^j(0)=\frac{4}{3} c_2^2 \mathrm{i} \sum_{j_2\in S} j_2^3\,\lvert j_2 \rvert \xi_{j_2}+8 c_2 c_3 \mathrm{i} \sum_{j_2\in S} j_2\,\lvert j_2 \rvert \xi_{j_2}+12 c_3^2 \mathrm{i} \sum_{j_2\in S} j_2^{-1}\,\lvert j_2 \rvert \xi_{j_2}=0.
\end{equation*}
Finally, we have
\begin{align}
&\mathcal{L}_5:=\Phi_2^{-1} \mathcal{L}_4 \Phi_2=\Pi_S^{\perp} (\mathcal{D}_{\omega}+m_3 \partial_{xxx}+(\tilde{d}_1+\varepsilon^2 c(\xi))\,\partial_x + R_5)\Pi_S^{\perp},\label{L5}\\
&R_5:=(\Phi_2^{-1}-\mathrm{I})\Pi_S^{\perp} (\tilde{d}_1+\varepsilon^2 c(\xi))\partial_x+\Phi_2^{-1} \Pi_S^{\perp} \tilde{R}_5.
\end{align}
\begin{lem}
$R_5$ satisfies the same estimates \eqref{EstimateOnR4} as $R_4$ (with a possibly larger $\sigma$).
\end{lem}

\subsection{Descent method}
The goal of this section is to transform $\mathcal{L}_5$ in \eqref{L5} in order to make constant the coefficient in front of $\partial_x$. We conjugate $\mathcal{L}_5$ via a symplectic map of the form
\begin{equation}\label{S}
\mathcal{S}:=\exp(\Pi_S^{\perp}(w \partial_x^{-1}))\Pi_S^{\perp}=\Pi_S^{\perp}(\mathrm{I}+w \partial_x^{-1})\Pi_S^{\perp}+\hat{\mathcal{S}}, \quad \hat{\mathcal{S}}:=\sum_{k\geq 2} \frac{1}{k!} [\Pi_S^{\perp} (w\partial_x^{-1})]^k \Pi_S^{\perp},
\end{equation}
where $w\colon \T^{\nu+1}\to \mathbb{R}$ is a function. Note that $\Pi_S^{\perp} (w \partial_x^{-1})\Pi_S^{\perp}$ is the Hamiltonian vector field generated by $-\frac{1}{2} \int_{\T} w (\partial_x^{-1} h)^2\,dx, h\in H_S^{\perp}$. We calculate
\begin{align}
\mathcal{L}_5 \mathcal{S}&-\mathcal{S} \Pi_S^{\perp} (\mathcal{D}_{\omega}+ m_3 \partial_{xxx}+m_1 \partial_x) \Pi_S^{\perp}=\Pi_S^{\perp} (3 m_3 w_x+\tilde{d}_1+\varepsilon^2 c(\xi)-m_1) \partial_x \Pi_S^{\perp}+\tilde{R}_6,\label{L5}\\\notag
\tilde{R}_6:&=\Pi_S^{\perp} \{ (3 m_3 w_{xx}+(\tilde{d}_1+\varepsilon^2 c(\xi))\Pi_S^{\perp} w-m_1 w)\pi_0+(\mathcal{D}_{\omega} w+m_3 w_{xxx}+(\tilde{d}_1+\varepsilon^2 c(\xi))\Pi_S^{\perp} w_x) \partial_x^{-1}\\\notag
&+\mathcal{D}_{\omega} \hat{S}+m_3 [\partial_{xxx}, \hat{S}]+(\tilde{d}_1+\varepsilon^2 c(\xi)) \partial_x \hat{S}-m_1 \hat{S} \partial_x+R_5 \mathcal{S}\} \Pi_S^{\perp}  \notag
\end{align}
where $\tilde{R}_6$ collects all the bounded terms. By \eqref{DefdK}, \eqref{OriginalCdiXi}, we solve 
\[
3 m_3 w_x+\tilde{d}_1+\varepsilon^2 c(\xi)-m_1=0
\]
choosing $w:=-(3 m_3)^{-1} \partial_x^{-1} (\tilde{d}_1+\varepsilon^2 c(\xi)-m_1)$.
 For $\varepsilon$ sufficiently small, the operator $\mathcal{S}$ is invertible and, by \eqref{L5},
\begin{equation}\label{L6}
\mathcal{L}_6:=\mathcal{S}^{-1} \mathcal{L}_4 \mathcal{S}=\Pi_S^{\perp} (\mathcal{D}_{\omega}+m_3 \partial_{xxx}+m_1 \partial_x)\Pi_S^{\perp}+R_6, \quad R_6:=\mathcal{S}^{-1} \tilde{R}_6.
\end{equation}
Since $\mathcal{S}$ is symplectic, $\mathcal{L}_6$ is Hamiltonian.
\begin{lem}\label{LemmaS}
There is $\sigma:=\sigma(\nu, \tau)>0$ (possibly larger than in Lemma \ref{PossiblyLarg}) such that
\[
\lvert \mathcal{S}^{\pm 1}-\mathrm{I} \rvert_s^{Lip(\gamma)}\le_s \varepsilon^7 \gamma^{-2}+\varepsilon \lVert \mathfrak{I}_{\delta} \rVert_{s+\sigma}^{Lip(\gamma)}, \quad \lvert \partial_i \mathcal{S}^{\pm 1} [\hat{\imath}]\rvert_s\le_s \varepsilon (\lVert \hat{\imath} \rVert_{s+\sigma}+\lVert \mathfrak{I}_{\delta} \rVert_{s+\sigma} \lVert \hat{\imath} \rVert_{s_0+\sigma}).
\]
The remainder $R_6$ satisfies the same estimates of $R_4$ (with a possibly larger $\sigma$).
\end{lem}
\begin{proof}
By \eqref{m3}, \eqref{estimateM1}, \eqref{dK}, $\lVert w \rVert_s^{Lip(\gamma)}\le_s \varepsilon^7 \gamma^{-2}+\varepsilon \lVert \mathfrak{I}_{\delta} \rVert_{s+\sigma}^{Lip(\gamma)}$, and the lemma follows by the definition of $\mathcal{S}$, see \eqref{S}. Since $\hat{S}=O(\partial_x^{-2})$ the commutator $[\partial_{xxx}, \hat{S}]=O(\partial_x^0)$ and $\lvert [\partial_{xxx}, \hat{S}]  \rvert_s^{Lip(\gamma)}\le_s \lVert w \rVert_{s_0+3}^{Lip(\gamma)}\lVert w\rVert_{s+3}^{Lip(\gamma)}$.
\end{proof}

\subsection{KAM reducibility and inversion of $\mathcal{L}_{\omega}$}
The coefficients $m_3, m_1$ of the operator $\mathcal{L}_6$ in \eqref{L6} are constants, and the remainder $R_6$ is a bounded operator of order $\partial_x^0$ with small matrix decay norm. Then we can diagonalize $\mathcal{L}_6$ by applying the iterative KAM reducibility Theorem $4.2$ in \cite{Airy} along the sequence of scales
\begin{equation}\label{Nn}
N_n:=N_0^{\chi^n}, \quad n=0, 1, 2, \dots, \quad \chi:=3/2, \quad N_0>0.
\end{equation}
In Section $9$, the initial $N_0$ will (slightly) increase to infinity as $\varepsilon\to 0$, see \eqref{SmallnessConditionNM}. The required smallness condition (see $(4.14)$ in \cite{Airy}) is 
\begin{equation}
N_0^{C_0} \lvert R_6 \rvert_{s_0+\beta}^{Lip(\gamma)} \gamma^{-1} \le 1,
\end{equation}
where $\beta=7 \tau+6$ (see $(4.1)$ in \cite{Airy}), $\tau$ is the diophantine exponent in \eqref{0diMelnikov} and \eqref{OmegoneInfinito}, and the constant $C_0:=C_0(\tau, \nu)>0$ is fixed in Theorem $4.2$ in \cite{Airy}. By Lemma \ref{LemmaS}, the remainder $R_6$ satisfies the bound \eqref{EstimateOnR4}, and using \eqref{IpotesiPiccolezzaIdelta} we get
\begin{equation}
\lvert R_6 \rvert_{s_0+\beta}^{Lip(\gamma)}\le C \varepsilon^{7-2 b} \gamma^{-1}=C \varepsilon^{3- 2 a}, \qquad \lvert R_6 \rvert_{s_0+\beta}^{Lip(\gamma)} \gamma^{-1}\le C \varepsilon^{1- 3 a}.
\end{equation}
We use that $\mu$ in \eqref{IpotesiPiccolezzaIdelta} is assumed to satisfy $\mu \geq \sigma+\beta$ where $\sigma:=\sigma(\tau, \nu)$ is given in Lemma \ref{LemmaS}.\\

\begin{teor}{\textbf{(Reducibility)}}\label{Reducibility}
Assume that $\omega \mapsto i_{\delta}(\omega)$ is a Lipschitz function defined on some subset $\Omega_0\subseteq \Omega_{\varepsilon}$ (recall \eqref{OmegaEpsilon}), satisfying \eqref{IpotesiPiccolezzaIdelta} with $\mu \geq \sigma+\beta$ where $\sigma:=\sigma(\tau, \nu)$ is given in Lemma \ref{LemmaS} and $\beta:=7 \tau+6$. Then there exists $\delta_0\in (0, 1)$ such that, if
\begin{equation}\label{PiccolezzaperKamred}
N_0^{C_0} \varepsilon^{7- 2 b} \gamma^{-2}=N_0^{C_0} \varepsilon^{1-3 a}\le \delta_0, \quad \gamma:=\varepsilon^{2+a}, \quad a\in (0, 1/6),
\end{equation}
then
\begin{itemize}
\item[(i)] \textbf{(Eigenvalues)}. For all $\omega\in \Omega_{\varepsilon}$ there exists a sequence
\begin{align}\label{FinalEigenvalues}
&\mu_j^{\infty}(\omega):=\mu_j^{\infty}(\omega, i_{\delta}(\omega)):=-\mathrm{i}\tilde{m}_3(\omega)\,j^3+\mathrm{i}\tilde{m}_1 (\omega) j+r_j^{\infty}(\omega), \quad j\in S^c,
\end{align}
where $\tilde{m}_3, \tilde{m}_1$ coincide with the coefficients of $\mathcal{L}_6$ of \eqref{L6} for all $\omega\in \Omega_0$. Furthermore, for all $j\in S^c$
\begin{equation}\label{stimeautovalfinali}
\begin{aligned}
&\lvert \tilde{m}_3-1 \rvert^{Lip(\gamma)}\le C \varepsilon^2, \quad \lvert \tilde{m}_1-\varepsilon^2 c(\xi) \rvert^{Lip(\gamma)}\le C \varepsilon^{3-2 a},
\end{aligned}
\end{equation}
for some $C>0$. All the eigenvalues $\mu_j^{\infty}$ are purely imaginary. We define, for convenience, $\mu_0^{\infty}(\omega):=0$.
\item[(ii)] \textbf{(Conjugacy)}. For all $\omega$ in the set
\begin{equation}\label{OmegoneInfinito}
\Omega^{2 \gamma}_{\infty}:=\Omega^{2 \gamma}_{\infty}(i_{\delta}):=\left\{ \omega\in\Omega_0 : \lvert \mathrm{i} \omega\cdot l+\mu_j^{\infty}(\omega)-\mu_k^{\infty}(\omega)\rvert\geq \frac{2 \gamma\,\lvert j^3-k^3 \rvert}{\langle l \rangle^{\tau}}, \,\,\forall l\in\mathbb{Z}^{\nu}, \,\,\forall j, k\in S^c \cup \{0\}  \right\}
\end{equation}
there is a real, bounded, invertible, linear operator $\Phi_{\infty}(\omega)\colon H^s_{S^{\perp}}(\T^{\nu+1})\to H_{S^{\perp}}^s(\T^{\nu+1})$, with bounded inverse $\Phi_{\infty}^{-1}(\omega)$, that conjugates $\mathcal{L}_6$ in \eqref{L6} to constant coefficients, namely 
\begin{equation}
\begin{aligned}
&\mathcal{L}_{\infty}(\omega):=\Phi_{\infty}^{-1} (\omega) \circ \mathcal{L}_5 \circ \Phi_{\infty}(\omega)=\omega\cdot \partial_{\varphi}+\mathcal{D}_{\infty}(\omega),\\
& \mathcal{D}_{\infty}(\omega):=\mbox{diag}_{j\in S^c} \{ \mu_j^{\infty}(\omega) \}.
\end{aligned}
\end{equation}
The transformations $\Phi_{\infty}, \Phi_{\infty}^{-1}$ are close to the identity in matrix decay norm, with
\begin{equation}
\lvert \Phi^{\pm 1}_{\infty}-\mathrm{I} \rvert^{Lip(\gamma)}_{s, \Omega_{\infty}^{2 \gamma}}\le_s \varepsilon^7 \gamma^{-3}+\varepsilon \gamma^{-1} \lVert \mathfrak{I}_{\delta} \rVert_{s+\sigma}^{Lip(\gamma)}.
\end{equation}
Moreover $\Phi_{\infty}, \Phi_{\infty}^{-1}$ are symplectic, and $\mathcal{L}_{\infty}$ is a Hamiltonian operator.
\end{itemize}
\end{teor}
\begin{remark}
Theorem $4. 2$ in \cite{Airy} also provides the Lipschitz dependence of the (approximate) eigenvalues $\mu_j^n$ with respect to the unknown $i_{0}(\varphi)$, which is used for the measure estimate in Lemma \ref{InclusionideiBadSets}.
\end{remark}
Observe that all the parameters $\omega\in \Omega_{\infty}^{2 \gamma}$ satisfy also the first Melnikov condition, namely
\begin{equation}
\lvert \mathrm{i} \omega\cdot l+\mu_j^{\infty}(\omega) \rvert\geq 2 \gamma \lvert j \rvert^3 \langle l \rangle^{-\tau}, \quad \forall l\in \mathbb{Z}^{\nu}, \,\,j\in S^c,
\end{equation}
because, by definition, $\mu_0^{\infty}=0$, and the diagonal operator $\mathcal{L}_{\infty}$ is invertible.\\
In the following theorem we verify the inversion assumption \eqref{InversionAssumption} for $\mathcal{L}_{\omega}$.
\begin{teor}\label{InversionLomega}
Assume the hypotesis of Theorem \ref{Reducibility} and \eqref{PiccolezzaperKamred}. Then there exists $\sigma_1:=\sigma_1(\tau, \nu)>0$ such that, for all $\omega\in \Omega_{\infty}^{2 \gamma}(i_{\delta})$ (see \eqref{OmegoneInfinito}), for any function $g\in H^{s+\sigma_1}_{S^{\perp}}(\T^{\nu+1})$ the equation $\mathcal{L}_{\omega} h=g$ has a solution $h=\mathcal{L}_{\omega}^{-1} g\in H_{S^{\perp}}^s(\T^{\nu+1})$, satisfying
\begin{equation}\label{TameEstimateLomega}
\begin{aligned}
\lVert \mathcal{L}_{\omega}^{-1} g \rVert_s^{Lip(\gamma)} &\le_s \gamma^{-1} (\lVert g \rVert_{s+\sigma_1}^{Lip(\gamma)}+\varepsilon\gamma^{-1}\lVert \mathfrak{I}_{\delta} \rVert_{s+\sigma_1}^{Lip(\gamma)}\lVert g \rVert_{s_0}^{Lip(\gamma)})\\
&\le_s \gamma^{-1} (\lVert g \rVert_{s+\sigma_1}^{Lip(\gamma)}+\varepsilon\gamma^{-1}\{ \lVert \mathfrak{I}_0 \rVert_{s+\sigma_1+\sigma}^{Lip(\gamma)}+\gamma^{-1} \lVert \mathfrak{I}_0 \rVert_{s_0+\sigma}^{Lip(\gamma)}\lVert Z \rVert_{s+\sigma_1+\sigma}^{Lip(\gamma)} \}\lVert g \rVert_{s_0}^{Lip(\gamma)}).
\end{aligned}
\end{equation}
\end{teor}

\section{The Nash-Moser nonlinear iteration}
In this section we prove Theorem \ref{IlTeorema}. It will be a consequence of the Nash-Moser theorem \ref{NashMoser}.\\
Consider the finite-dimensional subspaces
\[
E_n:=\{ \mathfrak{I}(\varphi)=(\Theta, y, z)(\varphi) :  \Theta=\Pi_n\Theta, y=\Pi_n y, z=\Pi_n z \}
\]
where $N_n:=N_0^{\chi^n}$ are introduced in \eqref{Nn}, and $\Pi_n$ are the projectors (which, with a small abuse of notation, we denote with the same symbol)
\begin{equation}\label{Tagli}
\begin{aligned}
&\Pi_n \Theta(\varphi):=\sum_{\lvert l \rvert<N_n} \Theta_l\, e^{\mathrm{i} l\cdot \varphi}, \,\,\Pi_n y(\varphi):=\sum_{\lvert l \rvert<N_n} y_l \,e^{\mathrm{i} l\cdot \varphi},\,\,\mbox{where}\,\,\Theta(\varphi)=\sum_{l\in\mathbb{Z}^{\nu}} \Theta_l\,e^{\mathrm{i} l\cdot \varphi}, \,\,y(\varphi)=\sum_{l\in \mathbb{Z}^{\nu}} y_l \,e^{\mathrm{i} l\cdot\varphi},\\[2mm]
&\Pi_n z(\varphi, x):=\sum_{\lvert (l, j) \rvert<N_n} z_{l j}\,e^{\mathrm{i}(l\cdot \varphi+j x)}, \,\,\mbox{where}\quad z(\varphi, x)=\sum_{l\in\mathbb{Z}^{\nu}, j\in S^c} z_{l j}\,e^{\mathrm{i} (l\cdot \varphi+j x)}.
\end{aligned}
\end{equation}
We define $\Pi_n^{\perp}=\mathrm{I}-\Pi_n$. The classical smoothing properties hold, namely, for all $\alpha, s\geq 0$,
\begin{equation}\label{Smoothing}
\lVert \Pi_n \mathfrak{I} \rVert_{s+\alpha}^{Lip(\gamma)}\le N_n^{\alpha} \lVert \mathfrak{I}_{\delta} \rVert_s^{Lip(\gamma)}, \quad \forall \mathfrak{I}(\omega)\in H^s, \quad \lVert \Pi_n^{\perp} \mathfrak{I} \rVert_s^{Lip(\gamma)}\le N_n^{-\alpha} \lVert \mathfrak{I} \rVert_{s+\alpha}^{Lip(\gamma)}, \quad \forall \mathfrak{I}(\omega)\in H^{s+\alpha}.
\end{equation}
We define the following constants
\begin{equation}\label{parametriNM}
\begin{aligned}
&\mu_1:=3\mu+9, \qquad \qquad  \qquad\alpha:=3\mu_1+1, \qquad \qquad \qquad \alpha_1:=(\alpha-3 \mu)/2,\\
&k:=3(\mu_1+\rho^{-1})+1, \qquad \beta_1:=6 \mu_1+3 \rho^{-1} +3, \qquad 0<\rho<\frac{1-3 a}{C_1 (1+a)}.
\end{aligned}
\end{equation}
where $\mu:=\mu(\tau, \nu)>0$ is the ``loss of regularity'' given by the Theorem \ref{ApproxInverse} and $C_1$ is fixed below. We note that the constants in \eqref{parametriNM} are the same of the ones defined in \cite{KdVAut}, but with a different (larger) $\mu$. 
\begin{teor}{\textbf{(Nash-Moser)}}\label{NashMoser}
Assume that $f\in C^q$ with $q>S:=s_0+\beta_1+\mu+3$. Let $\tau\geq \nu+2$. Then there exist $C_1>\max\{ \mu_1+\alpha, C_0 \}$ (where $C_0:=C_0(\tau, \nu)$ is the one in Theorem \ref{Reducibility}), $\delta_0:=\delta_0(\tau, \nu)>0$ such that, if
\begin{equation}\label{SmallnessConditionNM}
N_0^{C_1} \varepsilon^{b_*+1} \gamma^{-2}<\delta_0, \quad \gamma:=\varepsilon^{2+a}=\varepsilon^{2 b}, \quad N_0:=(\varepsilon\gamma^{-1})^{\rho}, \quad b_*=6- 2 b,
\end{equation}
then, for all $n\geq 0$:
\begin{itemize}
\item[$(\mathcal{P}1)_n$] there exists a function $(\mathfrak{I}_n, \zeta_n)\colon \mathcal{G}_n \subseteq \Omega_{\varepsilon} \to E_{n-1}\times \mathbb{R}^{\nu}, \omega\mapsto (\mathfrak{I}_n(\omega), \zeta_n(\omega)), (\mathfrak{I}_0, \zeta_0):=0, E_{-1}:=\{0\}$, satisfying $\lvert \zeta_n \rvert^{Lip(\gamma)}\le C \lVert \mathcal{F}(U_n) \rVert_{s_0}^{Lip(\gamma)}$,
\begin{equation}\label{Convergenza}
\lVert \mathfrak{I}_n \rVert_{s_0+\mu}^{Lip(\gamma)}\le C_* \varepsilon^{b_*} \gamma^{-1}, \quad \lVert \mathcal{F}(U_n) \rVert_{s_0+\mu+3}^{Lip(\gamma)}\le C_* \varepsilon^{b_*},
\end{equation}
where $U_n:=(i_n, \zeta_n)$ with $i_n(\varphi)=(\varphi, 0, 0)+\mathfrak{I}_n(\varphi)$. The sets $\mathcal{G}_n$ are defined inductively by:
\begin{equation}\label{Gn}
\begin{aligned}
&\mathcal{G}_0:=\{ \omega\in\Omega_{\varepsilon} : \lvert \omega\cdot l \rvert\geq 2 \gamma \langle l \rangle^{-\tau}, \,\,\forall l\in \mathbb{Z}^{\nu}\setminus\{0\} \},\\
&\mathcal{G}_{n+1}:=\left\{ \omega\in \mathcal{G}_n : \lvert \mathrm{i} \omega\cdot l+\mu_j^{\infty}(i_n)-\mu_k^{\infty}(i_n) \rvert\geq \frac{2\,\gamma_n\,\lvert j^3-k^3 \rvert}{\langle l \rangle^{\tau}}, \,\,\forall j, k\in S^c \cup \{0\}, l\in\mathbb{Z}^{\nu} \right\},
\end{aligned}
\end{equation}
where $\gamma_n:=\gamma (1+2^{-n})$ and $\mu_j^{\infty}(\omega):=\mu_j^{\infty}(\omega, i_n(\omega))$ are defined in \eqref{FinalEigenvalues} (and $\mu_0^{\infty}(\omega)=0$).\\
The differences $\hat{\mathfrak{I}}_n:=\mathfrak{I}_n-\mathfrak{I}_{n-1}$ (where we set $\hat{\mathfrak{I}}_0:=0$) is defined on $\mathcal{G}_n$, and satisfy
\begin{equation}\label{FrakHat}
\lVert \hat{\mathfrak{I}}_1 \rVert_{s_0+\mu}^{Lip(\gamma)}\le C_* \varepsilon^{b_*} \gamma^{-1}, \quad \lVert \hat{\mathfrak{I}}_n \rVert_{s_0+\mu}^{Lip(\gamma)}\le C_* \varepsilon^{b_*} \gamma^{-1} N_{n-1}^{-\alpha}, \quad \forall n>1.
\end{equation}
\item[$(\mathcal{P} 2)_n$] $\lVert \mathcal{F}(U_n) \rVert_{s_0}^{Lip(\gamma)}\le C_* \varepsilon^{b_*} N_{n-1}^{-\alpha}$ where we set $N_{-1}:=1$.
\item[$(\mathcal{P}3)_n$]{(High Norms)}. $\lVert \mathfrak{I}_n \rVert_{s_0+\beta_1}^{Lip(\gamma)}\le C_* \varepsilon^{b_*} \gamma^{-1} N_{n-1}^k$ and $\lVert \mathcal{F}(U_n) \rVert_{s_0+\beta_1}^{Lip(\gamma)}\le C_* \varepsilon^{b*} N_{n-1}^k$.
\item[$(\mathcal{P} 4)_n$]{(Measure)}. The measure of the ``Cantor-like'' sets $\mathcal{G}_n$ satisfies
\begin{equation}\label{Misure}
\lvert \Omega_{\varepsilon}\setminus \mathcal{G}_0 \rvert\le C_* \varepsilon^{2 (\nu-1)} \gamma, \quad \lvert \mathcal{G}_n \setminus \mathcal{G}_{n+1} \rvert\le C_* \varepsilon^{2 (\nu-1)} \gamma N_{n-1}^{-1}.
\end{equation}  
\end{itemize}
All the Lip norms are defined on $\mathcal{G}_n$, namely $\lVert \cdot \rVert_s^{Lip(\gamma)}=\lVert \cdot \rVert_{s, \mathcal{G}_n}^{Lip(\gamma)}$.
\end{teor}
\begin{proof}
\begin{itemize}
\item \textit{Proof of $(\mathcal{P}_1)_0, (\mathcal{P}_2)_0, (\mathcal{P}_3)_0$}. Recalling \eqref{NonlinearFunctional}, we have, by the second estimate in \eqref{EstimatesVecField}, 
\[
\lVert \mathcal{F}(U_0)\rVert_s=\lVert \mathcal{F}((\varphi, 0, 0), 0) \rVert_s=\lVert X_P(i_0) \rVert_s\le_s \varepsilon^{6-2 b}.
\]
Hence the smallness conditions in $(\mathcal{P}_1)_0, (\mathcal{P}_2)_0, (\mathcal{P}_3)_0$ hold taking $C_*:=C_*(s_0+\beta_1)$ large enough.
\item \textit{Assume that $(\mathcal{P}_1)_n, (\mathcal{P}_2)_n, (\mathcal{P}_3)_n$ hold for some $n\geq 0$, and prove $(\mathcal{P}_1)_{n+1}, (\mathcal{P}_2)_{n+1}, (\mathcal{P}_3)_{n+1}$}. By \eqref{parametriNM} and \eqref{SmallnessConditionNM}
\[
N_0^{C_1} \varepsilon^{b_*+1} \gamma^{-2}=N_0^{C_1} \varepsilon^{1-3 a}=\varepsilon^{1-3 a-\rho\,C_1 (1+a)}<\delta_0
\]
for $\varepsilon$ small enough. If we take $C_1\geq C_0$ then \eqref{PiccolezzaperKamred} holds. Moreover \eqref{Convergenza} imply \eqref{Assumption}, and so \eqref{IpotesiPiccolezzaIdelta}, and Theorem \ref{InversionLomega} applies. Hence the operator $\mathcal{L}_{\omega}:=\mathcal{L}_{\omega}(\omega, i_n (\omega))$ defined in \eqref{Lomega} is invertible for all $\omega\in \mathcal{G}_{n+1}$ and the last estimate in \eqref{TameEstimateLomega} holds. This means that the assumption \eqref{InversionAssumption} of Theorem  \ref{TeoApproxInv} is verified with $\Omega_{\infty}=\mathcal{G}_{n+1}$. By Theorem \ref{TeoApproxInv} there exists an approximate inverse $\textbf{T}_n(\omega):=\textbf{T}_0(\omega, i_n(\omega))$ of the linearized operator $L_n(\omega):=d_{i, \zeta} \mathcal{F}(\omega, i_n(\omega))$, satisfying \eqref{TameEstimateApproxInv}. By \eqref{SmallnessConditionNM}, \eqref{Convergenza}
\begin{align}
&\lVert \textbf{T}_n g\rVert_s\le_s \gamma^{-1} (\lVert g \rVert_{s+\mu}+\varepsilon\gamma^{-1} \{ \lVert \mathfrak{I}_n \rVert_{s+\mu}+\gamma^{-1} \lVert \mathfrak{I}_n \rVert_{s_0+\mu} \lVert \mathcal{F}(U_n) \rVert_{s+\mu} \} \lVert g \rVert_{s_0+\mu})\label{9.10}\\
&\lVert \textbf{T}_n g \rVert_{s_0}\le_{s_0} \gamma^{-1} \lVert g \rVert_{s_0+\mu}
\end{align}
and, by \eqref{6.41}, using also \eqref{SmallnessConditionNM}, \eqref{Convergenza}, \eqref{Smoothing},
\begin{align}
\lVert (L_n\circ \textbf{T}_n-\mathrm{I}) g \rVert_s \le_s & \gamma^{-1} (\lVert \mathcal{F}(U_n) \rVert_{s_0+\mu}\lVert g \rVert_{s+\mu}+\lVert \mathcal{F}(U_n) \rVert_{s+\mu}\lVert g \rVert_{s_0+\mu}\notag\\
& + \varepsilon\gamma^{-1} \lVert \mathfrak{I}_n \rVert_{s+\mu}\lVert \mathcal{F}(U_n) \rVert_{s_0+\mu}\lVert g \rVert_{s_0+\mu})\\
\lVert (L_n\circ \textbf{T}_n-\mathrm{I}) g \rVert_{s_0} \le_{s_0}& \gamma^{-1} \lVert \mathcal{F}(U_n) \rVert_{s_0+\mu} \lVert g \rVert_{s_0+\mu}\notag\\\notag
\le_{s_0}& \gamma^{-1} (\lVert \Pi_n \mathcal{F}(U_n) \rVert_{s_0+\mu}+\lVert \Pi_n^{\perp} \mathcal{F}(U_n) \rVert_{s_0+\mu})\lVert g \rVert_{s_0+\mu}\\
\le_{s_0}& N_n^{\mu} \gamma^{-1} (\lVert \mathcal{F}(U_n) \rVert_{s_0}+N_n^{-\beta_1}\lVert \mathcal{F}(U_n) \rVert_{s_0+\beta_1}) \lVert g \rVert_{s_0+\mu}.
\end{align}
The index $\beta_1$ in \eqref{parametriNM} is an ultraviolet cut, and it has to be define in order to obtain the convergence of the iteration scheme.\\
Now, for all $\omega\in\mathcal{G}_{n+1}$, we can define, for $n\geq 0$,
\begin{equation}\label{HnDef}
U_{n+1}:=U_n+H_{n+1}, \quad H_{n+1}:=(\hat{\mathfrak{I}}_{n+1}, \hat{\zeta}_{n+1}):=-\tilde{\Pi}_n \textbf{T}_n \Pi_n \mathcal{F}(U_n)\in E_n\times\mathbb{R}^{\nu},
\end{equation}
where $\tilde{\Pi}_n(\mathfrak{I}, \zeta):=(\Pi_n \mathfrak{I}, \zeta)$ with $\Pi_n$ defined in \eqref{Tagli}. Since $L_n:=d_{i, \zeta} \mathcal{F}(i_n)$, we write 
\[
\mathcal{F}(U_{n+1})=\mathcal{F}(U_n)+L_n H_{n+1}+Q_n,
\]
where
\begin{equation}
Q_n:=Q(U_n, H_{n+1}), \quad Q(U_n, H):=\mathcal{F}(U_n+H)-\mathcal{F}(U_n)-L_n H, \quad H\in E_n\times\mathbb{R}^{\nu}.
\end{equation}
Then, by the definition of $H_{n+1}$ in \eqref{HnDef}, using $[L_n, \Pi_n]$ and writing $\tilde{\Pi}_n^{\perp}(\mathfrak{I}, \zeta):=(\Pi_n^{\perp} \mathfrak{I}, 0)$ we have
\begin{equation}
\begin{aligned}
\mathcal{F}(U_{n+1})&=\mathcal{F}(U_n)-L_n \tilde{\Pi}_n \textbf{T}_n \Pi_n \mathcal{F}(U_n)+Q_n=\mathcal{F}(U_n)-L_n \textbf{T}_n \Pi_n \mathcal{F}(U_n)+L_n \tilde{\Pi}_n^{\perp} \textbf{T}_n \Pi_n \mathcal{F}(U_n)+Q_n\\
&=\mathcal{F}(U_n)-\Pi_n L_n \textbf{T}_n \Pi_n \mathcal{F}(U_n)+(L_n \tilde{\Pi}_n^{\perp}-\Pi_n^{\perp}L_n)\textbf{T}_n \Pi_n \mathcal{F}(U_n)+Q_n\\
&=\Pi_n^{\perp} \mathcal{F}(U_n)+R_n+Q_n+Q'_n
\end{aligned}
\end{equation}
where
\begin{equation}
R_n:=(L_n \tilde{\Pi}_n^{\perp}-\Pi_n^{\perp}L_n)\textbf{T}_n \Pi_n \mathcal{F}(U_n), \quad Q'_n:=-\Pi_n (L_n \textbf{T}_n-\mathrm{I})\Pi_n \mathcal{F}(U_n).
\end{equation}
\begin{lem}{(Lemma $9.2$ in \cite{KdVAut})}
Define 
\begin{equation}\label{DefWnBn}
w_n:=\varepsilon \gamma^{-2} \lVert \mathcal{F}(U_n) \rVert_{s_0}, \qquad B_n:=\varepsilon\gamma^{-1} \lVert \mathfrak{I}_n \rVert_{s_0+\beta_1}+\varepsilon\gamma^{-2} \lVert \mathcal{F}(U_n) \rVert_{s_0+\beta_1}.
\end{equation}
Then there exists $K:=K(s_0, \beta_1)>0$ such that, for all $n \geq 0$, setting $\mu_1:=3\mu+9$
\begin{equation}\label{disuguaglianze}
w_{n+1}\le K N_n^{\mu_1+\rho^{-1}-\beta_1} B_n+K N_{n}^{\mu_1} w_n^2, \qquad B_{n+1}\le K N_n^{\mu_1+\rho^{-1}} B_n.
\end{equation}
\end{lem}
\item \textit{Proof of $(\mathcal{P}_3)_{n+1}$}. By \eqref{disuguaglianze} and $(\mathcal{P}_3)_n$
\begin{equation}\label{Bound}
B_{n+1}\le K N_n^{\mu_1+\rho^{-1}} B_n\le 2 C_* K \varepsilon^{b_*+1} \gamma^{-2} N_n^{\mu_1+\rho^{-1}} N_{n-1}^k\le C_* \varepsilon^{b_*+1} \gamma^{-2} N_n^k,
\end{equation}
provided $2 K N_n^{\mu_1+\rho^{-1}-k}N_{n-1}^k\le 1, \forall n\geq 0$. Choosing $k$ as in \eqref{parametriNM} and $N_0$ large enough, i.e. for $\varepsilon$ small enough. By \eqref{DefWnBn} and the bound \eqref{Bound} $(\mathcal{P}_3)_{n+1}$ holds.
\item \textit{Proof of $(\mathcal{P}_2)_{n+1}$}. Using \eqref{DefWnBn}, \eqref{disuguaglianze} and $(\mathcal{P}_2)_n, (\mathcal{P}_3)_n$, we get
\[
w_{n+1}\le K N_n^{\mu_1+\rho^{-1}-\beta_1} B_n+K N_n^{\mu_1} w_n^2\le K N_n^{\mu_1+\rho^{-1}-\beta_1} 2 C_* \varepsilon^{b_*+1} \gamma^{-2} N_{n-1}^k+K N_n^{\mu_1} (C_*\varepsilon^{b_*+1}\gamma^{-2} N_{n-1}^{-\alpha})^2
\]
and $w_{n+1}\le C_* \varepsilon^{b_*+1} \gamma^{-2} N_n^{-\alpha}$ provided that
\begin{equation}\label{9.36}
4 K N_n^{\mu_1+\rho^{-1}-\beta_1+\alpha} N_{n-1}^k\le 1, \quad 2 K C_* \varepsilon^{b_*+1} \gamma^{-2}N_n^{\mu_1+\alpha} N_{n-1}^{-2 \alpha}\le 1, \,\, \forall n\geq 0.
\end{equation}
The inequalities in \eqref{9.36} hold by \eqref{SmallnessConditionNM}, taking $\alpha$ as in \eqref{parametriNM}, $C_1>\mu_1+\alpha$ and $\delta_0$ in \eqref{SmallnessConditionNM} small enough. By \eqref{DefWnBn}, the inequality $w_{n+1}\le C_* \varepsilon^{b_*+1} \gamma^{-2} N_n^{-\alpha}$ implies $(\mathcal{P}_2)_{n+1}$.
\item \textit{Proof of $(\mathcal{P}_1)_{n+1}$}. The bound \eqref{FrakHat} for $\hat{\mathfrak{I}}_1$ follows by \eqref{HnDef}, \eqref{9.10} (for $s=s_0+\mu$) and $\lVert \mathcal{F}(U_0) \rVert_{s_0+2 \mu}=\lVert \mathcal{F}((\varphi, 0, 0), 0) \rVert_{s_0+2\mu}\le_{s_0+2\mu} \varepsilon^{b_*}$. The bound \eqref{FrakHat} for $\hat{\mathfrak{I}}_{n+1}$ follows by \eqref{Tagli}, $(\mathcal{P}_2)_n$ and \eqref{parametriNM}. It remains to prove that \eqref{Convergenza} holds at the step $n+1$. We have
\begin{equation}
\lVert \mathfrak{I}_{n+1} \rVert_{s_0+\mu}\le \sum_{k=1}^{n+1} \lVert \hat{\mathfrak{I}}_k \rVert_{s_0+\mu}\le C_* \varepsilon^{b_*} \gamma^{-1} \sum_{k \geq 1} N_{k-1}^{-\alpha_1}\le C_* \varepsilon^{b_*} \gamma^{-1}
\end{equation}
taking $\alpha_1$ as in \eqref{parametriNM} and $N_0$ large enough, i.e. $\varepsilon$ small enough. Moreover, using \eqref{Tagli}, $(\mathcal{P}_2)_{n+1}, (\mathcal{P}_3)_{n+1}$, \eqref{parametriNM} we get
\begin{align*}
\lVert \mathcal{F}(U_{n+1}) \rVert_{s_0+\mu+3} &\le N_n^{\mu+3} \lVert \mathcal{F}(U_{n+1}) \rVert_{s_0}+N_n^{\mu+3-\beta_1}\lVert \mathcal{F}(U_{n+1}) \rVert_{s_0+\beta_1}\\
&\le C_* \varepsilon^{b_*} N_n^{\mu+3-\alpha} +C_* \varepsilon^{b_*}N_n^{\mu+3-\beta_1+k}\le C_* \varepsilon^{b_*},
\end{align*}
which is the second inequality in \eqref{Convergenza} at the step $n+1$. The bound $\lvert \zeta_{n+1} \rvert^{Lip(\gamma)}\le C \lVert \mathcal{F}(U_{n+1}) \rVert_{s_0}^{Lip(\gamma)}$ is a consequence of Lemma \eqref{Lemma6.1}.
\subsection{Measure estimates}
In this section we prove $(\mathcal{P}_4)_n$ for all $n\geq 0$. 
Fixed $n\in\mathbb{N}$, we have 
\begin{equation}\label{Union}
\mathcal{G}_n \setminus \mathcal{G}_{n+1}=\bigcup_{l\in\mathbb{Z}^{\nu}, j, k\in S^c\cup\{0\}} R_{l j k}(i_n)
\end{equation}
where 
\begin{equation}\label{BadSets}
R_{l j k}(i_n):=\{ \omega\in\mathcal{G}_n : \lvert \mathrm{i} \omega\cdot l+\mu_j^{\infty}(i_n)-\mu_k^{\infty}(i_n) \rvert< 2\,\gamma_n\, \lvert j^3-k^3 \rvert \langle l \rangle^{-\tau} \}.
\end{equation}
Since, by \eqref{0diMelnikov}, $R_{l j k}(i_n)=\varnothing$ for $j=k$, in the sequel we assume that $j\neq k$.
\begin{lem}{(Lemma $9.3$ in \cite{KdVAut})}\label{InclusionideiBadSets}
For $n\geq 1, \lvert l \rvert\le N_{n-1}$, one has the inclusion $R_{l j k}(i_n)\subseteq R_{l j k}(i_{n-1})$.
\end{lem}
By definition, $R_{l j k}(i_n)\subseteq \mathcal{G}_n$ (see \eqref{BadSets}). By Lemma \ref{InclusionideiBadSets}, for $n\geq 1$ and $\lvert l \rvert\le N_{n-1}$ we also have $R_{l j k}(i_n)\subseteq R_{l j k}(i_{n-1})$. On the other hand, $R_{l j k}(i_n)\cap \mathcal{G}_n=\varnothing$ (see \eqref{Gn}). As a consequence, $R_{l j k}(i_n)=\varnothing$ for all $\lvert l \rvert\le N_{n-1}$, and
\begin{equation}\label{differenzeInsiemiMisura}
\mathcal{G}_n \setminus \mathcal{G}_{n+1}\subseteq \bigcup_{\substack{j, k\in S^c\cup\{0\}\\ \lvert l \rvert> N_{n-1}}} R_{l j k}(i_n) \quad \forall n\geq 1.
\end{equation}
\begin{lem}\label{Brexit}
Let $n \geq 0$. If $R_{l j k}(i_n)\neq \varnothing$, then $\lvert l \rvert\geq C_1 \lvert j^3-k^3 \rvert\geq \frac{C_1}{2}\,(j^2+k^2)$ for some constant $C_1>0$ (independent of $l, j, k, n, i_n, \omega$).
\end{lem}

By Lemma \ref{Brexit} it is sufficient to study the measure of the resonant sets $R_{l j k}(i_n)$ defined in \eqref{BadSets} for $(l, j, k)\neq (0, j, j)$. In particular we will prove the following Lemma.
\begin{lem}\label{LemmaMisura}
For all $n\geq 0$ and for a generic choice of the tangential sites, the measure $\lvert R_{l j k}(i_n) \rvert\le C \varepsilon^{2(\nu-1)} \gamma \langle l \rangle^{-\tau}$.
\end{lem}
By \eqref{BadSets}, we have to bound the measure of the sublevels of the function $\omega\mapsto\phi(\omega)$ defined by
\begin{equation}\label{phi(omega)}
\begin{aligned}
\phi(\omega):&=\mathrm{i} \omega\cdot l+\mu_j^{\infty}(\omega)-\mu_k^{\infty}(\omega)=\mathrm{i} \omega\cdot l-\mathrm{i} m_3(\omega) (j^3-k^3)+\mathrm{i} m_1 (j-k)+(r_j^{\infty}-r_k^{\infty})(\omega)
\end{aligned}
\end{equation}
Note that $\phi$ also depends on $l, j, k, i_n$. We recall that
\begin{equation}
m_3=1+\varepsilon^2 d(\xi)+\mathtt{r}_{m_3}(\omega), \qquad m_1=\varepsilon^2 c(\xi)+\mathtt{r}_{m_1}(\omega)
\end{equation}
where 
\begin{equation}\label{rm3}
\lvert\mathtt{r}_{m_3}\rvert^{Lip(\gamma)}\le C \varepsilon^3 \qquad \lvert \mathtt{r}_{m_1}\rvert^{Lip(\gamma)}\le C\varepsilon^{3-2 a}
\end{equation}
and $d(\xi)$, $c(\xi)$ are defined in \eqref{OriginalDxi} and \eqref{cXi} respectively.


It will be useful to consider $\phi(\omega)$ in \eqref{phi(omega)} as a small perturbation of an affine function in $\omega$. We write it as
\begin{equation}
\phi(\omega):=a_{j k}+b_{l j k}\cdot \omega+q_{j k}(\omega), \qquad l\in\mathbb{Z}^{\nu},\, j, k\in S^c,
\end{equation}
where, by \eqref{Frequency-AmplitudeMAP}, \eqref{OriginalDxi}, \eqref{cXi},
\begin{align}
a_{j k}:=&-\mathrm{i} \{ (j^3-k^3) [1-d(\mathbb{M}^{-1}\overline{\omega})]+ (j-k) c(\mathbb{M}^{-1}\overline{\omega})\},\label{Ajk}\\
b_{l j k}:=&\mathrm{i} \{ l-(j^3-k^3)[(24 c_4-48 c_1^2) \mathbb{M}^{-T} v_3+(4 c_6-\frac{16}{3} c_2^2) \mathbb{M}^{-T} v_1], \label{Bljk}\\\notag
&+ (j-k) [(-4 c_6+\frac{16}{3} c_2^2) \mathbb{M}^{-1} v_3-(24 c_7-16 c_2 c_3) \mathbb{M}^{-1} v_1]\}\\
q_{j k}(\omega):=&-\mathrm{i}\,\mathtt{r}_{m_3}(\omega)\,(j^3-k^3)+\mathrm{i}\,\mathtt{r}_{m_1}(\omega)\,(j-k)+r_j^{\infty}(\omega)-r_k^{\infty}(\omega)\label{Qjk}
\end{align}
and by \eqref{estimateM1}, \eqref{rm3}, \eqref{Qjk},
\begin{equation}\label{LipQljk}
\begin{aligned}
\lvert q_{j k}(\omega) \rvert^{sup}&\le \varepsilon^{3}\lvert j^3-k^3\rvert+\varepsilon^{3-2 a}\lvert j-k\rvert+\varepsilon^{3-2a},\\[2mm]
\lvert q_{j k}(\omega) \rvert^{lip}&\le \lvert \mathtt{r}_{m_3}(\omega) \rvert^{lip}\lvert j^3-k^3\rvert+\lvert \mathtt{r}_{m_1}(\omega)\rvert^{lip}\lvert j-k \rvert+\lvert r_j^{\infty}-r_k^{\infty}\rvert^{lip}\\
&\le \varepsilon^{3}\gamma^{-1}\lvert j^3-k^3\rvert+\varepsilon^{3-2 a}\gamma^{-1}\lvert j-k\rvert+\varepsilon^{1- 3a}.
\end{aligned}
\end{equation}
\begin{remark}\label{Degeneratecase}
The idea of the proof of Lemma \ref{LemmaMisura} is that \textit{generically} (see Definition \ref{Generic}) $a_{j k}$ has to be sufficiently far from zero or the modulus of the ``derivative'' $b_{l j k}$ has to be big enough. 
\end{remark}
We shall use the following \textit{non-degeneracy} assumptions 
\begin{align}
&(\mathtt{H} 1) \qquad d(\xi)-1\neq 0\quad\mbox{at} \quad \xi=\mathbb{M}^{-1}\overline{\omega},\label{H1}\\[2mm] 
&(\mathtt{H} 2)_{j, k} \quad \mbox{Fixed} \,\, j, k\in S^c,\, j\neq k,\quad\det(\mathbb{M}+B(j, k))\neq 0,\label{H2}
\end{align}
where
\begin{equation}\label{Bjk}
\begin{aligned}
B(j, k):=&-(24 c_4-48 c_1^2+\frac{12 c_6-16 c_2^2}{3(j^2+k^2+j k)}) D_S^3 U D_S^3\\[2mm]
&+(\frac{16 c_2^2}{3}-4 c_6+\frac{(16 c_2 c_3-24 c_7)}{j^2+k^2+j k}) D_S U D_S^3.
\end{aligned}
\end{equation}
In the next lemmata we prove that if the coefficients $c_1, \dots, c_7$ are non-resonant and conditions $(\mathtt{C}1)$-$(\mathtt{C}2)$ hold, then there exist a generic choice of the tangential sites for which Lemma \ref{FuoriPalla} and Lemma \ref{LemmaKuksinPoeschel} hold true.

\begin{lem}\label{Lemmata1}
Fix $\nu\in\mathbb{N}$. If the coefficients $c_1, \dots, c_7$ are non-resonant and
\begin{equation}\label{Luomochesapevatroppo}
\left(7-16\nu \right) c_2^2\neq 6\,(1- 2\nu) c_6
\end{equation}
then the polynomial $P(\overline{\jmath}_1, \dots, \overline{\jmath}_{\nu}):=d(\mathbb{M}^{-1}\overline{\omega})-1$ is not identically zero. As a consequence, the assumption $(\mathtt{H} 1)$ is verified for a generic choice of the tangential sites.
\end{lem}
\begin{proof}
Suppose that $d(\mathbb{M}^{-1} \overline{\omega})=1$, namely
\begin{equation}\label{P}
P(\overline{\jmath}_1, \dots, \overline{\jmath}_{\nu}):=\{(24 c_4-48\,c_1^2) v_3+(4 c_6-\frac{16}{3} c_2^2) v_1\}\cdot \mathbb{M}^{-1}\overline{\omega}-1=0.
\end{equation}
We evaluate the polynomial P at the point $(\overline{\jmath}_1, \dots, \overline{\jmath}_{\nu})=\lambda (1, \dots, 1)=\lambda \vec{1}$, for some $\lambda$ to be determined, and we claim that this is not a zero. This implies that the polinomial $P$ in \eqref{P} cannot be identically zero. We have
\begin{align*}
P(\lambda \vec{1})= \{ \lambda^5\,(24 c_4-48\,c_1^2) + \lambda^3\,(4 c_6-\frac{16}{3} c_2^2) \} \,(\vec{1}\cdot \mathbb{M}(\lambda \vec{1})^{-1}\vec{1})-1
\end{align*}
and
$
\mathbb{M}(\lambda \vec{1})=a(\lambda) \mathrm{I}+b(\lambda) \,U,
$
where
\begin{align}
&a(\lambda):=(24 c_1^2-12 c_4) \lambda^6+(\frac{14}{3} c_2^2-4 c_6) \lambda^4+(12 c_2 c_3 -12 c_7)\,\lambda^2-6\,c_3^2,\label{Alambda}\\ 
&b(\lambda):= (-48 c_1^2+24 c_4) \lambda^6+(-\frac{32}{3} c_2^2+8 c_6) \lambda^4+(-16 c_2 c_3 +24 c_7)\,\lambda^2.\label{Blambda}
\end{align}
We note that $a(\lambda)\neq 0$, because the coefficients are non-resonant. Moreover, by assumption \eqref{Luomochesapevatroppo} $a(\lambda)+ \nu b(\lambda)\neq 0$ and we have
\[
(\mathbb{M}(\lambda \vec{1}))^{-1}=\frac{\mathrm{I}}{a(\lambda)}-\frac{b(\lambda)}{a(\lambda)\,(a(\lambda)+b(\lambda)\,\nu)}\, U
\]
and, by $\vec{1}\cdot \vec{1}=\nu,\,\, \vec{1}\cdot U \vec{1}=\nu^2$, we get
\begin{equation}
\vec{1}\cdot \mathbb{M}(\lambda \vec{1})^{-1} \vec{1}=\frac{\nu}{a(\lambda)+b(\lambda) \nu}.
\end{equation}
Then $P(\lambda \vec{1})=0$ is equivalent to $p(\lambda)=0$, where
\begin{align*}
p(\lambda):&=\lambda^6 \{ 24 c_1^2-12 c_4  \}+\lambda^4 \{ (\frac{14}{3}-\frac{16\,\nu}{3}) c_2^2-4 (1-\nu) c_6 \}\\
&+\lambda^2 \{ (12-16 \nu) c_2 c_3-12 (1- 2 \nu) c_7 \}-6 c_3^2.
\end{align*}
Suppose that $c_3\neq 0$, then $p(\lambda)$ is not trivial. If $c_3= 0$ and $2 c_1^2\neq c_4$ then we conclude the same, because the monomial of degree six is not naught. If $c_3=0, 2 c_1^2=c_4$ then the monomial of minimum degree, namely three, it is not zero if $c_7\neq 0$, indeed $\nu\in\mathbb{N}$. Suppose now that $c_3=c_7=0, 2 c_1^2=c_4$. Eventually, by assumption \eqref{Luomochesapevatroppo} the monomial of maximum degree, namely four, is not naught and we conclude.
\end{proof}

\begin{lem}\label{Lemmata2}
Fix $\nu\in\mathbb{N}$. If $c_1, \dots, c_7$ are non-resonant and 
\begin{equation}\label{GliUccelli}
\nu\,\,\frac{ 3 c_6-4 c_2^2}{9 c_4-18 c_1^2}\notin \{ j^2+k^2+j k \,:\, j, k\in \mathbb{Z}\setminus\{0\}, \,j\neq k\},
\end{equation}
then the polynomials $P_{j k}(\overline{\jmath}_1, \dots, \overline{\jmath}_{\nu}):=\det (\mathbb{M}+B(j, k))$ are not identically zero, for all $j, k\in S^c$, $j\neq k$.
\end{lem}
\begin{proof}
By \eqref{Bjk} we have
\begin{align*}
\mathbb{M}+B(j, k)&=(24 c_1^2-12 c_4) D_S^6+(\frac{14}{3} c_2^2-4 c_6) D_S^4+(4 c_6-\frac{16}{3} c_2^2)D_S^3 U D_S (\mathrm{I}-\frac{1}{j^2+k^2+j k} D_S^2)\\
&+12 (c_2 c_3-c_7) D_S^2-6 c_3^2 \mathrm{I} +(16 c_2 c_3-24 c_7) D_S U D_S (\mathrm{I}-\frac{1}{j^2+k^2+j k} D_S^2).
\end{align*}
If $c_3\neq 0$ then the lowest order monomial of $\det(\mathbb{M}+B(j, k))$ is not zero and the same holds if $c_3=0$ and $c_7\neq 0$. If $c_3=c_7=0$ then the monomial of maximal degree is
\[
D_S^3 \left( (24 c_1^2-12 c_4)\mathrm{I}+\frac{12 c_6-16 c_2^2}{3(j^2+k^2+j k)} U  \right) D_S^3
\]
and this is invertible if \eqref{GliUccelli} holds.
\end{proof}
\begin{remark}\label{ancheLem2gen}
By Lemma \ref{Lemmata2}, if $(\mathtt{C}2)$ holds, then the assumptions $(\mathtt{H}2)_{j, k}$ are satisfied by a generic choice of the tangential sites when $j, k$ vary in a finite set of integers. 
\end{remark}
The rest of the section is devoted to the proof of Lemma \ref{LemmaMisura}.

\begin{lem}\label{FuoriPalla}
Assume $(\mathtt{H1})$. Then, for a generic choice of the tangential sites, there exists $C_0>0$ such that for all $j\neq k$, $j, k\in S^c$, with $j^2+k^2>C_0$ and $l\in\mathbb{Z}^{\nu}$, we have $\lvert R_{l j k} \rvert\le C \varepsilon^{2 (\nu-1)} \gamma \langle l \rangle^{-\tau}$.
\end{lem}
\begin{proof}
If $j^2+k^2>C_0$ for some constant $C_0$, then $\lvert j-k\rvert/\lvert j^3-k^3 \rvert\le 2 C_0^{-1}$ and
\[
\lvert a_{j k} \rvert\geq \lvert j^3-k^3 \rvert\,\left\{ \lvert 1-d(\mathbb{M}^{-1}\overline{\omega})\rvert-\frac{2}{C_0}\, \lvert c(\mathbb{M}^{-1}\overline{\omega})\rvert  \right\}.
\] 
If $d(\mathbb{M}^{-1} \overline{\omega})\neq 1$ then, by taking $C_0$ large enough, we get $\lvert a_{j k} \rvert\geq \delta_0 \lvert j^3-k^3 \rvert$, for some $\delta_0>0$. This implies that for $\delta:=\delta_0/2$ we have $\lvert b_{l j k}\cdot \omega\rvert\geq \delta \lvert j^3-k^3 \rvert$. Indeed, by \eqref{BadSets}, \eqref{LipQljk}
\[
\lvert b_{l j k} \cdot \omega \rvert\geq \lvert a_{j k} \rvert -\lvert \phi(\omega) \rvert-\lvert q_{j k}(\omega) \rvert\geq (\delta_0-2 \gamma_n-\lvert q_{j k}(\omega)\rvert^{sup})\lvert j^3-k^3 \rvert\geq \frac{\delta_0}{2} \lvert j^3-k^3 \rvert,
\]
for $\varepsilon$ small enough (recall that $\gamma_n=o(\varepsilon^2)$).\\
If $b:=b_{l j k}$ we have $\lvert b\cdot \omega \rvert\le 2\lvert b \rvert\lvert \overline{\omega} \rvert$, because $\lvert \omega \rvert\le 2 \lvert \overline{\omega} \rvert$. Hence $\lvert b \rvert\geq \delta_1\,\lvert j^3-k^3 \rvert$ where $\delta_1:=\delta/(2 \lvert \overline{\omega} \rvert)$. Split $\omega=s \hat{b}+v$ where $\hat{b}:=b/\lvert b \rvert$ and $v\cdot b=0$. Let $\Psi(s):=\phi(s \hat{b}+v)$. For $\varepsilon$ small enough, by \eqref{LipQljk}, we get
\begin{align*}
\lvert \Psi(s_1)-\Psi(s_2) \rvert&\geq (\lvert b \rvert-\lvert q_{j k} \rvert^{lip}) \lvert s_1-s_2 \rvert\geq \left(\delta_1-\frac{\lvert q_{j k} \rvert^{lip}}{\lvert j^3-k^3 \rvert}\right)\,\lvert j^3-k^3 \rvert\, \lvert s_1-s_2 \rvert\\
&\geq \frac{\delta_1}{2} \lvert j^3-k^3 \rvert \,\lvert s_1-s_2 \rvert.
\end{align*}
As a consequence, the set $\Delta_{l j k}(i_n):=\{ s: s \hat{l}+v\in R_{l j k}(i_n)\}$ has Lebesgue measure
\[
\lvert \Delta_{l j k}(i_n) \rvert\le \frac{2}{\delta_1\,\lvert j^3-k^3 \rvert} \,\frac{4\,\gamma_n\,\lvert j^3-k^3 \rvert}{\langle l \rangle^{\tau}}\le \frac{C\,\gamma}{\langle l \rangle^{\tau}}
\]
for some $C>0$. The Lemma follows by Fubini's theorem.
\end{proof}

\begin{lem}\label{Lgrande}
There exists $M>0$ such that for all $j\neq k$, $j, k\in S^c$, with $j^2+k^2\le C_0$ (see Lemma \ref{FuoriPalla}) and $\lvert l \rvert\geq M$, we have $\lvert R_{l j k} \rvert\le C \varepsilon^{2 (\nu-1)} \gamma \langle l \rangle^{-\tau}$.
\end{lem}
\begin{proof}
For $l\neq 0$, we decompose $\omega=s \hat{l}+v$, where $\hat{l}:=l/\lvert l \rvert, s\in\mathbb{R}$, and $l\cdot v=0$. Let $\psi(s):=\phi(s \hat{l}+v)$. We remark that $c(\xi)$ and $d(\xi)$ are affine functions of the unperturbed actions $\xi$, hence 
$$\varepsilon^2\lvert c(\xi) \rvert^{lip}, \varepsilon^2\lvert d(\xi) \rvert^{lip}\le K$$
for some constant $K$ depending only on the tangential sites and on the real coefficients $c_1,\dots, c_7$. 
Then
\begin{align*}
&\lvert \tilde{m}_3(s_1)-\tilde{m}_3(s_2)\rvert\le K \lvert s_1-s_2 \rvert,\\
&\lvert \tilde{m}_1(s_1)-\tilde{m}_1(s_2) \rvert\le (K+\varepsilon^{3-2 a} \gamma^{-1})\lvert s_1-s_2 \rvert\le 2 K\,\lvert s_1-s_2 \rvert,\\
&\lvert r_j^{\infty}(s_1)-r_j^{\infty}(s_2) \rvert\le \varepsilon^{3- 2 a}\gamma^{-1} \lvert s_1-s_2 \rvert.
\end{align*}
Then, if we take $M$ large enough and $\varepsilon$ small, we have
\begin{align*}
\lvert \psi(s_1)-\psi(s_2) \rvert&\geq \lvert j^3-k^3 \rvert \left(\frac{\lvert l \rvert}{\lvert j^3-k^3 \rvert}-K- \frac{2\,K}{\lvert j^2+k^2+ j k \rvert}-\frac{\varepsilon^{3-2a}\gamma^{-1}}{\lvert j^3-k^3\rvert}\right)\lvert s_1-s_2 \rvert\\
&\geq \frac{\delta}{4}\lvert j^3-k^3 \rvert\,\lvert s_1-s_2\rvert,
\end{align*}
where $\delta$ is a positive constant. Indeed, $C_0$ and $K$ are fixed and it is sufficient to choose $\lvert l \rvert$ such that
\[
\inf_{j\neq k, j^2+k^2\le C_0}\frac{\lvert l \rvert}{\lvert j^3-k^3 \rvert}-K-\frac{2 K}{C_0}\geq \delta>0.
\]
As a consequence, the set $\Delta_{l j k}(i_n):=\{ s: s \hat{l}+v\in R_{l j k}(i_n)\}$ has Lebesgue measure
\[
\lvert \Delta_{l j k}(i_n) \rvert\le \frac{\delta}{\lvert j^3-k^3 \rvert} \,\frac{\gamma_n\,\lvert j^3-k^3 \rvert}{\langle l \rangle^{\tau}}\le \frac{C\,\gamma}{\langle l \rangle^{\tau}}
\]
for some $C>0$. The Lemma follows by Fubini's theorem.
\end{proof}
It remains to investigate $R_{l j k}$ for a finite set of indeces $(l, j, k)$. We need the following Lemma.
\begin{lem}\label{ParteLineareNulla}
Suppose that $l\in\mathbb{Z}^{\nu}$ and $j, k\in S^c$ are such that $\overline{\omega}\cdot l\neq j^3-k^3$ and 
\begin{equation}
\lvert l \rvert\leq M, \qquad \lvert j^2+k^2\rvert\leq C_0
\end{equation}
for some positive constants $M$ and $C_0$. Then $R_{l j k}$ is empty.
\end{lem}
\begin{proof}
We have
\begin{align*}
\lvert \omega\cdot l +m_3 (k^3-j^3) \rvert&=\lvert m_3 (\overline{\omega}\cdot l+k^3-j^3)+(\omega-m_3 \overline{\omega})\cdot l \rvert\geq \lvert m_3 \rvert \lvert \overline{\omega}\cdot l +k^3-j^3 \rvert\\
&-\lvert \omega-m_3 \overline{\omega} \rvert \lvert l \rvert\geq 1-\lvert \omega-\overline{\omega}\rvert M-\lvert m_3-1 \rvert \lvert \overline{\omega}\rvert M\geq 1/2
\end{align*}
for $\varepsilon$ small enough, because $\lvert \omega-\overline{\omega}\rvert, \lvert m_3-1 \rvert\le C \varepsilon^2$. Thus, by \eqref{phi(omega)} we have
\begin{align*}
\lvert \phi(\omega) \rvert&\geq 1/2-\varepsilon^2 \lvert c(\xi) \rvert \lvert j-k \rvert-\lvert m_1-\varepsilon^2 c(\xi) \rvert \lvert j-k \rvert-\lvert r_j-r_k\rvert\\
&\geq 1/2-\varepsilon^2\sup_{\xi\in [1, 2]^{\nu}} (\lvert c(\xi) \rvert) -2 C\,\varepsilon^{3-2 a}\geq 1/4.
\end{align*}
\end{proof}

\begin{lem}\label{LemmaKuksinPoeschel}
If $j^2+k^2\le C_0$, $j, k\in S^c$, $\lvert l \rvert\le M$ (see Lemma \ref{FuoriPalla} and \ref{Lgrande}) and $(\mathtt{H}2)_{j, k}$ hold,
then, for a generic choice of the tangential sites, $\lvert R_{l j k} \rvert\le C \varepsilon^{2 (\nu-1)} \gamma \langle l \rangle^{-\tau}$.
\end{lem}
\begin{proof}
We can write \eqref{phi(omega)} as an affine function respect to the parameter $\xi$ as
\begin{equation}\label{phi(xi)}
\begin{aligned}
\phi(\xi)&=\mathrm{i}\,\overline{\omega}\cdot l-\mathrm{i}(j^3-k^3)+\mathrm{i}\varepsilon^2 \{ \mathbb{M}\xi\cdot l-d(\xi)(j^3-k^3)+c(\xi) (j-k) \}+q_{j k}(\alpha(\xi)),\\
q_{j k}(\alpha(\xi))&=-\mathrm{i}\mathtt{r}_{m_3}(\alpha(\xi))(j^3-k^3)+\mathrm{i} \mathtt{r}_{m_1}(\alpha(\xi)) (j-k)+r_j^{\infty}(\alpha(\xi))-r_k^{\infty}(\alpha(\xi)).
\end{aligned}
\end{equation}
By the relation \eqref{Frequency-AmplitudeMAP}, we can estimate the Lipschitz constant of $\phi(\omega)$ with the derivative respect to $\xi$ of the expression \eqref{phi(xi)}.\\
By Lemma \ref{ParteLineareNulla}, we consider the case $\overline{\omega}\cdot l=j^3-k^3$. 
Thus
\begin{equation}\label{fidiXi}
\begin{aligned}
\phi(\xi)=&\mathrm{i}\,\varepsilon^2 [\mathbb{M}\xi\cdot l-d(\xi)\,\overline{\omega}\cdot l+c(\xi)(j-k)]+ q_{j k}(\alpha(\xi))\\[2mm]
=&\mathrm{i}\,\varepsilon^2 [\mathbb{M}\xi\cdot l-d(\xi)\overline{\omega}\cdot l+\frac{c(\xi)}{j^2+k^2+ j k}\overline{\omega}\cdot l]+ q_{j k}(\alpha(\xi))\\
=&\mathrm{i}\,\varepsilon^2 [\mathbb{M} +B(j, k) ]\,l\cdot \xi+ q_{j k}(\alpha(\xi)).
\end{aligned}
\end{equation}
where $B(j, k)$ is defined in \eqref{Bjk}.
By assumption $(\mathtt{H}2)$, if $l\neq 0$, then
\begin{equation}\label{KuksinPoschel}
\delta_{l j k}:=(\mathbb{M} +B(j, k)) l\neq 0. 
\end{equation}
Hence, by \eqref{LipQljk}, \eqref{fidiXi} and \eqref{KuksinPoschel}, for $\varepsilon$ small enough, there exist a constant $C>0$ such that 
\[
\lvert \phi\rvert^{lip}\geq \delta_{l j k}-\lvert q_{j k}\rvert^{lip}\geq C \lvert j^3-k^3\rvert.
\]
Then we conclude as in Lemma \ref{Lgrande}.
\end{proof}
We have that Lemmata \ref{FuoriPalla}, \ref{Lgrande}, \ref{LemmaKuksinPoeschel} implies Lemma \ref{LemmaMisura}.
By \eqref{Union} and Lemma \ref{LemmaMisura} we get
\[
\lvert \mathcal{G}_0\setminus \mathcal{G}_1 \rvert\le \sum_{l\in\mathbb{Z}^{\nu}, \lvert j \rvert, \lvert k \rvert\le C\lvert l \rvert^{1/2}} \lvert R_{l j k}(i_0) \rvert\le \sum_{l\in\mathbb{Z}^{\nu}} \frac{C\,\varepsilon^{2(\nu-1)}\gamma}{\langle l \rangle^{\tau-1}}\le C' \varepsilon^{2(\nu-1)}\gamma.
\]
For $n\geq 1$, by \eqref{differenzeInsiemiMisura},
\[
\lvert \mathcal{G}_n \setminus \mathcal{G}_{n+1} \rvert\le \sum_{\substack{\lvert l \rvert>N_{n-1},\\ \lvert j \rvert, \lvert k \rvert\le C \lvert l \rvert^{1/2}}} \lvert R_{l j k}(i_n) \rvert\le \sum_{\lvert l \rvert> N_{n-1}} \frac{C\,\varepsilon^{2(\nu-1)}\gamma}{\langle l \rangle^{\tau-1}}\le C' \varepsilon^{2(\nu-1)}\gamma\,N_{n-1}^{-1}
\]
because $\tau\geq \nu+2$. The estimate $\lvert \Omega_{\varepsilon}\setminus \mathcal{G}_0 \rvert\le C\,\varepsilon^{2(\nu-1)}\gamma$ is elementary.
\end{itemize}

\end{proof}

\textbf{Conclusion of the Proof of Theorem \ref{IlTeorema}}. Theorem \ref{NashMoser} implies that the sequence $(\mathfrak{I}_n, \zeta_n)$ is well defined for $\omega\in \mathcal{G}_{\infty}:=\cap_{n\geq 0} \mathcal{G}_n$, and $\mathfrak{I}_n$ is a Cauchy sequence in $\lVert \cdot \rVert_{s_0+\mu, \mathcal{G}_{\infty}}^{Lip(\gamma)}$, see \eqref{FrakHat}, and $\lvert \zeta_n \rvert^{Lip(\gamma)}\to 0$. Therefore $\mathfrak{I}_n$ converges to a limit $\mathfrak{I}_{\infty}$ in norm $\lVert \cdot \rVert_{s_0+\mu, \mathcal{G}_{\infty}}^{Lip(\gamma)}$ and, by $(\mathcal{P} 2)_n$, for all $\omega\in\mathcal{G}_{\infty}, i_{\infty}(\varphi):=(\varphi, 0, 0)+\mathfrak{I}_{\infty}(\varphi)$, is a solution of
\[
\mathcal{F}(i_{\infty}, 0)=0 \qquad \mbox{with} \qquad \lVert \mathfrak{I}_{\infty} \rVert_{s_0+\mu, \mathcal{G}_{\infty}}^{Lip(\gamma)}\le C\,\varepsilon^{6- 2 b} \gamma^{-1}
\]
by \eqref{Convergenza}. Therefore $\varphi \mapsto i_{\infty}(\varphi)$ is an invariant torus for the Hamiltonian vector field $X_{H_{\varepsilon}}$ (recall \eqref{Hepsilon}). By \eqref{Misure},
\[
\lvert \Omega_{\varepsilon}\setminus \mathcal{G}_{\infty} \rvert \le \lvert \Omega_{\varepsilon} \setminus \mathcal{G}_0 \rvert+\sum_{n\geq 0} \lvert \mathcal{G}_n \setminus \mathcal{G}_{n+1} \rvert\le 2\,C_* \varepsilon^{2 (\nu-1)} \gamma+C_* \varepsilon^{2(\nu-1)}\gamma \sum_{n\geq 1} N_{n-1}^{-1}\le C \varepsilon^{2(\nu-1)} \gamma.
\]
The set $\Omega_{\varepsilon}$ in \eqref{OmegaEpsilon} has measure $\lvert \Omega_{\varepsilon} \rvert=O(\varepsilon^{2 \nu})$. Hence $\lvert \Omega_{\varepsilon}\setminus \mathcal{G}_{\infty} \rvert/\lvert \Omega_{\varepsilon} \rvert\to 0$ as $\varepsilon\to 0$ because $\gamma=o(\varepsilon^2)$, and therefore the measure of $\mathcal{C}_{\varepsilon}:=\mathcal{G}_{\infty}$ satisfies \eqref{frazionemisure}.

\subsection{Linear stability}
We show that the solution $i_{\infty}(\omega t)$ is linearly stable, in the sense that the norm of the solutions of the Hamiltonian system associated to \eqref{Hepsilon} linearized on the quasi-periodic solution $i_{\infty}$ does not increase in time.\\
By Section $6$, in particular by the Remark \ref{isotropic}, the system related to \eqref{Hepsilon} is conjugated to the linear system
\begin{equation}\label{sistemalinearizzato}
\begin{cases}
\dot{\psi}=K_{2 0}(\omega t) \eta+K_{11}^T (\omega t) w\\
\dot{\eta}=0\\
\dot{w}=\partial_x K_{0 2} (\omega t) w+\partial_x K_{11} (\omega t) \eta.
\end{cases}
\end{equation}
Thus the actions $\eta(t)$ do not evolve in time and the third equation of \eqref{sistemalinearizzato} reduces to the forced PDE
\begin{equation}\label{linearsystem}
\dot{w}=\partial_x K_{0 2} (\omega t) w+\partial_x K_{11}(\omega t) \eta (0).
\end{equation}
In Section $8$ we proved the reducibility of the linear system \eqref{linearsystem}, ignoring the quasi-periodic function $\partial_x K_{11}(\omega t) \eta (0)$. More precisely, we conjugated it to the diagonal system
\begin{equation}
\dot{v}_j+\mu_j^{\infty}\, v_j=0, \qquad j\in S^c, \quad \mu_j^{\infty}\in\mathrm{i} \mathbb{R},
\end{equation}
where
\begin{equation}
\mu_j^{\infty}:=\mathrm{i} (-m_3 j^3+m_1 j)+r_j^{\infty}
\end{equation}
with $m_3=1+O(\varepsilon^2), m_1=O(\varepsilon^2), r_j^{\infty}=O(\varepsilon^{3-2 a})$.
The eigenvalues $\mu_j^{\infty}$ are the \textit{Floquet exponents} of the linear, quasi-periodically depending on time, system \eqref{linearsystem}.
Then equation \eqref{linearsystem} is reduced to
\begin{equation}\label{forced}
\dot{v}_j+\mu_j^{\infty} v_j=f_j(\omega t), \qquad j\in S^c
\end{equation} 
for some quasi-periodic function $f_j$. The solutions of the scalar non-homogeneous equation \eqref{forced} are
\[
v_j(t)=c_j e^{\mu_j^{\infty} t}+\tilde{v}_j(t), \quad \tilde{v}_j(t):=\sum_{l\in\mathbb{Z}^{\nu}} \frac{f_{j l}}{\mathrm{i} \omega\cdot l+\mu_j^{\infty}} e^{\mathrm{i} \omega\cdot l t}.
\]
We note that $\tilde{v}_j$ is well defined, indeed the first Melnikov conditions hold at a solution. As a consequence, if $v$ is a solution of the system \eqref{linearsystem}, then there exist a constant $C>0$ such that
\[
\lVert v(t) \rVert_{H_x^s}\le C \lVert v(0) \rVert_{H_x^s},\quad \forall t\in\mathbb{R},
\]
hence its Sobolev norm does not increase in time.



\end{document}